\documentclass[12pt,a4paper,oneside,DIV=12]{scrartcl}
\usepackage[utf8]{inputenc}
\usepackage{amssymb,amsmath}
\allowdisplaybreaks[3]
\usepackage{graphicx}
\usepackage{wrapfig}
\usepackage{amsthm}
\usepackage{multicol}
\usepackage{subcaption}
\usepackage{mathtools}
\usepackage{tikz}
\usetikzlibrary{arrows.meta}
\usepackage{hyperref}
\usepackage{siunitx}
\usepackage{microtype}

\newtheorem{Lemma}{Lemma}
\newtheorem{Theorem}{Theorem}
\newtheorem{Proposition}{Proposition}
\newtheorem{Corollary}{Corollary}
\newtheorem{Remark}{Remark}
\newtheorem{Definition}{Definition}

\def\vph{\varphi}

\def\rme{\mathrm{e}}
\def\rmi{\mathrm{i}}

\def\R{\mathbb{R}}
\def\C{\mathbb{C}}
\def\Z{\mathbb{Z}}

\def\calO{\mathcal{O}}

\def\m{\boldsymbol{m}}
\def\M{\boldsymbol{M}}
\def\1{\boldsymbol{e}_1}
\def\3{\boldsymbol{e}_3}
\def\p{\boldsymbol{e}_p}
\def\h{\boldsymbol{h}}
\def\H{\boldsymbol{H}}

\def\cc{c_{\textnormal{cp}}}

\def\smu{\sqrt{-\mu}}
\def\mdrei{\textnormal{m}_3}

\def\b{\boldsymbol{\beta}}

\def\tH{\widetilde{H}}

\def\Im{\mathrm{Im}}
\def\Re{\mathrm{Re}}

\newcommand*{\tran}{^{\mkern-1.2mu\mathsf{T}}}
\title{Inhomogeneous domain walls in spintronic nanowires}
\author{L.\ Siemer\footnote{Universität Bremen, \texttt{lars.siemer@uni-bremen.de}; Corresponding author} \and I.\ Ovsyannikov\footnote{Universität Hamburg, Lobachevsky State University of Nizhny Novgorod} \and J.D.M.\ Rademacher\footnote{Universität Bremen}}
\date{\today}

\begin{document}
\maketitle
\begin{abstract}\label{sec:abstract}
\noindent In case of a spin-polarized current, the magnetization dynamics in nanowires are governed by the classical \emph{Landau-Lifschitz} equation with \emph{Gilbert} damping term, augmented by a typically non-variational \emph{Slonczewski} term. Taking axial symmetry into account, we study the existence of domain wall type coherent structure solutions, with focus on one space dimension and spin-polarization, but our results also apply to vanishing spin-torque term. Using methods from bifurcation theory for arbitrary constant applied fields, we prove the existence of domain walls with non-trivial azimuthal profile, referred to as \emph{inhomogeneous}. We present an apparently new type of domain wall, referred to as \emph{non-flat}, whose approach of the axial magnetization has a certain oscillatory character. Additionally, we present the leading order mechanism for the parameter selection of \emph{flat} and \emph{non-flat} inhomogeneous domain walls for an applied field below a threshold, which depends on anisotropy, damping, and spin-transfer. Moreover, numerical continuation results of all these domain wall solutions are presented.
\end{abstract}

\section{Introduction}\label{sec:introduction}
Magnetic domain walls (DWs) are of great interest both from a theoretical perspective and for applications, especially in the context of innovative magnetic storages~\cite{parkin2008magnetic}. Recent developments in controlled movement of DWs via spin-polarized current pulses in nanomagnetic structures, in particular in nanowires, are thought to lead to a new class of potential non-volatile storage memories, e.g.\ racetrack memory~\cite{parkin2008magnetic, reohr2002memories, lin200945nm, qureshi2009scalable}. These devices make use of the fact that spin-transfer driven effects can change the dynamics in sufficiently small ferromagnetic structures (e.g.\ nanowires), where regions of uniform magnetization, separated by DWs, can appear~\cite{mayergoyz2009nonlinear, hubert2008magnetic, bertotti2008spin}. This motivates further studies of the existence of magnetic domains and their interaction with spin-polarized currents as a building block for the theory in this context. In this paper we take a mathematical perspective and, in a model for nanomagnetic wires, rigorously study the existence of DWs. This led us to discover an apparently new kind of DWs with a certain inhomogeneous and oscillatory structure as explained in more detail below.

\medskip
The description of magnetization dynamics in nanomagnetic structures, governed by the \emph{Landau-Lifschitz-Gilbert} (LLG) equation, is based on works by Berger and Slonczewski assuming a spin-polarized current~\cite{berger1996emission, slonczewski1996current}. In the presence of a constant applied field and a spin-polarized current, the dynamics driven by the joint action of magnetic field and spin torque can be studied by adding a spin-transfer term in the direction of the current (current-perpendicular-to-plane (CPP) configuration). In case of a spatially uniform magnetization, the resulting \emph{Landau-Lifschitz-Gilbert-Slonczewski} (LLGS) equation for unit vector fields $(m_1,m_2,m_3)=\m=\m(x,t) \in \mathbb{S}^2$ (cf.\ Figure~\ref{fig:homdw}) reads
\begin{equation*}\label{LLGS}
\partial_t \m - \alpha \m \times \partial_t \m = -\m \times \h_{\textnormal{eff}} + \m \times \left(\m \times \boldsymbol{J}\right).\tag{LLGS}
\end{equation*}

with effective field  $\h_{\textnormal{eff}}$, Gilbert damping factor $\alpha>0$, and the last term is the so-called polarized spin transfer pseudotorque.

Note that the above equation reduces to the \ref{LLG} equation for $\boldsymbol{J}\equiv 0$, see \S\ref{sec:model} for more details.

\medskip
In this paper we consider the axially symmetric case and set
\begin{equation}\label{eq:effective field}
    \h_{\textnormal{eff}}\coloneqq\partial^2_x \m + \h- \mu \, \mdrei\, \3\,, \quad \boldsymbol{J} \coloneqq \frac{\beta}{1+\cc m_3}\3\,,
\end{equation}
where $\h=h\,\3$ with a uniform and time-independent field strength $h \in \R$, and $\mdrei = \langle\m, \boldsymbol{e}_3\rangle$, $\boldsymbol{e}_3 \in \mathbb{S}^2$. This effective field $\h_{\textnormal{eff}}$ also includes the diffusive exchange term $\partial_x^2\m$, the uniaxial anisotropy and demagnetization field. The specific here with parameter $\mu\in\R$ derives from a first order approximation in the thin film/wire limit for a uniformly magnetized body~\cite{hubert2008magnetic, osborn1945demagnetizing}. In the axially symmetric structure, $\beta\ge 0$ and $\cc\in(-1,1)$ describe the strength of the spin-transfer and the ratio of the polarization~\cite{bertotti2008spin, slonczewski2002currents}. The spin-transfer torque term may provide energy to the system under certain conditions and counterbalance dissipation associated to the Gilbert damping term, which gives rise to coherent non-variational dynamics, see e.g.~\cite{Melcher2017}. 

Notably, for $\beta=0$ one obtains the LLG-equation that does not account for spin transfer effects. Moreover, as shown in~\cite{Melcher2017}, this also holds up to parameter change in case $\cc=0$. Hence, solutions to the LLGS equation for $\beta=0$ or $\cc=0$ are also solutions to the LLG equation, so that all the analytical as well as numerical results for $\cc=0$ in this paper directly transfer to the LLG equation. 

\begin{figure}[h]
    \centering
    \begin{subfigure}[b]{0.25\textwidth}
    \includegraphics[width=\textwidth]{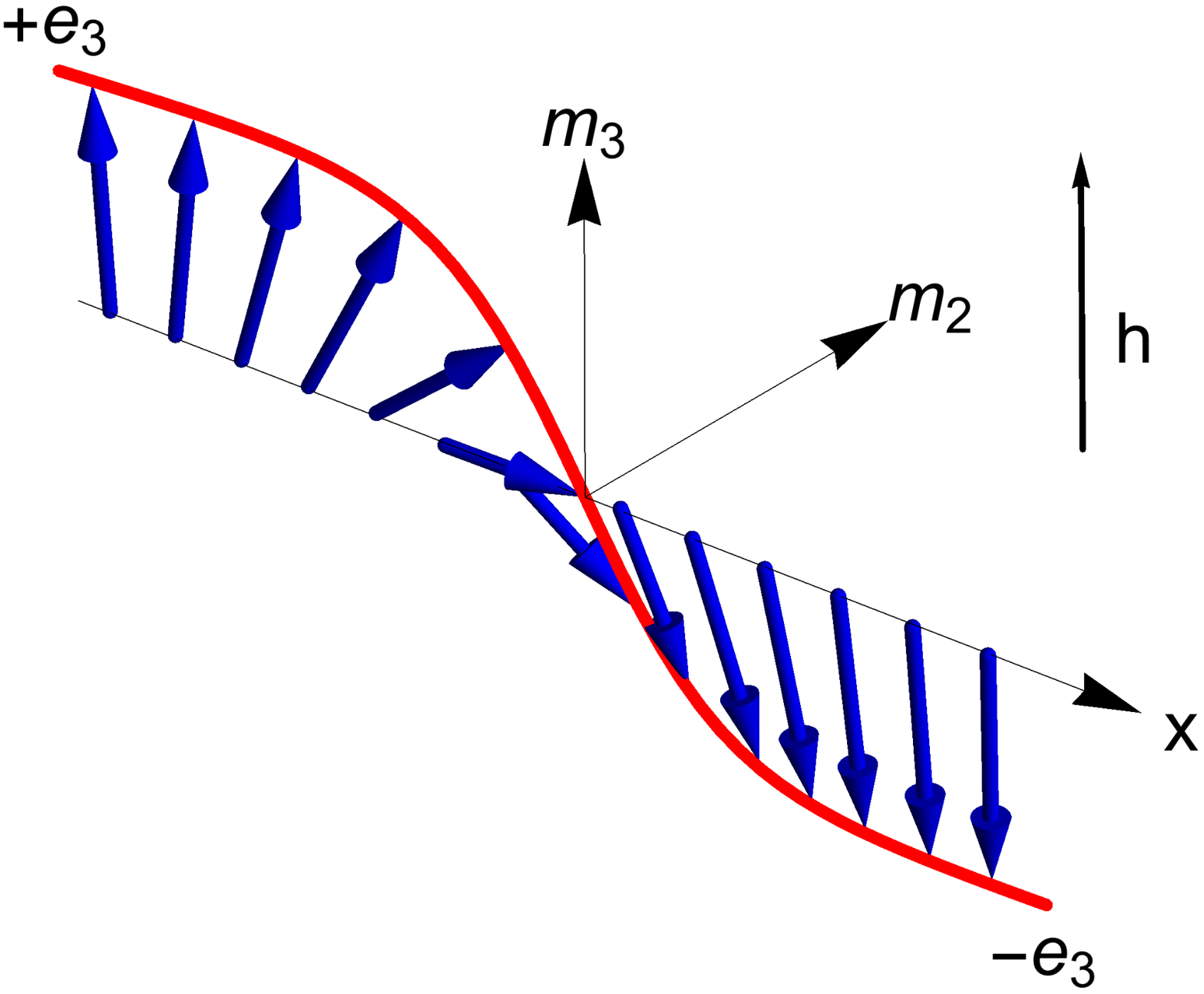}
    \caption{}
    \label{fig:domainwall}
    \end{subfigure}
    \hfill
    \begin{subfigure}[b]{0.25\textwidth}
    \includegraphics[width=\textwidth]{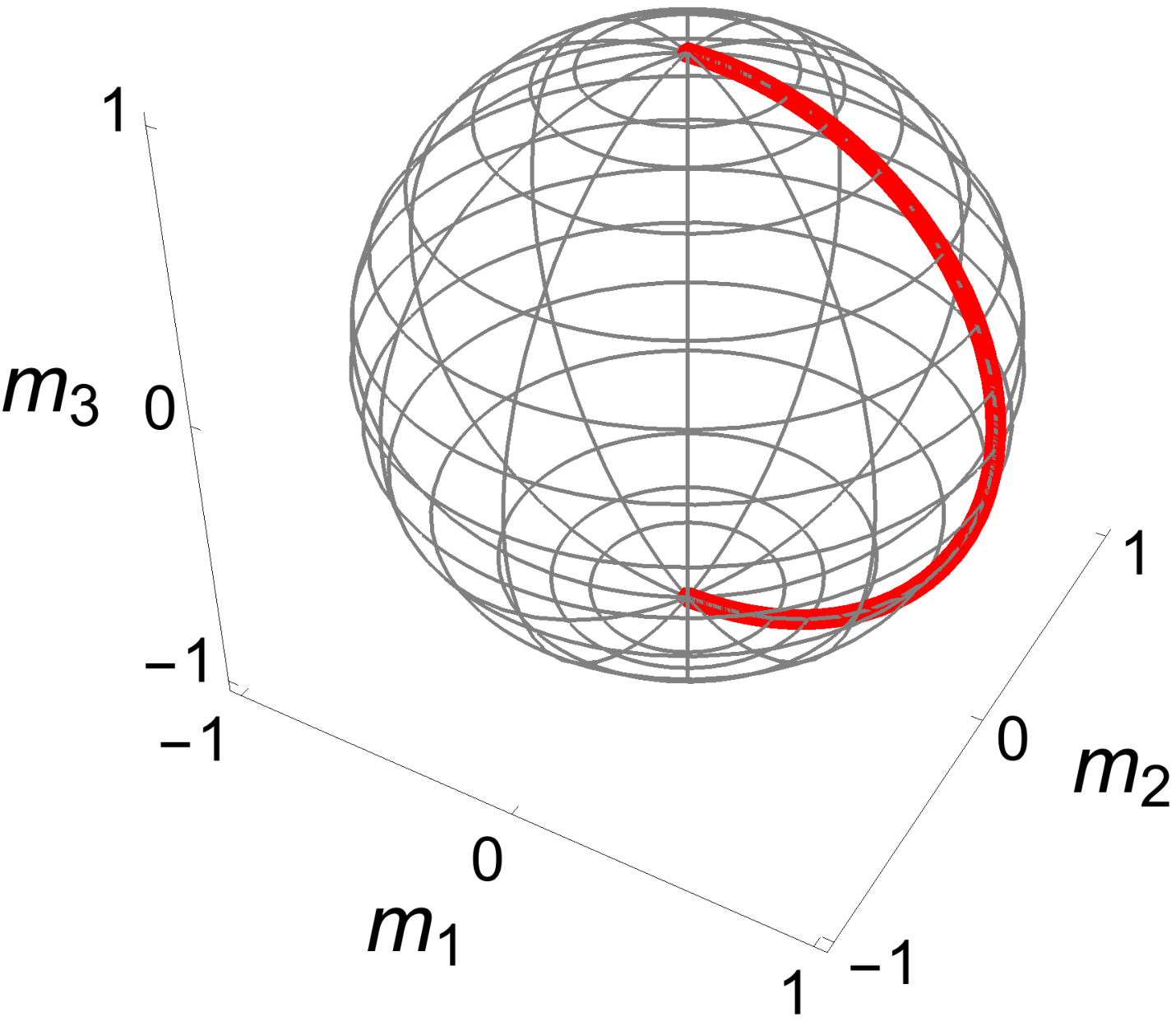}
    \caption{}
    \label{fig:initial arctan}
    \end{subfigure}
    \hfill
    \begin{subfigure}[b]{0.25\textwidth}
    \includegraphics[width=\textwidth]{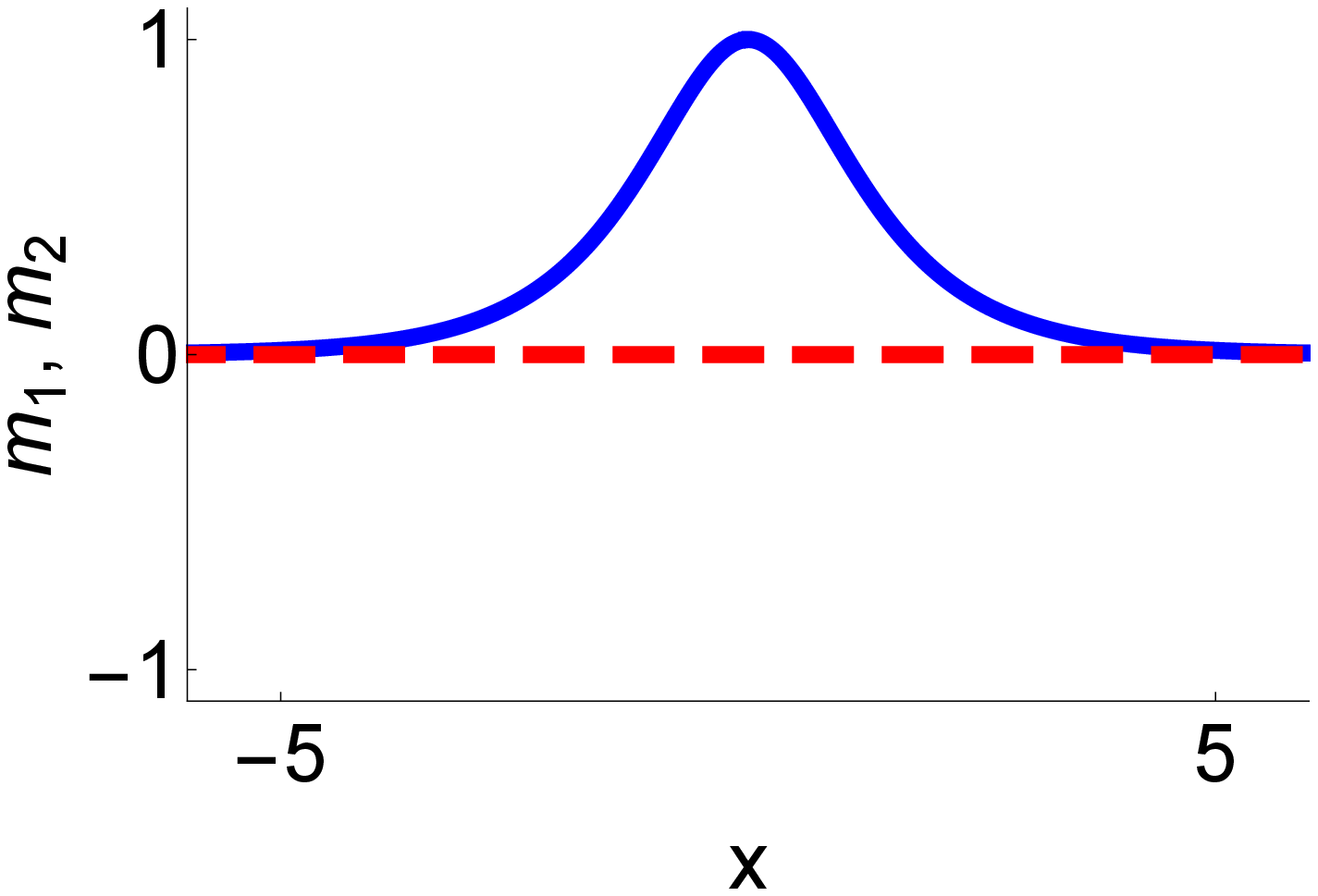}
    \caption{}
    \end{subfigure}
    \caption{\small Homogeneous DW profile ($q\equiv0$) with $\alpha=0.5, \beta=0.1, \mu=-1, h=50, \cc=0$. (a) $(m_2,m_3)$-profile. (b) Projection onto $\mathbb{S}^2$. (c) Zoom-in on $m_1$ (blue solid) and $m_2$ (red dashed).}
    \label{fig:homdw}
\end{figure}

\medskip
A key ingredient for the separation of uniformly magnetized states in space are interfaces between two magnetic domains. The most coherent form of such interfaces in the uniaxial setting are relative equilibria with respect to translation and rotation symmetry of the form $$\m(\xi,t)=\m_0(\xi)\rme^{\rmi \varphi (\xi,t)}, \quad\textnormal{where } \xi=x-st \textnormal{ and } \varphi(\xi,t)\coloneqq \phi(\xi)+\Omega t,$$ with \emph{speed} $s$ and \emph{frequency} $\Omega$. Here the complex exponential acts on $\m_0\in \mathbb{S}^2$ by rotation about the $\boldsymbol{e}_3$-axis, i.e., the azimuth, and in spherical coordinates we can choose $\m_0(\xi)=\left(\sin(\theta(\xi)),0,\cos(\theta(\xi))\right)$ with altitude angle $\theta$.

We refer to such solutions with $\m_0(\xi)\to \pm\3$ as $\xi\to \pm\infty$ or $\xi\to \mp\infty$ as \emph{domain walls}. A first classification of DWs is based on the \emph{local wavenumber} $q:=\phi'$, which determines $\phi$ uniquely due to the axial rotation symmetry and satisfies
\begin{equation}\label{eq:q}q(\xi)=\frac{\langle (m_1',m_2'),(-m_2,m_1)\rangle}{1-m_3^2}(\xi).
\end{equation}
\begin{Definition}\label{def:inhom}
We call a DW with constant $\phi$ \emph{homogeneous (hom)}, i.e., $q\equiv0$, and \emph{inhomogeneous} otherwise.
\end{Definition}
Inhomogeneous DWs have a spatially inhomogeneous varying azimuthal angle, compare Figures~\ref{fig:homdw} and \ref{fig:inhomogeneous DW sphere}.

In the case of uniaxial symmetry and the LLG case $\beta=0$, an explicit family of homogeneous DWs was discovered in~\cite{goussev2010domain} for applied fields with arbitrary strength and time dependence, cf.\ Figure~\ref{fig:homdw}. Furthermore, for constant applied fields and in case of $\cc\neq0$ it was shown in~\cite{Melcher2017} that DWs cannot be homogeneous, and the existence of inhomogeneous DWs was proven, whose spatial profile slowly converges to $\pm\3$ and where $|s|\gg 1$. This latter type of DWs is `weakly localized' and has large `width' in the sense that the inverse slope of $m_3$ at $m_3(0)=0$ tends to infinity as $|s|\to\infty$.

\begin{figure}[h]
    \centering
    \begin{subfigure}[b]{0.25\textwidth}
        \centering
        \caption*{\small small applied field}
        \includegraphics[width=\textwidth]{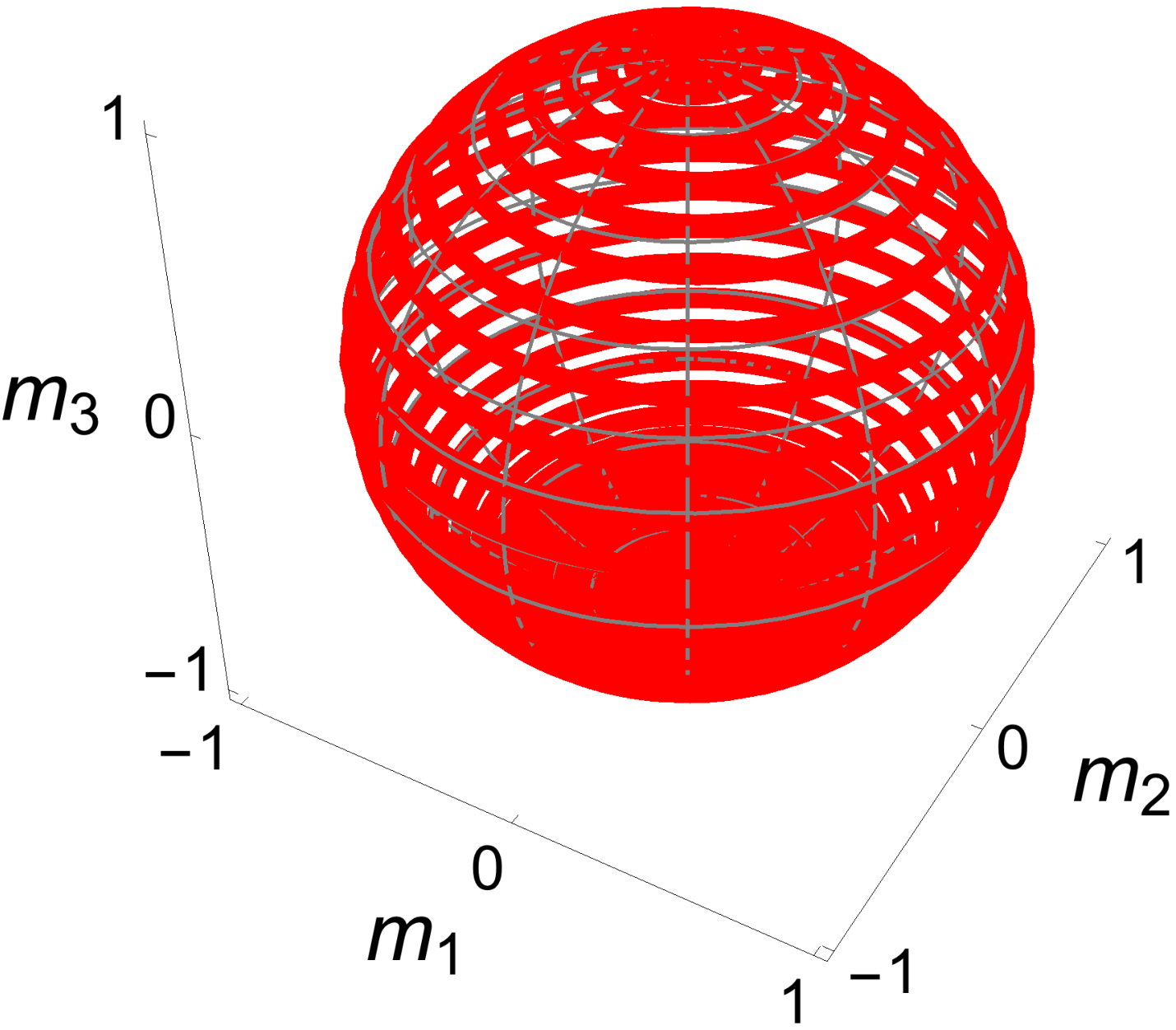}
        \caption{}
        \label{fig:continuation h=0.5 plus}
    \end{subfigure}
    \hspace{0.25\textwidth}
    \begin{subfigure}[b]{0.25\textwidth}  
        \centering
        \caption*{\small large applied field}
        \includegraphics[width=\textwidth]{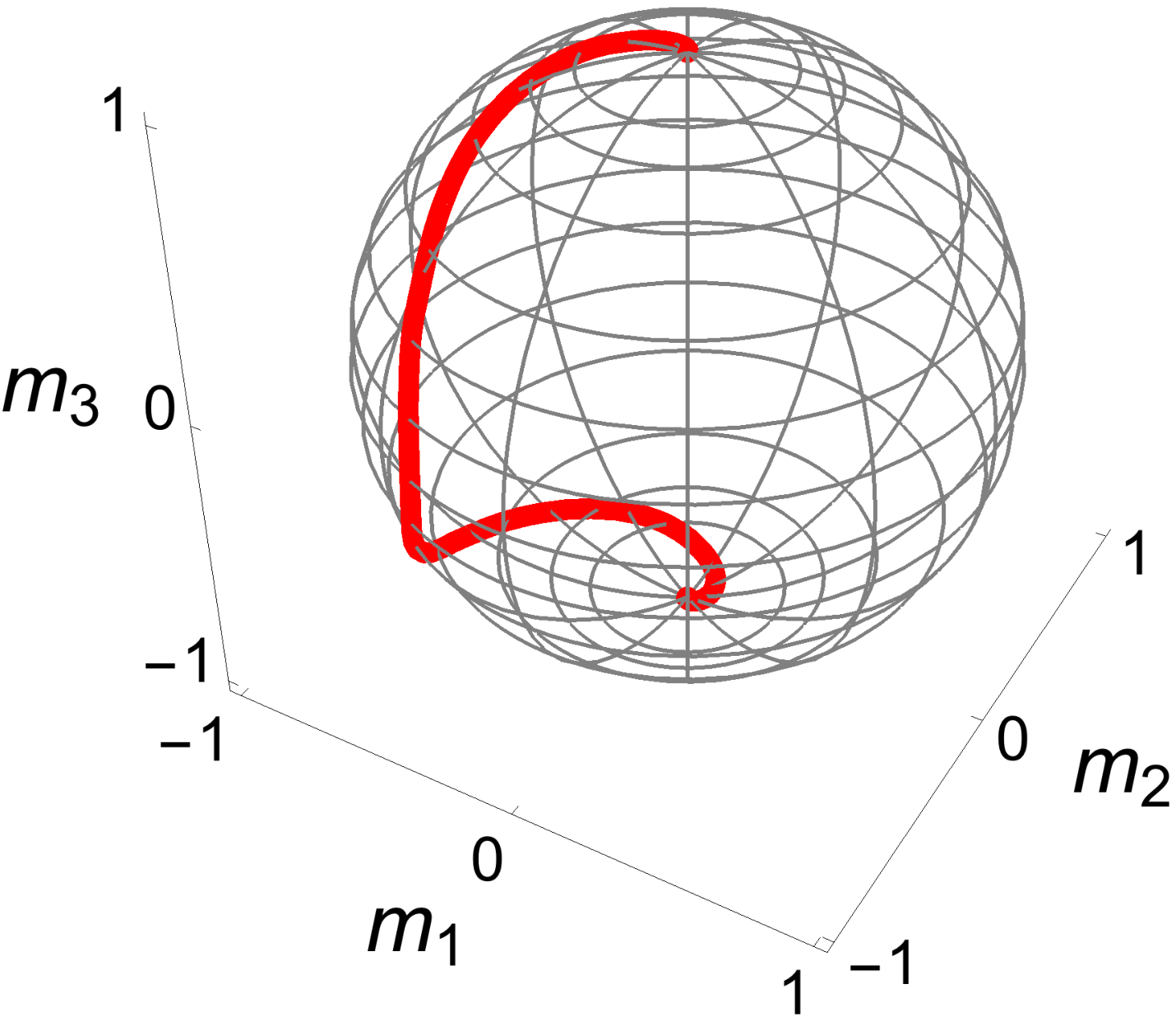}
        \caption{}
        \label{fig:continuation h=50 plus}
    \end{subfigure}
    
    \vspace*{12pt}
    \begin{subfigure}[b]{0.25\textwidth}  
        \centering 
        \includegraphics[width=\textwidth]{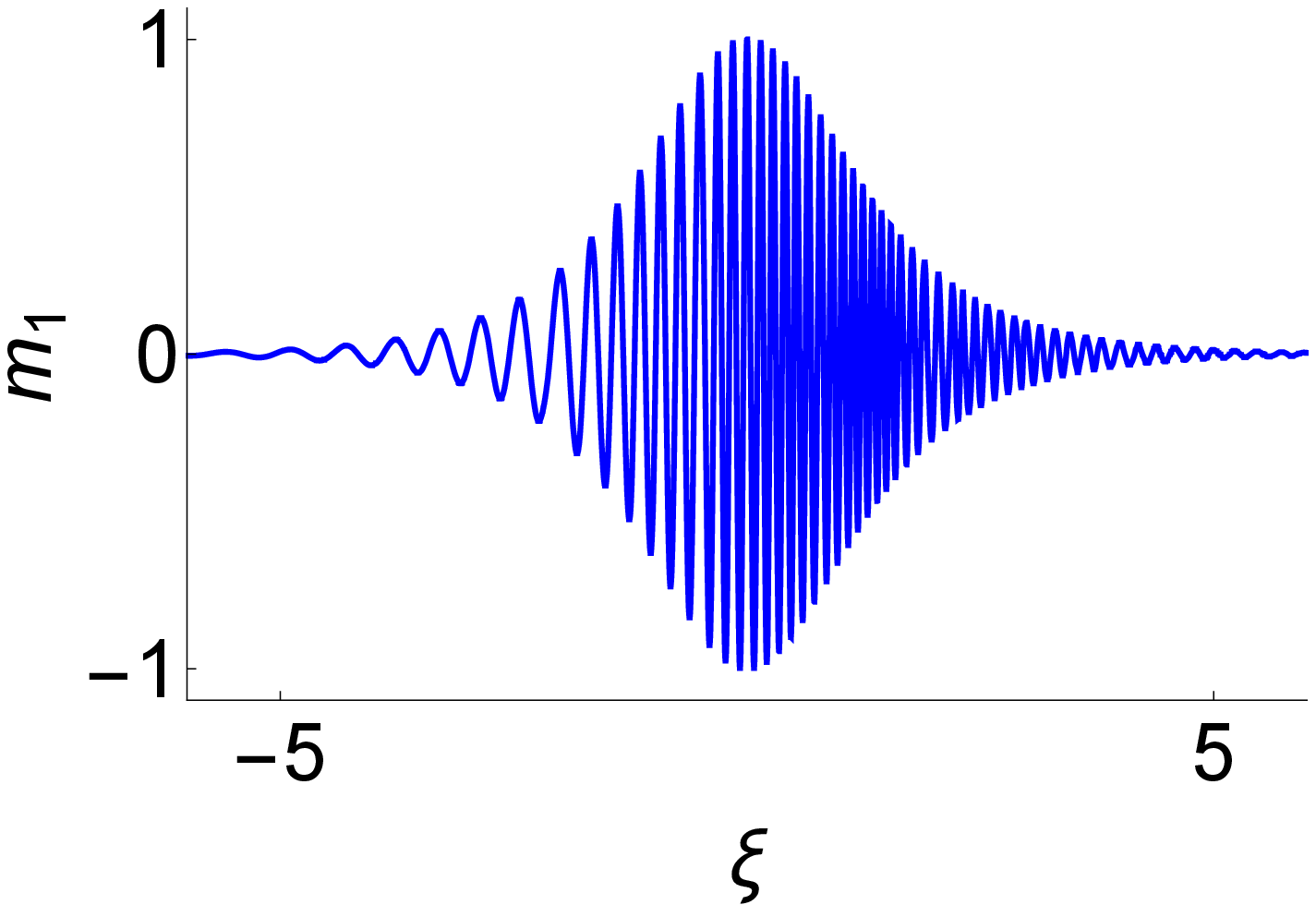}
        \caption{}
        \label{fig:continuation h=0.5 plus m1 component zoom-in}
    \end{subfigure}
    \hspace{0.25\textwidth}
    \begin{subfigure}[b]{0.25\textwidth}  
        \centering 
        \includegraphics[width=\textwidth]{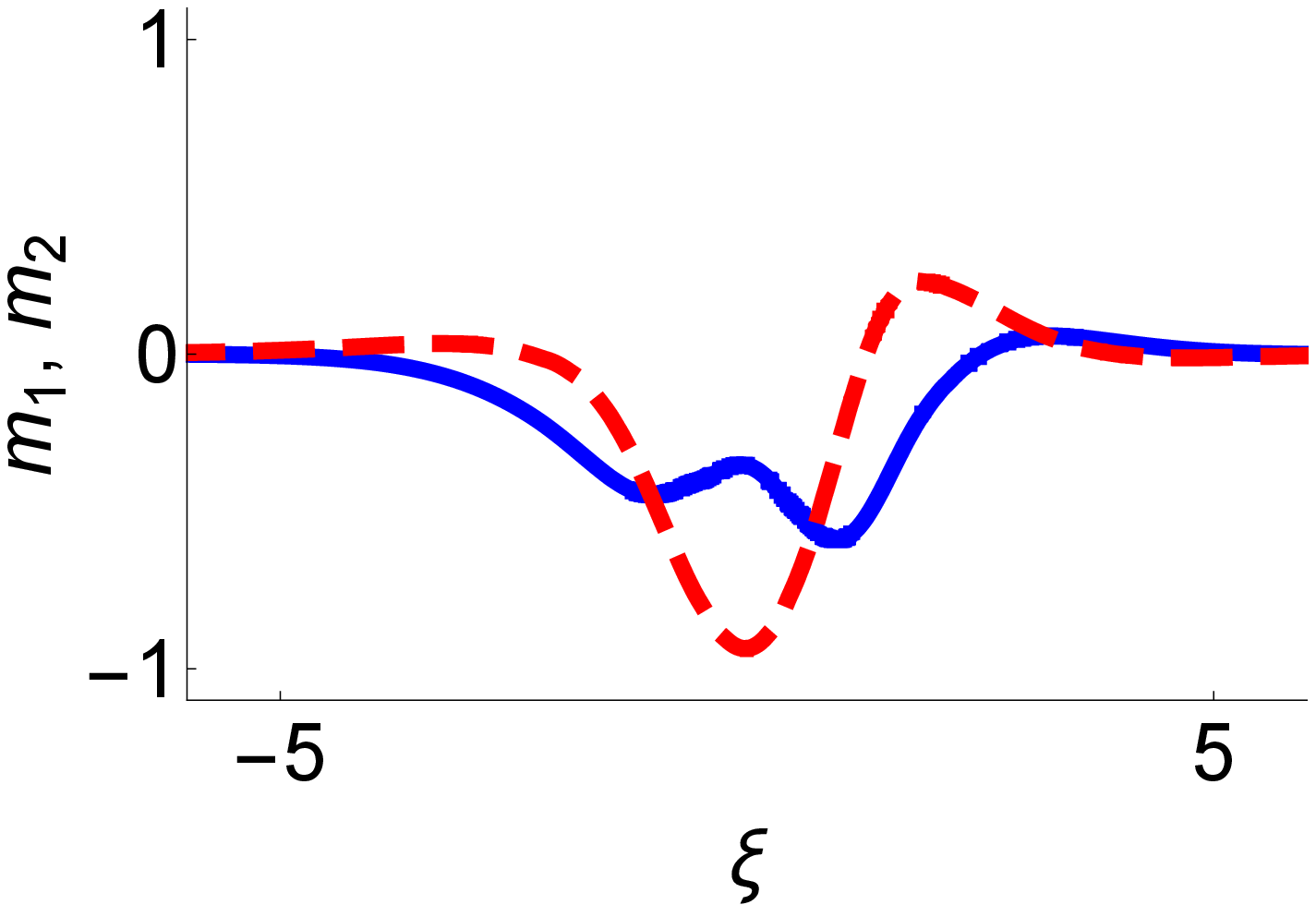}
        \caption{}
        \label{fig:continuation h=50 plus m1,m2 component}
    \end{subfigure}
    \caption{\small Shown are profiles of inhomogeneous DWs $\m(\xi)$ computed by parameter continuation, cf.\ \S\ref{sec:continuation}, in $\cc$ to $\cc=0.5$ with fixed $\alpha=0.5, \beta=0.1, \mu=-1$. (a,c) codim 2 case $h=0.5, s=0.112027, \Omega=0.447173$, (b,d) codim 0 case $h=50, s=19.92, \Omega=40.4$. (c) magnification of the $m_1$-component; note the change of frequency for small vs.\ large $\xi$. (d) Magnification of $m_1$ (blue solid) as well as $m_2$ (red dashed) component.} 
    \label{fig:inhomogeneous DW sphere}
\end{figure}

\medskip
An apparently thus far unrecognized distinction of DWs is based on the convergence behavior of $q$ as $\xi\to\pm\infty$. 

\begin{Definition}\label{def:flat}
We call a DW \emph{flat} if $|q(\xi)|$ has a limit on $\R\cup\lbrace\infty\rbrace$ as $|\xi|\to\infty$ and \emph{non-flat} otherwise.
\end{Definition}

Note that homogeneous DWs are flat ones by definition (recall $\phi'=q$). Moreover, for all DWs $m_0(\xi)$ converges to $\3$ or $-\3$ as $|\xi| \to\infty$.

\medskip
The main result of this paper is an essentially complete understanding of the existence and the type of DWs near the aforementioned explicit solution family for a nanowire geometry, i.e., $\mu<0$. This includes the LLG case $\beta\cdot\cc=0$, but our focus is on the spintronic case $\beta\cdot\cc \neq 0$ for which these results pertain $0<|\cc|\ll 1$ and any value of the (constant) applied field $h$.

\medskip
The different types of DWs occur in parameter regimes close to $\cc=0$ in the (spatial) ODE which results from the coherent structure ansatz. Since the parameters $\alpha$ and $\mu$ are material-dependent we take the applied field strength $h$ as the primary parameter.

\medskip
In brief, organized by stability properties of the steady states $\pm \3$ in the ODE, this leads to the following cases and existence results for localized DWs in nanowires ($\mu<0$):
\begin{itemize}
\item `codim-2' ($h_\ast<h< h^\ast$) : existence of flat inhomogeneous DWs with $s,\Omega$ selected by the other parameters,
\item `center' ($h=h_\ast$ or $h=h^\ast$) : existence of flat and non-flat inhomogeneous DWs, 
\item `codim-0' ($h<h_\ast$ or $h>h^\ast$) : existence of flat inhomogeneous DWs,
\end{itemize}
where $h_\ast\coloneqq \beta/\alpha +\frac{2 \mu}{\alpha^2} (1 + \alpha^2)$ as well as $h^\ast\coloneqq \beta / \alpha - \frac{2 \mu}{\alpha^2} (1 + \alpha^2)$. Note that $h_\ast<h^\ast$ always (recall $\alpha>0$ and $\mu<0$). Due to symmetry reasons, we mainly discuss the existence of right-moving DWs close to the explicit solution family and thus focus on an applied field $\beta/\alpha\le h$ (cf.\ \S3). The main results can be directly transferred to the case of left-moving DWs ($h\le\beta/\alpha$). Notably, the codim-0 case occurs for `large' magnetic field $h$ above a material dependent threshold. In the center and codim-2 cases there is a selection of $s$ and $\Omega$ by the existence problem.

\medskip The basic relation between the PDE and the ODE stability properties w.r.t.\ $h$ and $\cc$ are illustrated in Figure~\ref{fig:stability diagram} for $\alpha=1, \beta=0.5, \mu=-1$ fixed. Due to the fact that $s$ and $\Omega$ are ODE parameters only, the diagram illustrates a slice in the four dimensional parameter space with axes $h, \cc, s$, and $\Omega$. Note that homogeneous DWs (hom) can occur only on the line $\cc\equiv 0$ (see~\cite[Theorem 5]{Melcher2017} for details). The stability regions are defined as follows. \emph{monostable$^-$} (blue): $+\3$ unstable and $-\3$ stable, \emph{bistable} (shaded blue): both $+\3$ and $-\3$ stable, \emph{monostable$^+$} (red): $+\3$ stable and $-\3$ unstable, \emph{unstable} (shaded red): both $+\3$ and $-\3$ unstable. For a more detailed stability discussion, see Remark~\ref{rem:stab}. Note that the transition from bistable to monostable in the PDE does not coincide with the transition, of the homogeneous family, from codim-2 to codim-0 in the ODE. In contrast, the analogous transitions occur simultaneously for example in the well-known \emph{Allen-Cahn} or \emph{Nagumo} equation.

\begin{figure}[t]
    \centering
    \includegraphics[width=0.6\textwidth, trim=1cm 1cm 1cm 1cm, clip]{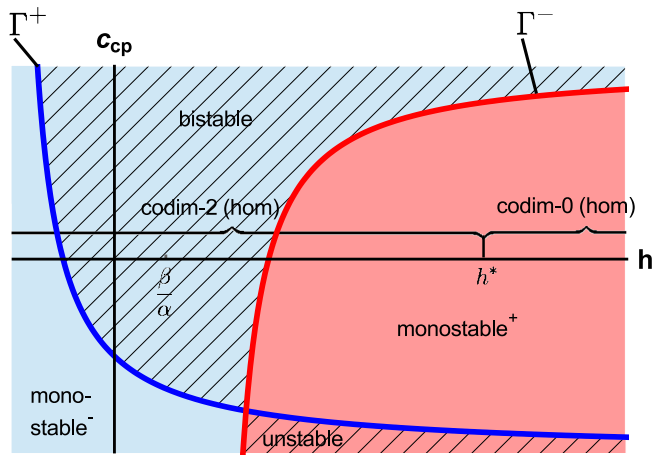}
    \caption{\small Stability diagram of homogeneous states $\pm\3$ in $h$ and $\cc$ for $\alpha=1, \beta=0.5$, and $\mu=-1$. State $+\3$ unstable to the left and stable to the right of $\Gamma^+$, $-\3$ stable to the left and unstable to the right of $\Gamma^-$. Homogeneous DWs (hom) exist only on the $h$-axis, i.e., $\cc\equiv0$. See text for further explanations. Note that also negative applied fields are shown.}
    \label{fig:stability diagram}
\end{figure}

In Figure~\ref{fig:freezselection} below, we present numerical evidence that inhomogeneous DWs are indeed also dynamically selected states, especially for large applied fields, also in the LLG case ($\beta\,,\cc=0$). 

The understanding of DW selection by stability properties generally depends on the existence problem discussed in this paper, which is therefore a  prerequisite for the \emph{dynamical} selection problem, cf.\ Remark~\ref{rem:stab}.

\begin{figure}[ht]
    \centering
    \begin{subfigure}[b]{0.25\textwidth}
    \includegraphics[width=\textwidth]{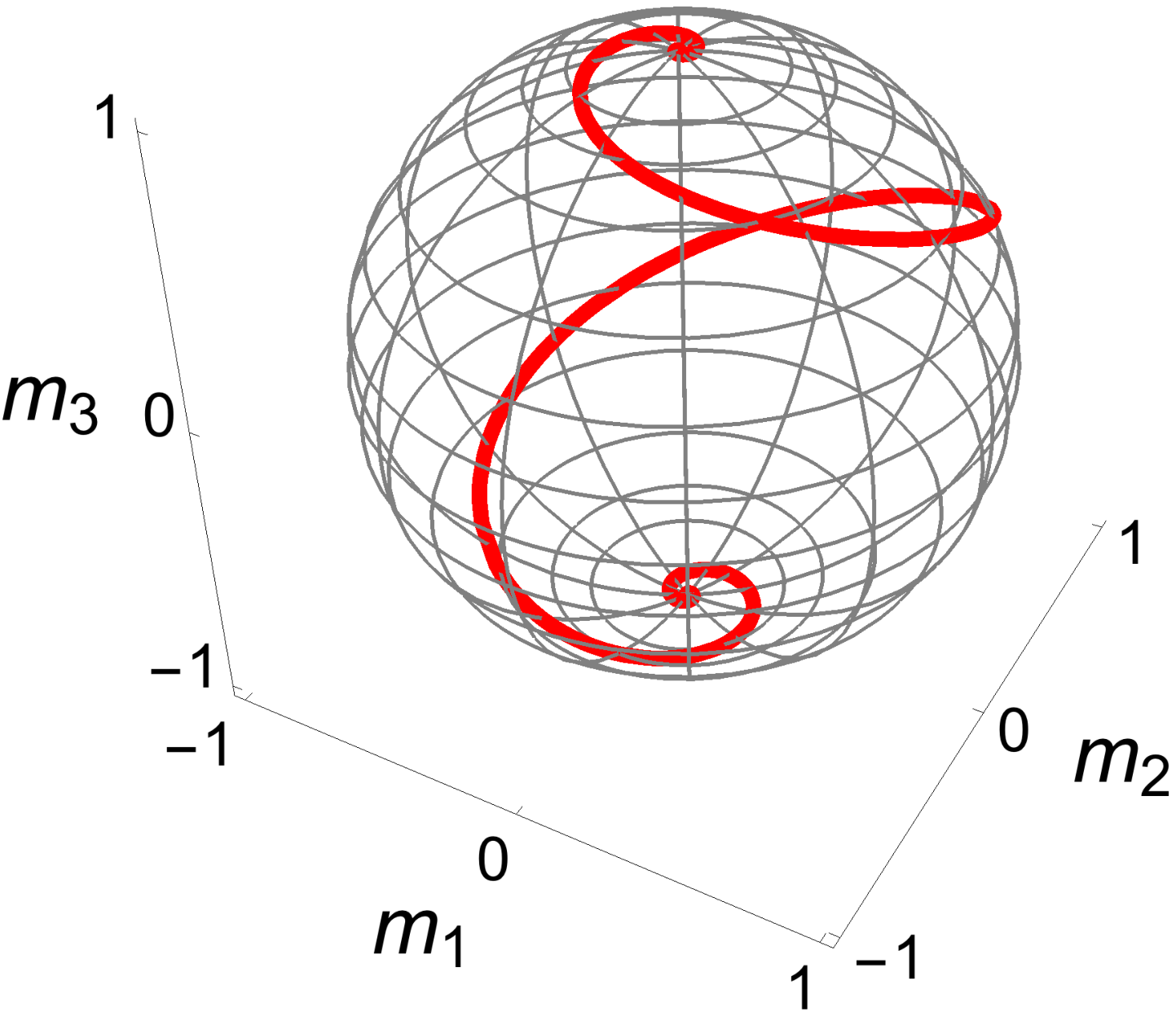}
    \vspace*{3mm}
    \caption{}
    \end{subfigure}
    \hfill
    \begin{subfigure}[b]{0.7\textwidth}
    \includegraphics[width=\textwidth]{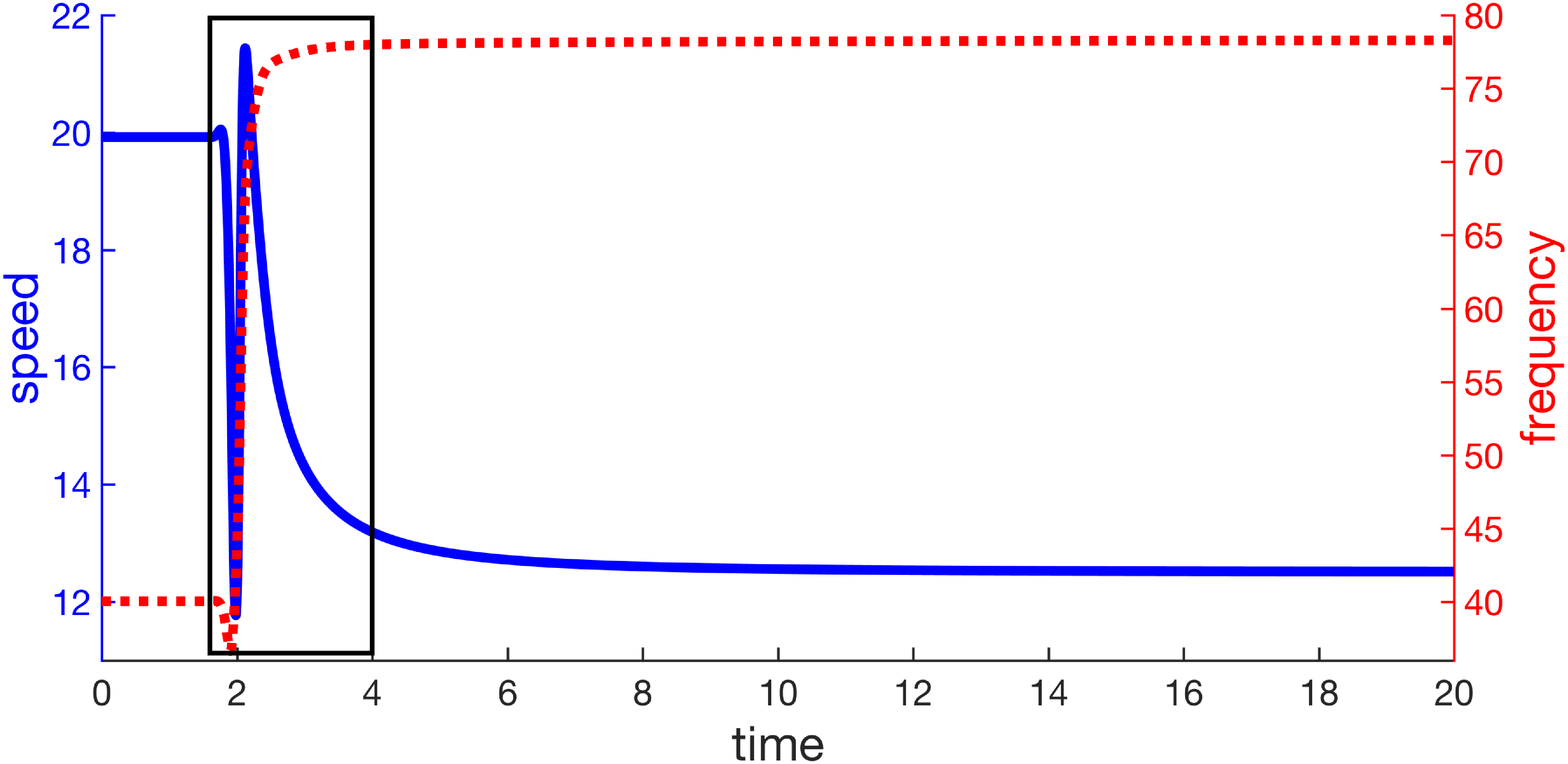}
    \caption{}
    \end{subfigure}
    
    \vspace*{12pt}
    \begin{subfigure}[b]{0.6\textwidth}
    \includegraphics[width=\textwidth,trim={3cm 1cm 6cm 2cm},clip]{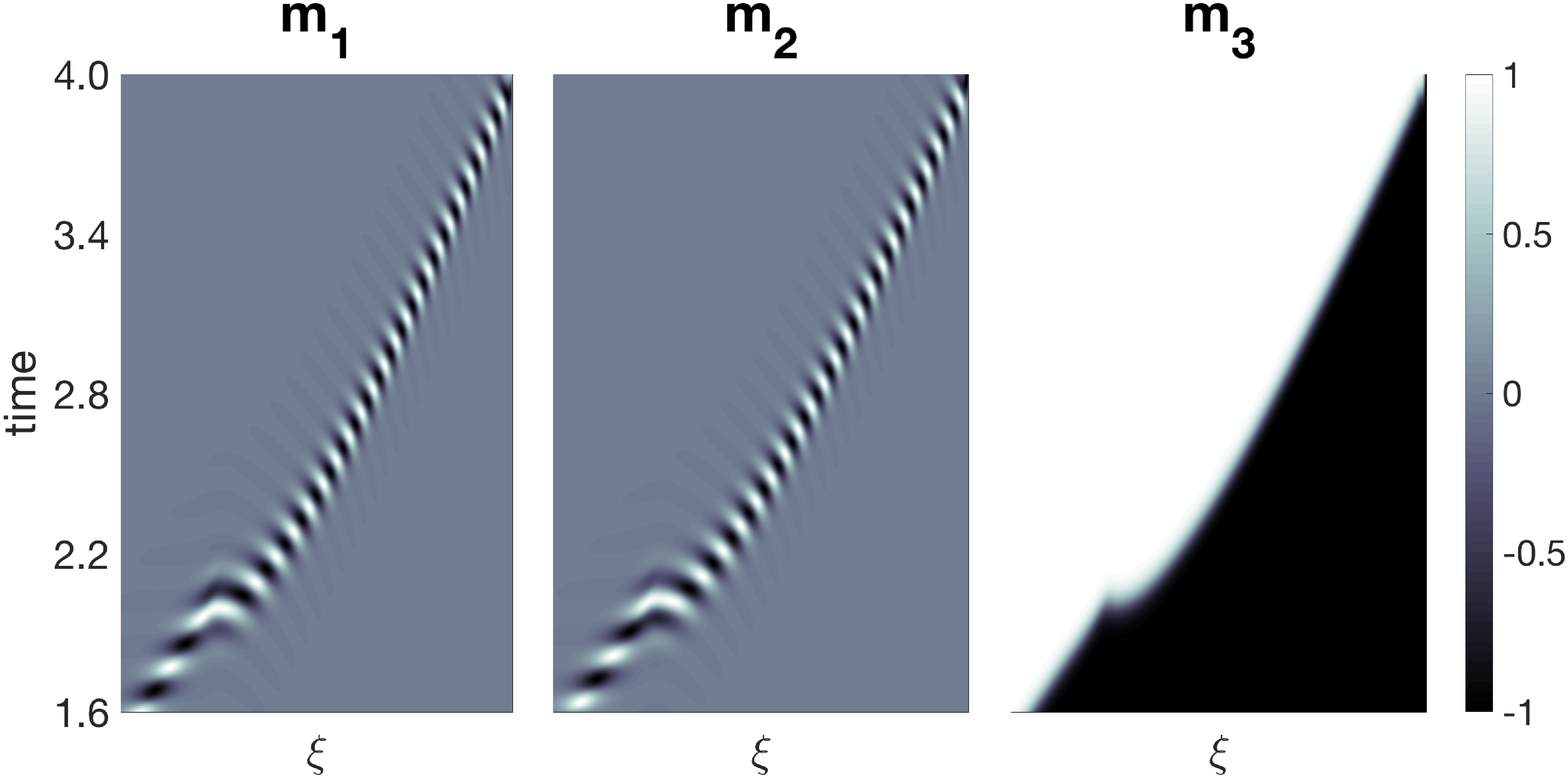}
    \caption{}
    \end{subfigure}
    \caption{\small Direct simulation of full PDE~\eqref{LLGS} for $\alpha=0.5, \beta=0.1, \mu=-1, h=50$, and $\cc=0$ with dynamical selection of an inhomogeneous DW. Initial condition near homogeneous DW~\eqref{eq:explicitsol} in codim-0 regime ($h^\ast=10.2, s_0=19.92$, and $\Omega_0=40.04$). (a) Profile at $t=20$ projected onto the sphere. (b) Speed and frequency of DW over time with asymptotic (selected) values $s=12.5$ and $\Omega=78.28$. (c) Space-time plots of DW components (without co-moving frame), range as in black box in (b). Final profile is heteroclinic connection in~\eqref{eq:coherent2}, cf.\ Proposition~\ref{prop:equilibriumconnection}.}
    \label{fig:freezselection}
\end{figure}

\medskip
To our knowledge, existence results of DWs for $\cc\neq0$ are new. In more detail, the existence of localised inhomogeneous, i.e.\ flat as well as non-flat, DWs for $\cc\neq0$ and especially for $\cc=0$ are new results. Indeed, the existence proof of non-flat DWs is the most technical result and entails an existence proof of heteroclinic orbits in an ODE between an equilibrium and a periodic orbit. These solutions indicate the presence of DWs in other regimes of spin driven phenomena and may be of interest for spin-torque transfer MRAM (Magnetoresistive random-access memory) systems~\cite{aakerman2005toward}.

\medskip
This paper is organized as follows. In \S2, the LLGS equation and coherent structures as
well as first properties are discussed. Section 3 more precisely introduces homogeneous and inhomogeneous as well as flat and non-flat DWs and it also includes the main results of this paper (Theorem 1, 2, and 3). The technical proofs of Theorem 2 as well as Theorem 3 are
deferred to Appendix 6.1 and 6.2. Section~\ref{sec:continuation} presents results of numerical continuation in parameter $\cc$ for the three regimes of the applied field (codim-2, center, and codim-0), where the center case is studied in more detail. We conclude with discussion and outlook in \S\ref{sec:discussion}.

\subsection*{Acknowledgements}
L.S.\ and J.R.\ acknowledge support by the Deutsche Forschungsgemeinschaft (DFG, German Research Foundation) - Projektnummer 281474342/GRK2224/1. J.R. also acknowledges support by DFG grant Ra 2788/1-1.
I.O.\ acknowledges funding of a previous position by Uni Bremen, where most of this paper was written, as well as support by the recent Russian Scientific Foundation grant 19-11-00280.

\section{Model equations and coherent structure form}\label{sec:model}

The classical model for magnetization dynamics was proposed by Landau and Lifschitz based on gyromagnetic precession, and later modified by Gilbert~\cite{landautheory, gilbert2004phenomenological}. See~\cite{lakshmanan2011fascinating} for an overview. The \emph{Landau-Lifschitz-Gilbert} equation for unit vector fields $\m(x,t) \in \mathbb{S}^2$ in one space dimension $x\in\R$ and in terms of normalized time in dimensionless form is
\begin{equation*}\label{LLG}
    \partial_t \m - \alpha \m \times \partial_t \m = -\m \times \h_{\textnormal{eff}}. \tag{LLG}
\end{equation*}
Here $\m=\M/M_S$ represents the normalized magnetization, $\h_{\textnormal{eff}}=\H_{\textnormal{eff}}/M_S$ the effective field, i.e.\ the negative variational derivative of the total magnetic free energy with respect to $\m$, both normalized by the spontaneous magnetization $M_S$. For gyromagnetic ratio $\gamma$ and saturation magnetization $M_S$ the time is measured in units of $(\gamma M_S)^{-1}$, and it is assumed that the temperature of the magnetic body is constant and below the \emph{Curie} temperature~\cite{mayergoyz2009nonlinear}. Finally, \emph{Gilbert} damping $\alpha>0$ turns $\m$ towards $\h_{\textnormal{eff}}$ and both vectors are parallel in the static solution.

\medskip
In modern spin-tronic applications, e.g.~Spin-Transfer Torque Magnetoresistive Random Access Memories (MRAM), the spin of electrons is flipped using a spin-polarized current. To take these effects into account, the~\ref{LLG} equation is supplemented by an additional spin transfer torque term. Using a semiclassical approach, Slonczewski derived an extended effective field 
$$\mathcal{H}_{\textnormal{eff}}= \h_{\textnormal{eff}} - \m \times \boldsymbol{J},$$ 
where $\boldsymbol{J}=\boldsymbol{J}(\m)$ depends on the magnetization and the second term is usually called Slonczewski term~\cite{slonczewski1996current}. In contrast to the~\ref{LLG} equation, which can be written as the gradient of free ferromagnetic energy, this generalized form is no longer variational and the energy is no longer a Lyapunov functional.

\medskip
As to the specific form of $\mathcal{H}_{\textnormal{eff}}$, including a leading order form of exchange interaction, uniaxial crystal anisotropy in direction $\3$, and \emph{Zeeman} as well as stray-field interactions with an external magnetic field, see e.g.~\cite{hubert2008magnetic}, gives the well known form~\eqref{eq:effective field}.

In this paper we consider a constant applied magnetic field $h\in\R$ along $\3$ and uniaxial anisotropy with parameter $\mu \in \R$, for which the anisotropy energy density is rotationally symmetric w.r.t.\ $\3$. According to the energetically preferred direction in the uniaxial case, minima of the anisotropy energy density correspond to \emph{easy} directions, whereas saddles or maxima correspond to \emph{medium-hard} or \emph{hard} directions, respectively. Therefore, one refers to $\mu<0$ as \emph{easy-axis} anisotropy and $\mu>0$ as \emph{easy-plane}, both with regard to $\3$. 

As mentioned before, the LLG equation with its variational structure appears as a special case of~\eqref{LLGS} for $\beta=0$ or $\cc=0$. While our main focus is the non-variational spintronic case $\beta\cdot \cc\neq 0$, all results contain the case $\beta\cdot\cc=0$ and thus carry over to~\eqref{LLG}.

\medskip
It is well-known that~\eqref{LLGS} admits an equivalent form as an explicit evolution equation of quasilinear parabolic type in the form, see e.g.~\cite{Melcher2017},
$$\partial_t \m=\partial_x\left(A(\m)\partial_x\m\right) + B\left(\m,\partial_x\m\right).$$

As a starting point, we briefly note the existence of spatially homogeneous equilibrium solutions of~\eqref{LLGS} for which $\m(x,t)$ is constant in $x$ and $t$.

\begin{Remark} The only $($spatially$)$ homogeneous equilibria of~\eqref{LLGS} for $\beta>0$ are the constant up- and down magnetization states $\pm\3$. Indeed, setting $\partial_t \m = \partial_x^2\m = 0$ in $($\ref{LLGS}$)$, for $\beta\neq0$ the last equation implies that $m_1 = m_2 = 0$ and thus the only solutions $\m_{\pm}^\ast \in \mathbb{S}^2$ are $\m_{\pm}^\ast=(0,0,\pm 1)\tran$.
\end{Remark}

\begin{Remark} In case $\beta=0$ as well as $|h / \mu| < 1$ there exist a family of additional homogeneous solutions of~\eqref{LLGS} given by
$\m^\ast=(m_1,m_2,h/\mu)\tran,$
with $m^2_1 + m^2_2 = 1 - (h / \mu)^2$. Note that similar cases occur for symmetry axis being $\boldsymbol{e}_1$ and $\boldsymbol{e}_2$, respectively $($cf.\ Brown's equations$)$.
\end{Remark}

\subsection{Coherent structure ODE}\label{sec:coherent structure ODE}
Due to the rotation symmetry around the $\3$-axis of~\eqref{LLGS}, it is natural to use  spherical coordinates
$$
\m = \begin{pmatrix}
\cos(\varphi) \sin(\theta) \\ \sin(\varphi) \sin(\theta) \\ \cos(\theta)
\end{pmatrix},
$$ 
where $\varphi = \varphi(x, t)$ and $\theta = \theta(x, t)$. This changes~\eqref{LLGS} to
\begin{align}\label{eq:spherical}
\begin{split}
\begin{pmatrix}
\alpha & -1 \\ 1 & \alpha
\end{pmatrix} \begin{pmatrix} \partial_t\vph \sin (\theta) \\ -\partial_t \theta \end{pmatrix} = &\begin{pmatrix} 2 \partial_x \vph \partial_x\theta \cos (\theta) \\ - \partial_x^2 \theta \end{pmatrix}\\ &+\sin (\theta) \begin{pmatrix} \partial_x^2 \vph + \beta / \left(1+\cc \cos (\theta)\right) \\(\partial_x \vph)^2 \cos (\theta) + h-\mu \cos (\theta) \end{pmatrix}
\end{split}
\end{align}

Note that the rotation symmetry has turned into the shift symmetry in the azimutal angle $\varphi$, as \eqref{eq:spherical} depends on derivatives of $\varphi$ only.

\medskip Recall that DW solutions spatially connect the up and down magnetization states $\pm \3$ in a coherent way as relative equilibria with respect to the translation symmetry in $x$ and $\phi$, which yields the ansatz
\begin{equation}\label{eq:ansatz}
\xi\coloneqq x - st, \,\theta =\theta(\xi),\,\vph=\phi(\xi)+\Omega t.
\end{equation}
Such solutions are generalized travelling waves that move with constant speed $s\in \R$ in space and rotate pointwise with a constant frequency $\Omega \in \R$ around the $\3$-axis; solutions with $\Omega=0$ are classical travelling waves.

As in~\cite{Melcher2017}, applying ansatz~\eqref{eq:ansatz} to~\eqref{eq:spherical} leads to the so-called \emph{coherent structure ODE}  
\begin{equation}\label{eq:coherent structure ODE}
\begin{split}
\begin{pmatrix}
\alpha & -1 \\ 1 & \alpha
\end{pmatrix} \begin{pmatrix} (\Omega - s\phi')\sin(\theta)\\s\theta' \end{pmatrix} = &\begin{pmatrix} 2 \phi'\theta' \cos (\theta) \\ -\theta'' \end{pmatrix}\\ &+\sin (\theta) \begin{pmatrix} \phi'' + \beta / \left(1+\cc \cos (\theta)\right) \\(\phi')^2 \cos (\theta) + h-\mu \cos (\theta) \end{pmatrix},
\end{split}
\end{equation}
where $'=d/d\xi$. This system of two second-order ODEs does not depend on $\phi$ and thus reduces to three dynamical variables $(\theta, \psi = \theta', q = \phi')$. Following standard terminology for coherent structures, we refer to $q$ as the \emph{local wavenumber}.

\medskip
Writing \eqref{eq:coherent structure ODE} as a first-order three-dimensional system gives
\begin{equation}\label{eq:coherent1}
    \begin{array}{rl}
    \theta' &= \psi\\
    \psi' &= \sin(\theta) \left[h - \Omega +sq + (q^2-\mu) \cos(\theta)\right] -\alpha s \psi\\
    q' &= \alpha \Omega - \beta/(1+\cc\cos(\theta)) - \alpha sq - \frac{s+2q \cos(\theta)}{\sin (\theta)}\psi
    \end{array},
\end{equation}
and DWs in the original PDE are in 1-to-1-correspondence with the ODE solutions connecting $\theta=0$ and $\theta=\pi$.

\subsubsection{Blow-up charts and asymptotic states}\label{sec:blow-up} As in~\cite{Melcher2017}, the  singularities at zeros of $\sin (\theta)$ in~\eqref{eq:coherent1} can be removed by the singular coordinate change $\psi:=p\sin(\theta)$, which is a blow-up transformation mapping the poles of the sphere $\pm \3$ to circles thus creating a cylinder. The resulting \emph{desingularized} system reads

\begin{equation}\label{eq:coherent2}
    \begin{array}{rl}
    \theta' &= \sin(\theta)p\\
    p' &= h - \Omega -\alpha s p +sq - (p^2-q^2+\mu) \cos(\theta)\\
    q' &= \alpha \Omega - \beta/(1+\cc\cos(\theta)) -sp - \alpha sq - 2pq\cos(\theta).
    \end{array}
\end{equation}

The coherent structure system~\eqref{eq:coherent1} is equivalent to the desingularized system~\eqref{eq:coherent2} for $\theta \neq n \pi\,, n \in \Z$ and therefore also for $\m$ away from $\pm \3$. Furthermore, the planar \emph{blow-up charts} $\theta = 0$ and $\theta = \pi$ are invariant sets of~\eqref{eq:coherent2}, which are mapped to the single points $\3$ and $-\3$, respectively by the blow-down transformation. System~\eqref{eq:coherent2} has a special structure (cf.\ Figure~\ref{fig:streamplot}) that will be relevant for the subsequent DW analysis. In the remainder of this section we analyze this in some detail.

\begin{Lemma}\label{lem:explicit}
Consider the equations for $p$ and $q$ in \eqref{eq:coherent2} for an artificially fixed value of $\theta$. In terms of $z\coloneqq p+\rmi q$ this subsystem can be written as the complex (scalar) ODE 
\begin{equation}\label{eq:complex ODE}
    z'=Az^2+Bz+C,
\end{equation}
where
$
A\coloneqq -\cos(\theta)\,, B\coloneqq -(\alpha+\rmi)s\,, \textnormal{ and }C^\theta\coloneqq h-\Omega +A \mu + \rmi\left(\alpha \Omega - \frac{\beta}{1-A\cc}\right)$. 

For $A\neq 0$ the solution with $z_0=z(\xi_0)$ away from the equilibria $z^\theta_+=-\frac{B}{2A} +\rmi\frac{\gamma^\theta}{2A}$ and $z^\theta_-=-\frac{B}{2A} -\rmi\frac{ \gamma^\theta}{2A}$, with $\gamma^\theta=\gamma(\theta) \coloneqq \sqrt{\textstyle{4AC^\theta-B^2}}$ , reads
\begin{equation}\label{eq:explicitsol}
z(\xi)=\frac{\gamma^\theta}{2A}\tan\left(\frac{\gamma^\theta}{2}\xi+\delta_0\right)-\frac{B}{2A}\,,
\end{equation}
where $$\delta_0=\arctan\left(\frac{2Az_0+B}{\gamma^\theta} \right)-\frac{\gamma^\theta}{2}\xi_0.$$
For $A=0$, the solution away from the equilibrium $z^{\pi/2}=-C^{\pi/2}/B$ is given by 
$$z(\xi)=\left(z_0+\frac{C^{\pi/2}}{B}\right)\rme^{B(\xi-\xi_0)}-\frac{C^{\pi/2}}{B}.$$
\end{Lemma}

Clearly, the solution of~\eqref{eq:complex ODE} relates only to those solutions of \eqref{eq:coherent2} for which $\theta$ is constant, i.e., $\theta=0, \pi$. Although we are mostly interested in the dynamics on the blow-up charts, we consider $\theta$ as a parameter in order to demonstrate the special behaviour of~\eqref{eq:coherent2} for $\theta$ artificially fixed. Notably, the equilibria $z_\pm^\theta$ of~\eqref{eq:complex ODE} for $\theta\neq0, \pi$ are not equilibria in the full dynamics, due to the fact that~\eqref{eq:coherent2} is only invariant for $\theta$ at the blow-up charts.

\begin{proof}
We readily verify the claimed form of the ODE and directly check the claimed solutions.
\end{proof}

\begin{Remark}\label{rem:equilibria}
Lemma~\ref{lem:explicit} states in particular that the desingularized ODE system~\eqref{eq:coherent2} can be solved explicitly on the invariant blow-up charts, where $\theta=0,\pi$ and thus $A=-1,1$, respectively. System~\eqref{eq:coherent2} possesses two real equilibria on each blow-up chart, 
$Z^0_\pm\coloneqq(0,p^0_\pm,q^0_\pm)\tran$ and $Z^\pi_\pm\coloneqq(\pi,p^\pi_\pm\,,q^\pi_\pm)\tran$.
Here $p^\theta_\sigma\coloneqq\Re(z^\theta_\sigma)\,,q^\theta_\sigma\coloneqq\Im(z^\theta_\sigma)$ for $\theta=0,\pi$, $\
\sigma=\pm$ and
$$z^0_+\coloneqq 1/2(B-\rmi\gamma^0)\,,\quad z^0_-\coloneqq 1/2(B+\rmi \gamma^0)$$
and analogously
$$z^\pi_+\coloneqq 1/2(-B+\rmi\gamma^\pi)\,,\quad z^\pi_-\coloneqq 1/2(-B-\rmi \gamma^\pi)\,,$$
where we set
$$\gamma^0\coloneqq \gamma\big\vert_{A=-1}=\sqrt{-4C^0-B^2} \quad \textnormal{and} \quad \gamma^\pi\coloneqq \gamma\big\vert_{A=1}=\sqrt{4C^\pi-B^2}$$
with $C^0\coloneqq C\big\vert_{A=-1}$ and $C^\pi\coloneqq C\big\vert_{A=1}$.
\end{Remark}

Due to the analytic solution~\eqref{eq:explicitsol}, we obtain the following more detailed result in case $\theta\neq\pi/2$ (cf.\ Figure~\ref{fig:streamplot}).

\begin{Lemma}\label{lem:chartconnections}
For each given $0\le\theta\le\pi$ with $\theta\neq \pi/2$ as a parameter, the fibers of~\eqref{eq:coherent2} with constant $\theta$ consists entirely of heteroclinic orbits between $z^\theta_-$ and $z^\theta_+$ in case $\Im(\gamma^\theta)\neq0$, or $\gamma^\theta\neq 0$, except for the equilibrium states. In case $\Im(\gamma^\theta)=0$ and $\Re(\gamma^\theta)\neq0$, the fiber at fixed $\theta$ is filled with periodic orbits away from the invariant line $\{q=\frac{s}{2A}\}$, for which the period of solutions close to it tends to infinity.
\end{Lemma}

\begin{proof}For $\theta$ fixed in~\eqref{eq:complex ODE}, consider the case $\Re(\gamma^\theta)=0$ and also $\Im(\gamma^\theta)\neq0$ for $A\neq0$ which leads to
\begin{align*}z(\xi)&=\rmi \frac{\Im(\gamma^\theta)}{2A}\cdot \tan\bigg(\rmi\Big( \underbrace{\frac{\Im(\gamma^\theta)}{2}\xi+\Im(\delta_0)}_{\eqqcolon \check{\xi}}\Big)+\Re(\delta_0)\bigg)-\frac{B}{2A}\\
&=\frac{\Im(\gamma^\theta)}{2A}\cdot\frac{\rmi\sin\left( 2\Re(\delta_0) \right) - \sinh \left( 2\check{\xi} \right)}{\cos\left(2\Re(\delta_0)\right) + \cosh\left( 2\check{\xi} \right)} - \frac{B}{2A}\,.
\end{align*}
For $\Re(\gamma^\theta)\neq0$ as well as $\Im(\gamma^\theta)\neq0$, we obtain
\begin{align*}z(\xi)&=\frac{\gamma^\theta}{2A}\tan \bigg( \underbrace{\frac{\Re(\gamma^\theta)}{2}\xi + \Re(\delta_0)}_{\eqqcolon \tilde{\xi}}+\rmi \frac{\Im(\gamma^\theta)}{2}\xi +\rmi \Im(\delta_0) \bigg) - \frac{B}{2A}\\&=\frac{\gamma^\theta}{2A} \cdot \frac{\sin(2\tilde{\xi}) + \rmi \sinh\left(2\left(\frac{\Im(\gamma^\theta)}{\Re(\gamma^\theta)}\tilde{\xi}- \frac{\Im(\gamma^\theta)}{\Re(\gamma^\theta)}\Re(\delta_0)+\Im(\delta_0)\right)\right)}{\cos(2\tilde{\xi})+\cosh\left(2\left(\frac{\Im(\gamma^\theta)}{\Re(\gamma^\theta)}\tilde{\xi}- \frac{\Im(\gamma^\theta)}{\Re(\gamma^\theta)}\Re(\delta_0)+\Im(\delta_0)\right)\right)}-\frac{B}{2A}\,,
\end{align*}
The asymptotic states are 
$$\Im\left(\gamma^\theta\right)>0:\qquad \lim_{\xi\to-\infty}z(\xi)=-\rmi \frac{\gamma^\theta}{2A}-\frac{B}{2A}\,,\quad \lim_{\xi\to+\infty}z(\xi)=\rmi \frac{\gamma^\theta}{2A}-\frac{B}{2A}\,,$$
as well as
$$\Im\left(\gamma^\theta\right)<0:\qquad \lim_{\xi\to-\infty}z(\xi)=\rmi \frac{\gamma^\theta}{2A}-\frac{B}{2A}\,,\quad \lim_{\xi\to+\infty}z(\xi)=-\rmi \frac{\gamma^\theta}{2A}-\frac{B}{2A}\,,$$
which simplify in case $\Re(\gamma^\theta)=0$ to
$$\Im\left(\gamma^\theta\right)>0:\qquad \lim_{\xi\to-\infty}z(\xi)=\frac{\Im(\gamma^\theta)-B}{2A}\,,\quad \lim_{\xi\to+\infty}z(\xi)=-\frac{\Im(\gamma^\theta)+B}{2A}\,,$$
as well as
$$\Im\left(\gamma^\theta\right)<0:\qquad \lim_{\xi\to-\infty}z(\xi)=-\frac{\Im(\gamma^\theta)+B}{2A}\,,\quad \lim_{\xi\to+\infty}z(\xi)=\frac{\Im(\gamma^\theta)-B}{2A}\,.$$
Note that the asymptotic states coincide if $\gamma^\theta=0$.

The last case to consider is $\Re\left(\gamma^\theta\right)\neq0$ and $\Im\left(\gamma^\theta\right)=0$, where the solutions are
$$z(\xi)=\frac{\Re\left(\gamma^\theta\right)}{2A}\cdot \frac{\sin(2\hat{\xi})+\rmi\sinh(2\Im(\delta_0))}{\cos(2\hat{\xi})+\cosh(2\Im(\delta_0))}-\frac{B}{2A}\,,$$
with $\hat{\xi}\coloneqq \frac{\Re(\gamma^\theta)}{2}\xi+\Re(\delta_0)$ and which leads to periodic solutions of~\eqref{eq:complex ODE} iff
\begin{align*}\Im(\delta_0)\neq 0&\Leftrightarrow 2A\Im(z_0)+\Im(B)\neq0\Leftrightarrow q_0 \neq \frac{s}{2A}\,,
\end{align*}
where $z_0=p_0+\rmi q_0$ and recall that $B=-(\alpha+\rmi)s$.
\end{proof}

Based on Lemma~\ref{lem:chartconnections}, explicitly on the blow-up chart $\theta=0$ the heteroclinic orbits are from $z_-^0$ to $z_+^0$ in case $\Im(-4C^0-B^2)>0$, or $\Im(-4C^0-B^2)=0$ and $\Re(-4C^0-B^2)\le 0$, and from $z_+^0$ to $z_-^0$ if $\Im(-4C^0-B^2)<0$ . For $\theta=\pi$, if $\Im(4C^\pi-B^2)>0$, or $\Im(4C^\pi-B^2)=0$ and $\Re(4C^\pi-B^2)\le 0$ they are connections from $z_-^\pi$ to $z_+^\pi$, and if $\Im(4C^\pi-B^2)<0$ from $z_+^\pi$ to $z_-^\pi$.

For $A\neq0$, the case $s=0$ is a special situation, which will be also discussed in the context of DWs in \S~\ref{sec:domain walls} later. It turns out that on the blow-up charts $\theta=0$ (or $\theta=\pi$), the solution with appropriate initial conditions has a limit as $|\xi|\to \infty$ if and only if $\Im(\sqrt{\textstyle{-C^0}})\neq0$ ($\Im(\sqrt{\textstyle{C^\pi}})\neq 0$). In terms of the parameters in~\eqref{eq:coherent2} and with
$$\b^-\coloneqq \frac{\beta}{1-\cc} \qquad \textnormal{and} \qquad \b^+\coloneqq \frac{\beta}{1+\cc}\,,$$
this leads to the conditions for $\theta=0$ given by:
\begin{equation}\label{eq:frequenzyzerospeedleft}
    \Omega\neq \frac{\b^+}{\alpha}\quad \textnormal{or} \quad \Omega=\frac{\b^+}{\alpha} \textnormal{ and } \Omega\le h-\mu\,,
\end{equation}
and for $\theta=\pi$ given by:
\begin{equation}\label{eq:frequenzyzerospeedright}
    \Omega\neq \frac{\b^-}{\alpha}\quad \textnormal{or} \quad \Omega=\frac{\b^-}{\alpha} \textnormal{ and } \Omega\ge h+\mu\,,
\end{equation}
In case $\cc=0$, i.e.\ for the~\ref{LLG} equation, the conditions in~\eqref{eq:frequenzyzerospeedleft} and~\eqref{eq:frequenzyzerospeedright} reduce to $$\Omega\neq \frac{\beta}{\alpha} \quad \textnormal{or} \quad \Omega=\frac{\beta}{\alpha} \textnormal{ and } 2\mu\le \frac{\beta}{\alpha}\,,$$ where the latter inequality always holds in case of a nanowire geometry ($\mu<0$). Hence standing domain walls in nanowires in case $\cc=0$ can only connect equilibria, if they exist.

\medskip
Lemma~\ref{lem:chartconnections} also states that the equilibria on the blow-up charts $\theta\in\{0,\pi\}$ are surrounded by periodic orbits in case $\Im(\gamma^0)=0$ and $\Re(\gamma^0)\neq0$ ($\Im(\gamma^\pi)=0$ and $\Re(\gamma^\pi)\neq0$). In fact, system~\eqref{eq:coherent2} is Hamiltonian (up to rescaling) on the blow-up charts for certain frequencies $\Omega$, as follows

\begin{figure}[ht]
    \centering
    \begin{subfigure}[b]{0.25\textwidth}
        \centering
        \includegraphics[width=\textwidth]{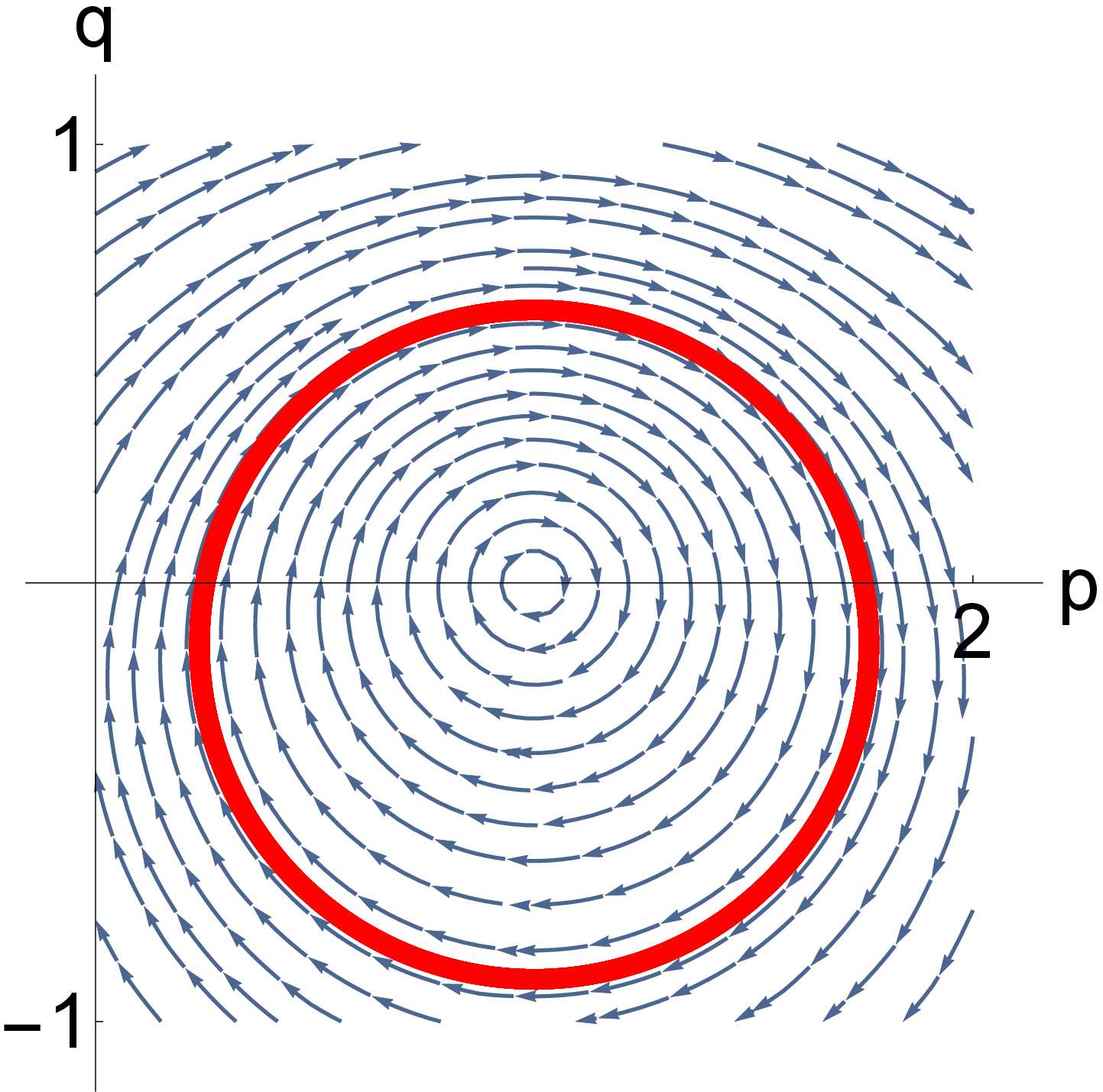}
        \caption{}
        \label{fig:hamiltonphase}
    \end{subfigure}
    \hspace*{20mm}
    \begin{subfigure}[b]{0.25\textwidth}
        \centering 
        \includegraphics[width=\textwidth]{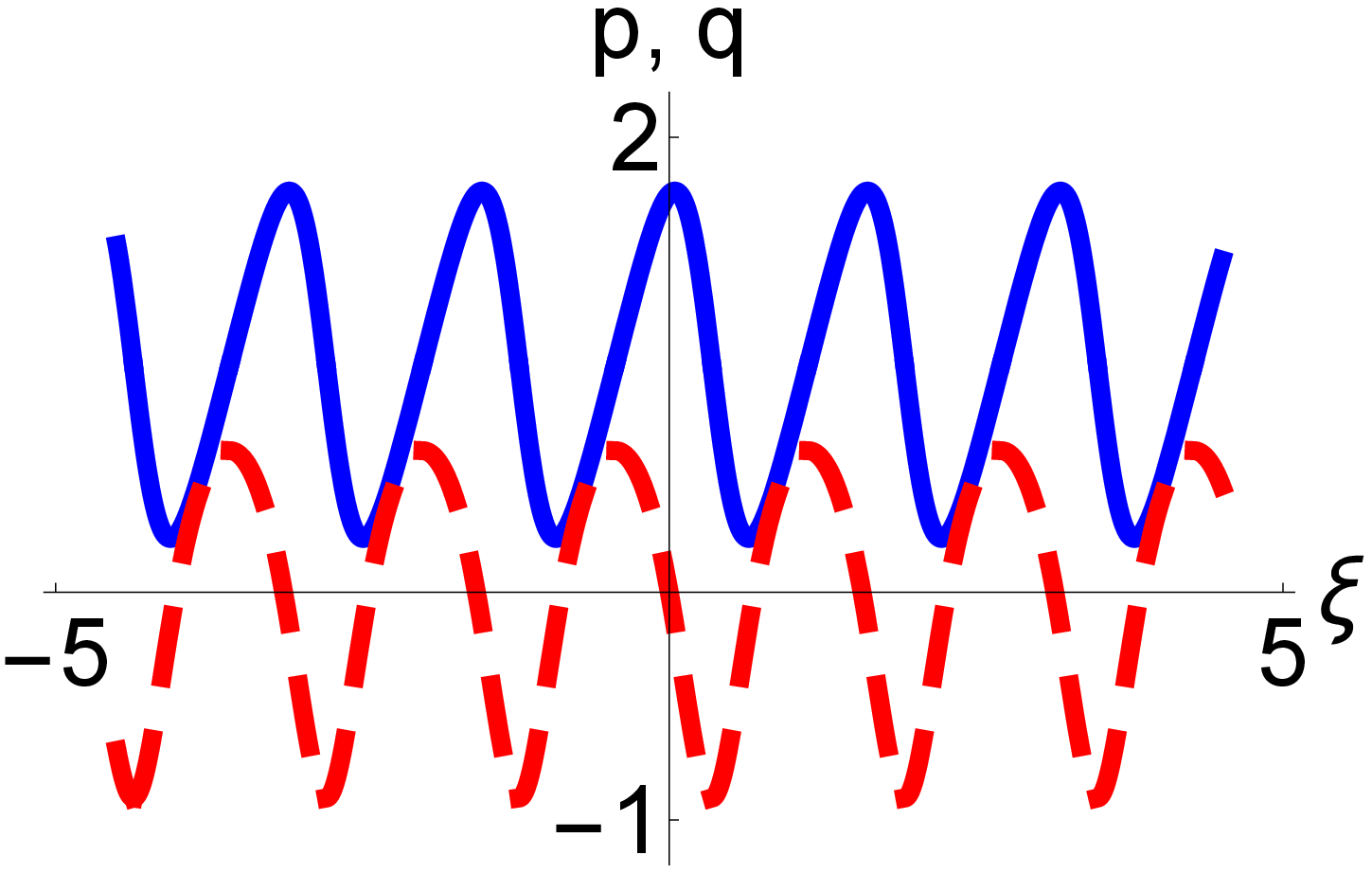}
        \vspace*{3mm}
        \caption{}
        \label{fig:initial arctan hamilton}
    \end{subfigure}
    \caption{\small (a) Phase plane streamplot with \textsc{Mathematica} of~\eqref{eq:hamilton} around the equilibrium $z^\pi_-$, i.e., \eqref{eq:coherent2} at $\theta=\pi$, for $\alpha=0.5, \beta=0.1, \mu=-1, h=10.2, s=4, \Omega=8.2$ and $\cc=0$, which leads to $\left(p^\pi_-,q^\pi_-\right)\tran=\left(1,0\right)\tran$. The red solid line marks the trajectory with initial condition $(p_0,q_0)=(7/4,0)$ (cf.~plot of solutions in b). (b) Plot of the profile for the solution highlighted in (a), where $p$ (solid blue line) and $q$ (dashed red line) are given by~\eqref{eq:explicitsol} for the parameter set as in (a).}
    \label{fig:Hamilton}
\end{figure}

\begin{Proposition}\label{prop:hamilton_general}
The dynamics of~\eqref{eq:coherent2} on the invariant blow-up chart $\theta=0$ in case $\Omega=\frac{\b^+}{\alpha}-\frac{s^2}{2}$ possesses the invariant line $\{q=-\frac{s}{2}\}$ and, after time-rescaling, for $q\neq -\frac{s}{2}$ the Hamiltonian 
$$H^0(p,q)=-\frac{p^2+q^2+\alpha s p +sq -h+\b^+/\alpha +\mu}{q+\frac{s}{2}}$$
along solutions of~\eqref{eq:complex ODE}. Analogously on the chart $\theta=\pi$, in case
\begin{equation}\label{eq:center case}
    \Omega=\frac{\b^-}{\alpha}+\frac{s^2}{2}
\end{equation}
possesses the invariant line $\{q=\frac{s}{2}\}$ and for $q\neq s/2$ the Hamiltonian

\begin{equation}\label{eq:hamiltonian fct}
H^\pi(p ,q)=\frac{p ^2+q^2-\alpha s p-sq+h-\b^-/\alpha+\mu}{q-\frac{s}{2}}.
\end{equation}

Moreover, each half plane $\{\theta=0, q\le -\frac{s}{2}\}$, $\{\theta=0, q\ge - \frac{s}{2} \}$ $\left(\{\theta=\pi, q\le \frac{s}{2}\}\right.$ and $\left.\{\theta=\pi, q\ge \frac{s}{2}\}\right)$ is filled with periodic orbits encircling the equilibria at $z_\pm^0$ $\left(z^\pi_\pm\right)$ if additionally $\Omega>h-\mu+\frac{s^2}{4}(\alpha^2-1)$ $\left( \Omega<h+\mu+\frac{s^2}{4}(1+\alpha^2) \right)$.
\end{Proposition}

\begin{proof} For the sake of clarity, we will only present the computation for the blow-up chart $\theta=\pi$; the computation on $\theta=0$ can be done in the same manner. With respect to the parameters of equation~\eqref{eq:coherent2}, the condition $\Im(4C-B^2)=0$ is equivalent to $\Omega=\frac{\b^-}{\alpha}+\frac{s^2}{2}$ and in this case the system~\eqref{eq:coherent2} on $\{\theta=\pi\}$ is given by

\begin{equation}\label{eq:hamilton}
    \begin{array}{rl}
    p' &= p^2 - q^2 - \alpha s p + s q +h+\mu-\frac{\b^-}{\alpha}-\frac{s^2}{2} \\
    q' &= 2pq- s p- \alpha s q+\frac{\alpha s^2}{2}
    \end{array}
\end{equation}
We readily compute that for solutions of this
$$\frac{dH^\pi}{d\xi}=\frac{\partial H^\pi}{\partial p}p' + \frac{\partial H^\pi}{\partial q}q'=\frac{q'\cdot p'-p'\cdot q'}{(q-\frac{s}{2})^2} = 0\,,$$
which shows the canonical Hamiltonian structure of~\eqref{eq:hamilton} up to time rescaling. If additionally $\Omega<h+\mu+\frac{s^2}{4}(1+\alpha^2)$, it follows from Lemma~\ref{lem:chartconnections} that each half plane is filled with periodic orbits.
\end{proof}

Proposition~\ref{prop:hamilton_general} concerns the special case that $\cc\in(-1,1)$ and $\beta, \Omega$ are such that~\eqref{eq:center case} holds, which is henceforth referred to as the \textbf{center-case}. In particular, each orbit except the line $q\equiv s/2$ on the blow-up chart $\theta=\pi$ can by identified via the quantity~\eqref{eq:hamiltonian fct}, and each equilibrium $z^\pi_\pm$ has a neighborhood filled with periodic orbits if additionally $h>\frac{\b^-}{\alpha}-\mu+\frac{s^2}{4}(1-\alpha^2)$ (cf.\ Figure~\ref{fig:Hamilton}). Note the relation between the conditions~\eqref{eq:frequenzyzerospeedleft} and~\eqref{eq:frequenzyzerospeedright} and the conditions in Proposition~\ref{prop:hamilton_general} in case $s=0$.

\medskip
Based on Lemma~\ref{lem:chartconnections}, we also state the following uniqueness result.

\begin{Proposition}\label{prop:equilibriumconnection}
For $\Omega<\frac{\b^+}{\alpha}-\frac{s^2}{2}$ $\left[ \Omega>\frac{\b^+}{\alpha}-\frac{s^2}{2}\right]$, or $\Omega=\frac{\b^+}{\alpha}-\frac{s^2}{2}$ and $\Omega\le h-\mu+\frac{s^2}{4}(\alpha^2-1)$ there is a unique orbit with $(\theta,p,q)\tran(\xi)$ with $\theta(\xi)\to 0$ as $\xi\to-\infty$, and it holds that $(p+\rmi q)(\xi)\to z^0_-$ $\left[ (p+\rmi q)(\xi)\to z_+^0\right]$ as $\xi\to-\infty$ .
\end{Proposition}
\begin{proof}
The conditions on $\Omega$ are equivalent to those in Lemma~\ref{lem:chartconnections}. If the statement were false, it nevertheless follows from Lemma~\ref{lem:chartconnections} that $(p+\rmi q)(\xi)\to z^0_-$ as $\xi\to-\infty$. However,  transverse to the blow-up chart, the equilibrium state $Z^0_-$ is stable for increasing $\xi$ and thus repelling for decreasing $\xi$. This contradicts the requirement $\theta(\xi)\to 0$ as $\xi\to-\infty$. Together with the fact that $Z^0_-$ has a one-dimensional unstable manifold uniqueness follows. Analogously in case $\Omega>\frac{\b^+}{\alpha}-\frac{s^2}{2}$.
\end{proof}

\begin{figure}[ht]
    \centering
    \begin{subfigure}[b]{0.25\textwidth}
        \centering
        \includegraphics[width=\textwidth]{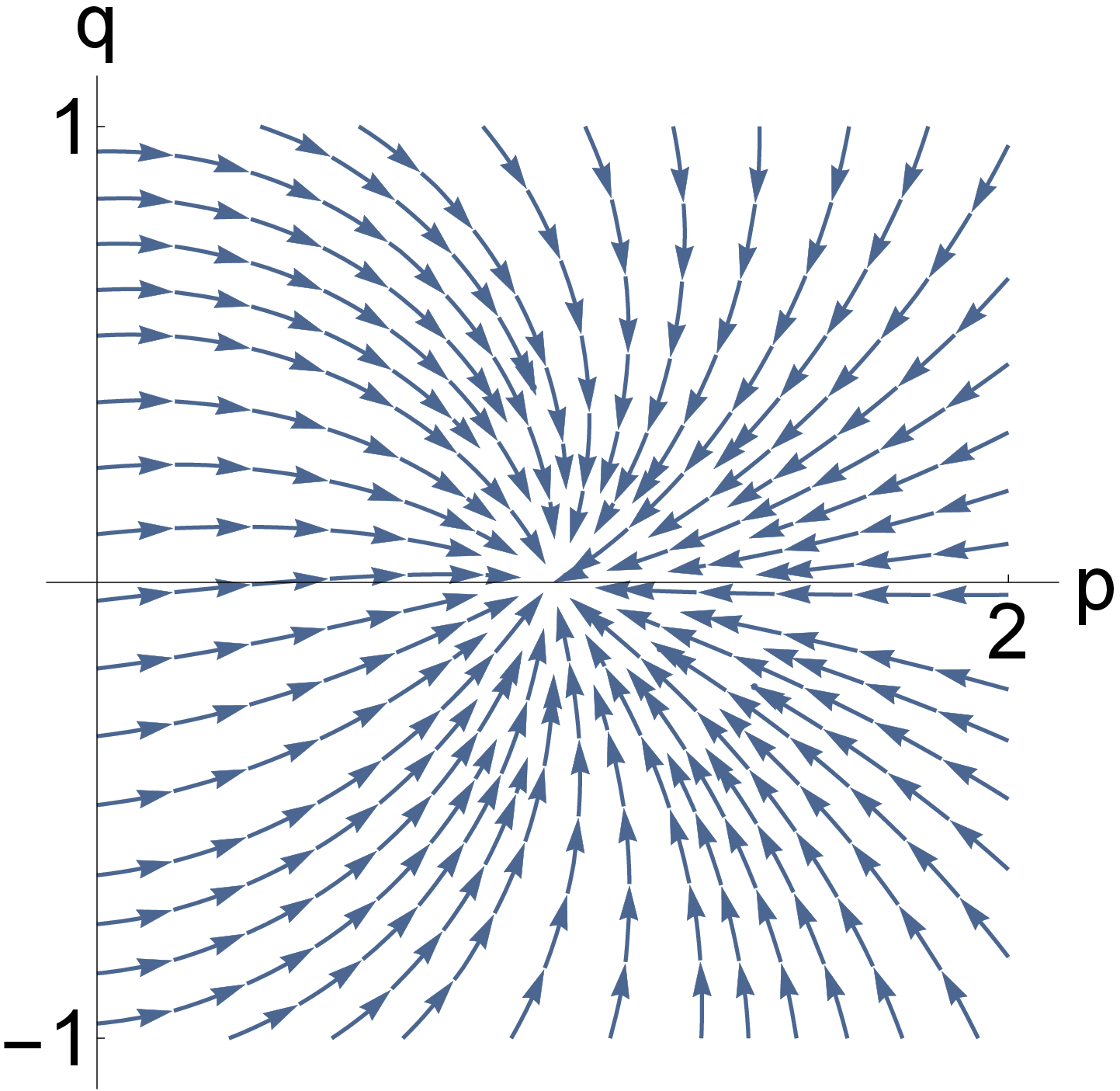}
        \caption{}
    \end{subfigure}
    \hfill
    \begin{subfigure}[b]{0.25\textwidth}
        \centering 
        \includegraphics[width=\textwidth]{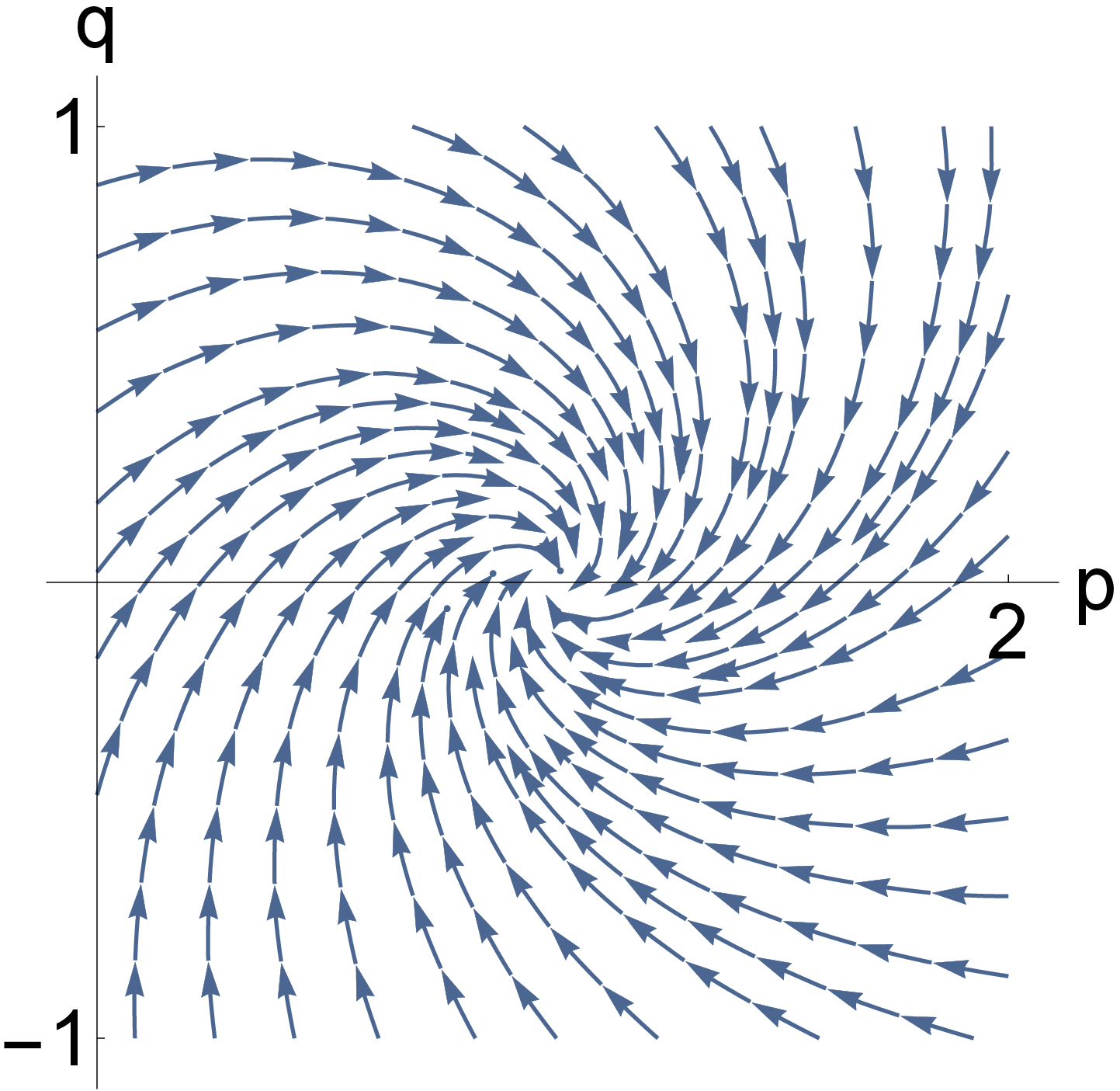}
        \caption{}
    \end{subfigure}
    \hfill
    \begin{subfigure}[b]{0.25\textwidth}
        \centering 
        \includegraphics[width=\textwidth]{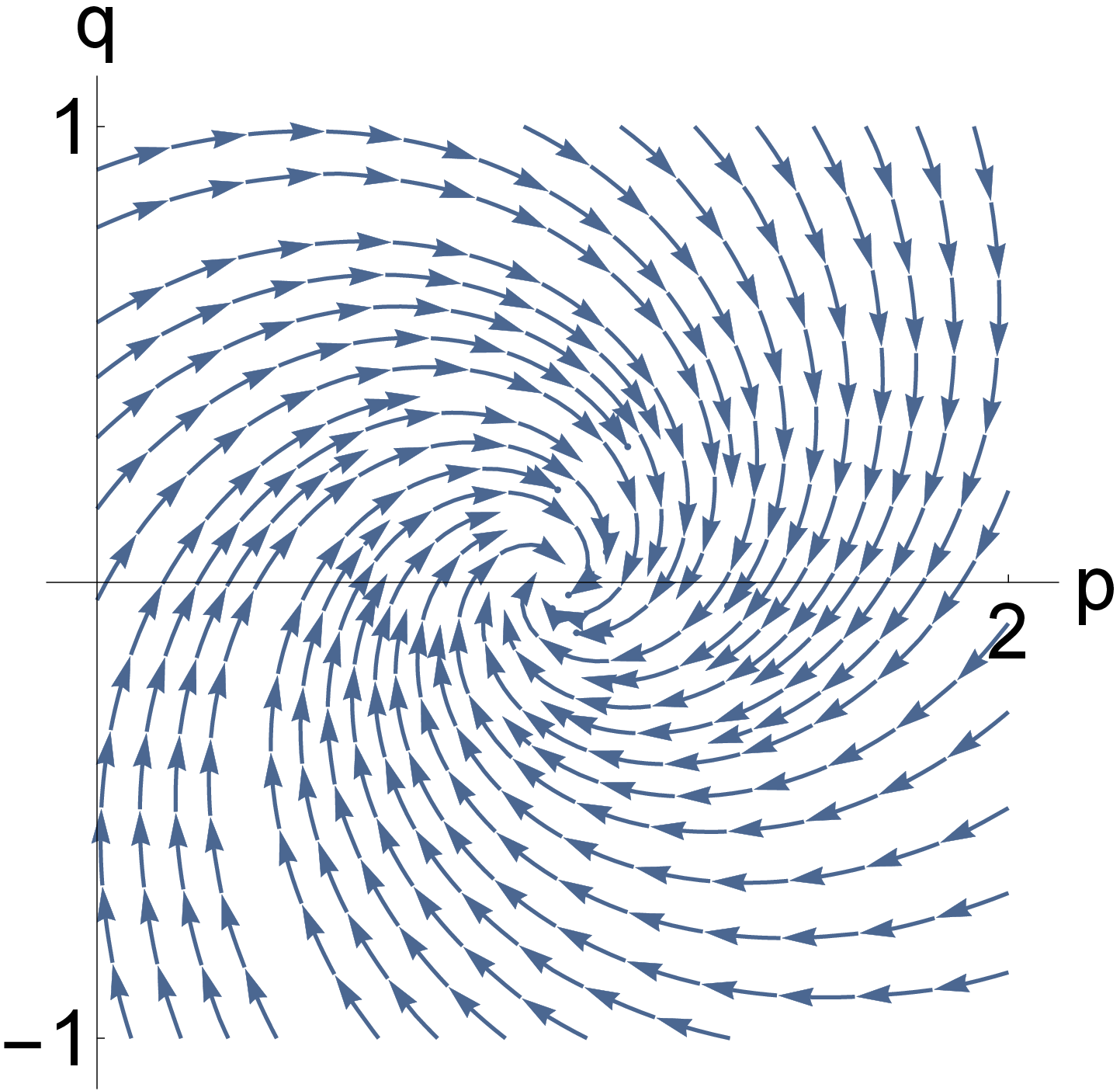}
        \caption{}
    \end{subfigure}

    \begin{subfigure}[b]{0.25\textwidth}
    \centering 
        \includegraphics[width=\textwidth]{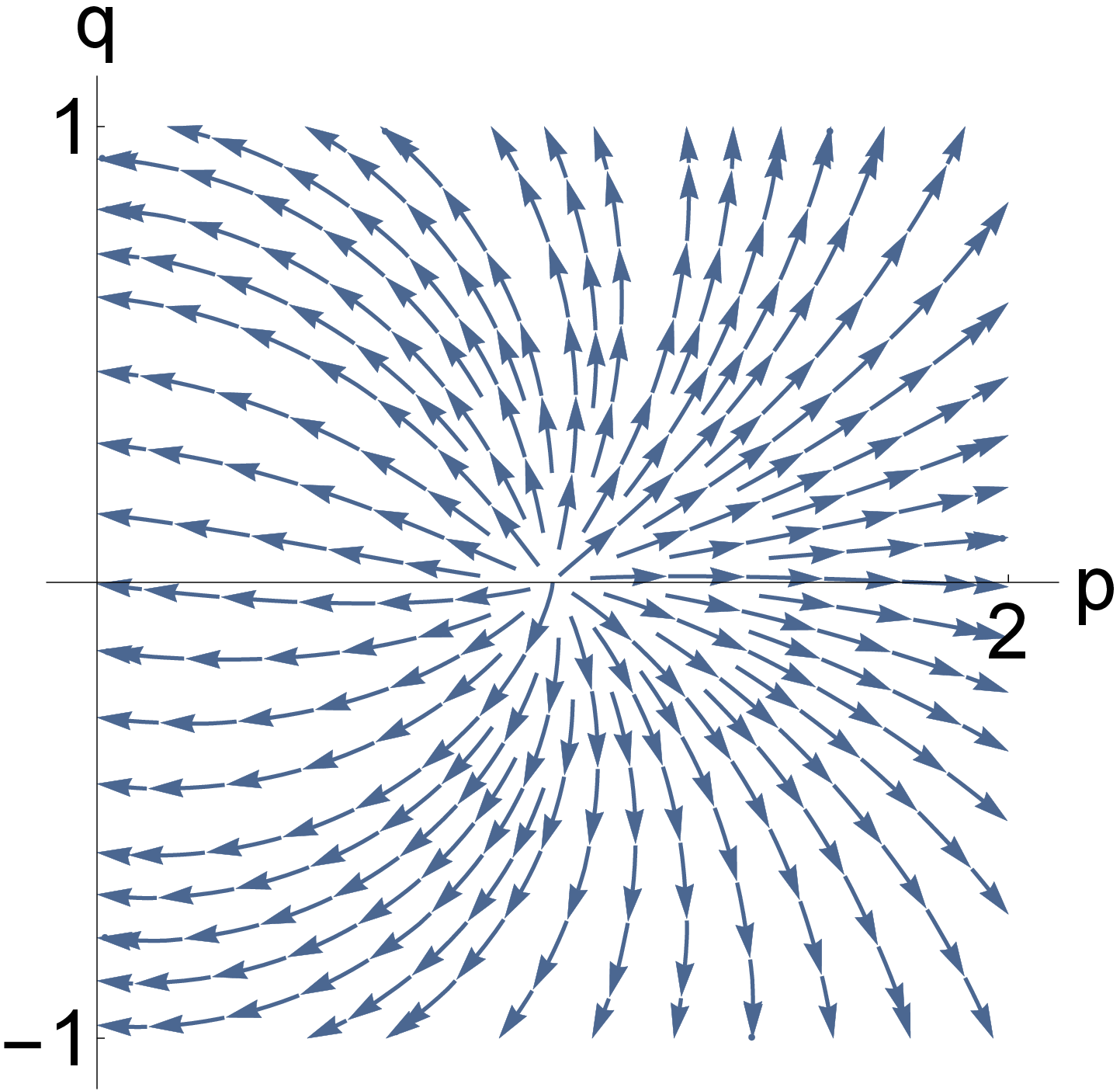}
        \caption{}
    \end{subfigure}
    \hfill
    \begin{subfigure}[b]{0.25\textwidth}
        \centering 
        \includegraphics[width=\textwidth]{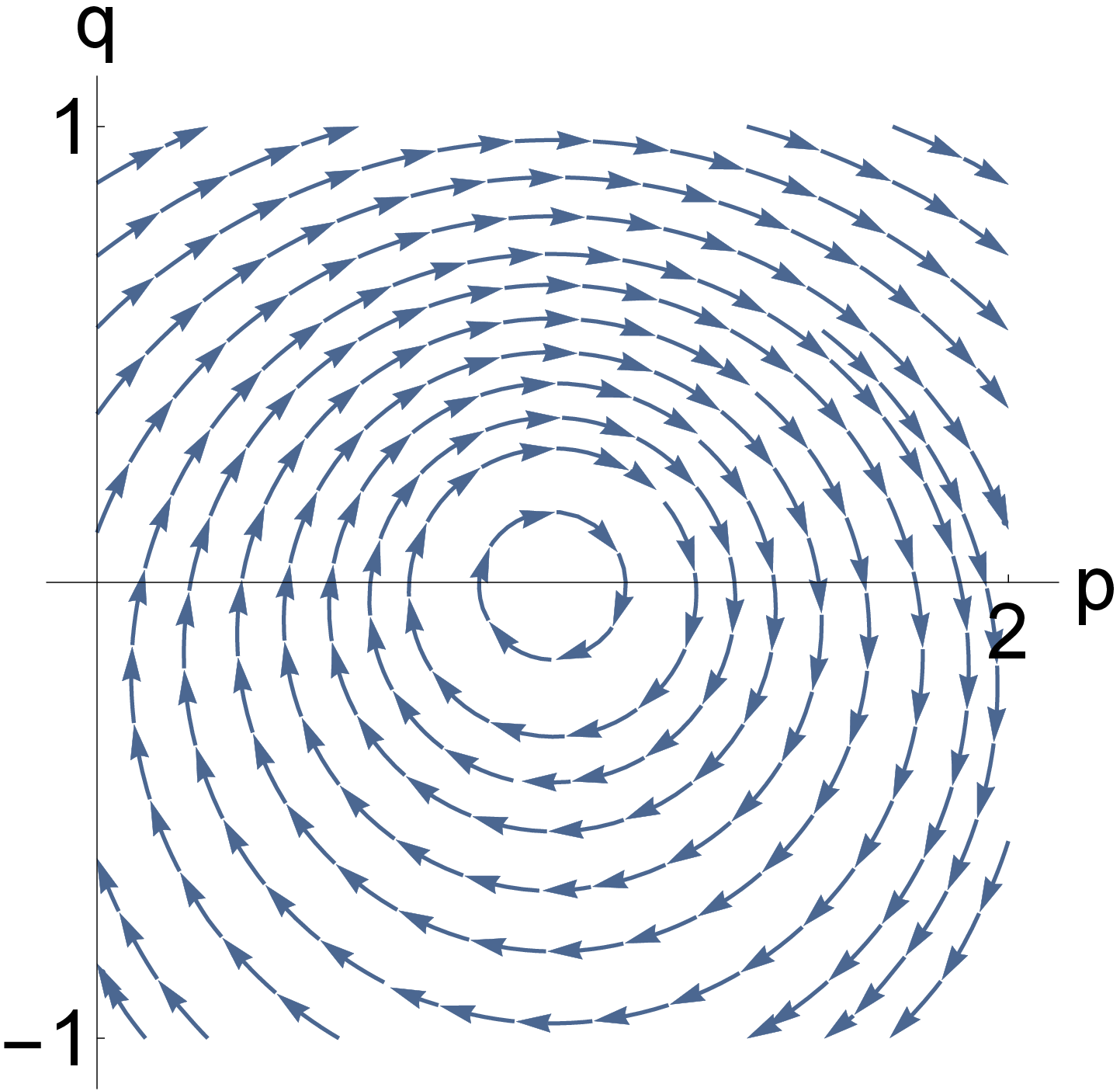}
        \caption{}
    \end{subfigure}
    \hfill
    \begin{subfigure}[b]{0.25\textwidth}
        \centering 
        \includegraphics[width=\textwidth]{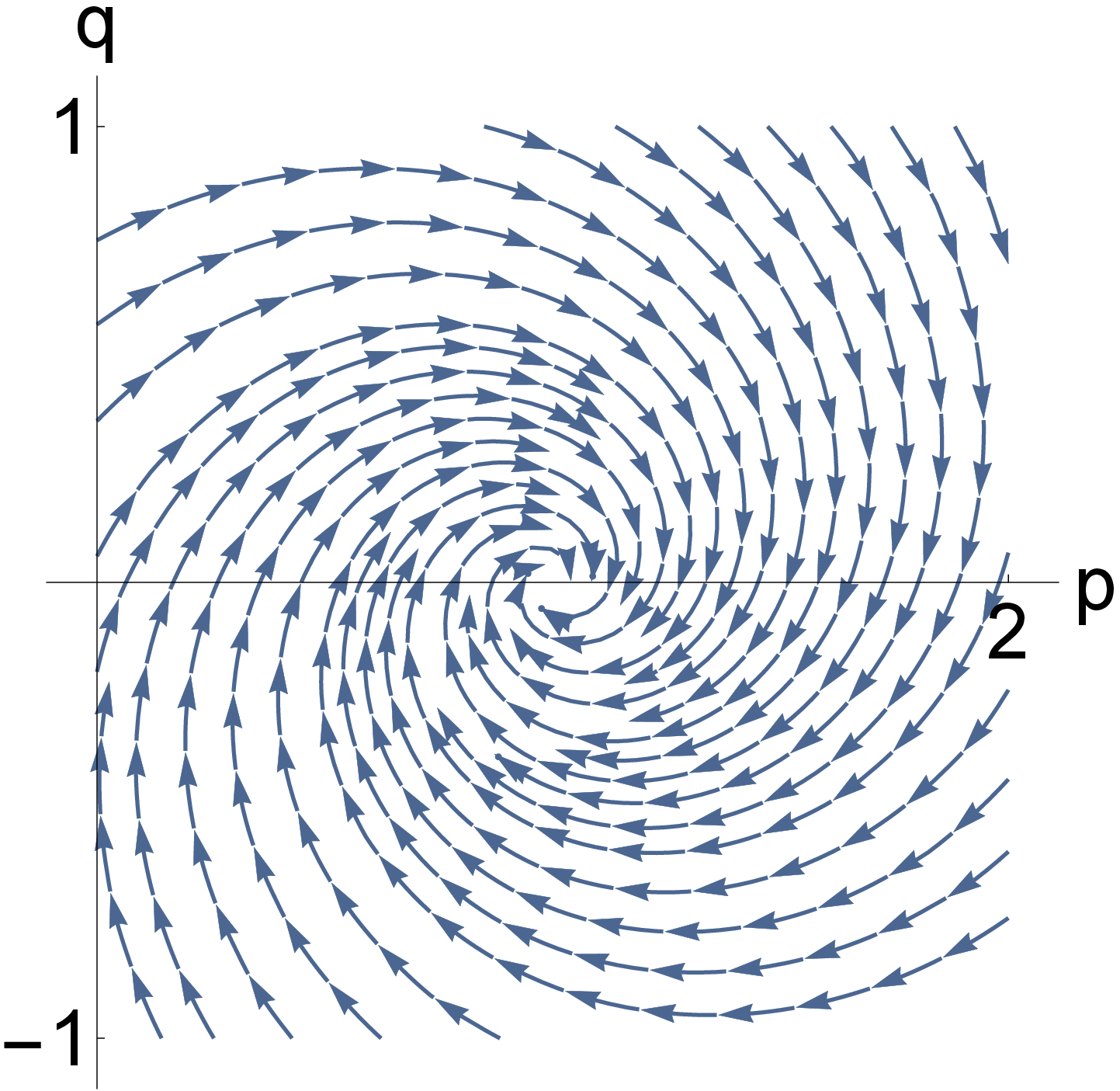}
        \caption{}
    \end{subfigure}
    \caption{\small Phase plane streamplots (with \textsc{Mathematica}) in blow-up charts  near the equilibrium $z_-^\pi=(1,0)$ for $\alpha=0.5\,,\beta=0.1\,,\mu=-1\,,\cc=0$, i.e., the second and third equation of~\eqref{eq:coherent2}. (a-c) $\theta=0$ and (d-f) $\theta=\pi$. (a,d) codim-2 regime, (b,e) center case, where $\Omega=\beta/\alpha+s^2/2$ holds on the chart $\theta=\pi$, and (c,f) codim-0 regime. The remaining parameters and equilibria in (a,d): $h=0.5$, $s_0=0.12$, $\Omega_0=0.44$, and $z_+^0=-1.06-0.12\rmi$, $z_+^\pi=-0.94+0.12\rmi$. In (b,e): $h=10.2$, $s_0=4$, $\Omega_0=8.2$, and $z_+^0=-3-4\rmi$, $z_+^\pi=1+4\rmi$. In (c,f): $h=50.0$, $s_0=19.92$, $\Omega_0=40.04$, and $z_+^0=-10.96-19.92\rmi$, $z_+^\pi=8.96+19.92\rmi$.}
    \label{fig:streamplot}
\end{figure}

Domain walls are heteroclinic orbits between the blow-up charts and decisive for their bifurcation structure are the dimensions (and directions) of un/stable manifolds of the equilibria on these charts. Hence, we next discuss the equilibria $Z^0_\pm$ and $Z^\pi_\pm$ and their stability.

Transverse to the blow-up charts in $\theta$-direction we readily compute the linearization $\partial_\theta(\sin(\theta)p)=\cos(\theta)p$, i.e., the transverse eigenvalue is $-A^{\theta} \cdot \Re(z^{\theta}_\pm)$ at $\theta=0$ and $\pi$, respectively. The eigenvalues within the blow-up charts are determined by $\pm \rmi \gamma$. With $\sigma=\pm$, respectively, the eigenvalues for $Z^0_\sigma$ are
\begin{equation}\label{eq:eig1}
\nu^0_{1,\sigma}=-\sigma\rmi \gamma^0 \,, \quad \nu^0_{2,\sigma}=\overline{\nu^0_{1,\sigma}}\,, \quad \nu^0_{3,\sigma}=\Re(z^0_\sigma)\end{equation} 
and for $Z^\pi_\sigma$  
\begin{equation}\label{eq:eig4}
\nu^\pi_{1,\sigma}=\sigma\rmi\gamma^\pi\,,\quad \nu^\pi_{2,\sigma}=\overline{\nu^\pi_{1,\sigma}}\,, \quad \nu^\pi_{3,\sigma}=-\Re (z^\pi_\sigma).
\end{equation}

Therefore, the signs of the real parts within each blow-up chart are opposite at $Z^{\pi}_+$ compared to $Z^{\pi}_-$ and determined by the sign of $\Re\left(\nu^{0,\pi}_{1,+}\right)$. Hence, within the blow-up charts each equilibrium is either two-dimensionally stable, unstable or a linearly neutral center point.

\medskip
For completeness, we next note that the equilibria on both blow-up charts can be neutral centers simultaneously (cf.\ Figure~\ref{fig:stability diagram}). However, this requires a negative spin polarization and a small Gilbert damping factor, and is not further studied in this paper.

\begin{Remark}\label{rem:cycle-to-cycle}
The equilibria of both blow-up charts are centers simultaneously, if and only if $\Im(\pm\gamma^{0,\pi})=0$ and $\gamma^{0,\pi}\neq0$ (compare Lemma~\ref{lem:chartconnections}). For example if $\alpha=0.5,\beta=0.1,\mu=-1,\cc=-0.99,h=10$
$$s^2=\frac{3960}{199}\,, \quad \Omega=\frac{\beta/\alpha}{1-\cc}+\frac{s^2}{2}=\frac{2000}{199}\,,$$
we obtain
$$\gamma^0=3.33551\,, \quad \gamma^\pi=3.27469\,.$$
\end{Remark}

\section{Domain Walls}\label{sec:domain walls}
All domain walls between $\pm\3$ that we are aware of are of coherent structure type, and thus in one-to-one correspondence to heteroclinic connections between the blow-up charts $\{\theta=0\}$ and $\{\theta=\pi\}$  in~\eqref{eq:coherent2}. Typically we expect these to be heteroclinics between equilibria within the charts, but this is not necessary. Based on the previous analysis, there are three options for heteroclinics between the charts: point-to-point, point-to-cycle, and cycle-to-cycle. We study the first two in this section, for which Proposition~\ref{prop:equilibriumconnection} implies uniqueness of the DW (up to translations/rotations) for a given set of parameters. The third case can occur at most in a relatively small set of parameters (see Remark~\ref{rem:cycle-to-cycle}). Its analysis is beyond the scope of this paper.

\medskip
Note that in case of an existing connection between an equilibrium and a periodic orbit (see Proposition~\ref{prop:hamilton_general}), the domain wall is automatically an inhomogeneous non-flat one. Moreover, via the singular coordinate change any such heteroclinic solution is heteroclinic between $\theta=0,\pi$ in \eqref{eq:coherent1} and through the spherical coordinates it is a heteroclinic connection between $\pm\3$ in the sphere, possibly with unbounded $\varphi$.

\subsection{Homogeneous Domain Walls}\label{sec:homogeneous DW}

It is known from~\cite{goussev2010domain} in case $\beta=0$ and from~\cite{Melcher2017} in case $\cc=0$ (and arbitrary $\beta$) that~\eqref{eq:coherent2} admits for $\mu<0$ a family of explicit homogeneous DWs $\m_0$ given by
\begin{equation}\label{eq:atan}
\begin{pmatrix}\theta_0\\p_0\\q_0\end{pmatrix}=\begin{pmatrix}2\arctan\left( \rme^{\sigma\sqrt{-\mu}\xi} \right)\\\sigma\sqrt{-\mu}\\0\end{pmatrix}
\end{equation}
and parameterized by $\Omega=\frac{h+\alpha\beta}{1+\alpha^2}$, $s^2 = -\frac{(\beta-\alpha h)^2}{\mu(1+\alpha^2)^2}>0$, and $\sigma=1$ for positive speed $s$ and $\sigma=-1$ for negative $s$; the family extends to $s=0$ in the limit $h\to \frac{\beta}{\alpha}$ with scaling of the frequency by $\Omega=\frac{\beta}{\alpha}+\frac{\smu}{\alpha}s$. For $s=0$ (standing) fronts with both orientations exist simultaneously  $\cc=0$ and are given by
\begin{equation*}
\begin{pmatrix}\theta_0\\p_0\\q_0\end{pmatrix}=\begin{pmatrix}2\arctan\left( \rme^{\pm\sqrt{-\mu}\xi} \right)\\\pm\sqrt{-\mu}\\0\end{pmatrix}.
\end{equation*}

Hence, the branches of left and right moving walls as parametrized by $s$ each have termination point at $s=0$ (cf.\ Figure~\ref{fig:existencecont}).

\medskip
The family of explicit DWs~\eqref{eq:atan} have domain wall width $\smu$, a profile independent of the applied field $h$ and propagate along a nanowire ($\mu<0$) with velocity $s$ while precessing with azimuthal velocity $\Omega$. Since these are unique up to spatial reflection symmetry, the direction of motion is related to the spatial direction of connecting $\pm\3$ through $\sigma$,
\begin{equation}\label{eq:moving wall}
\begin{aligned}
    \theta(-\infty)&=0 & \theta(+\infty)&=\pi & \Leftrightarrow s&>0 \textnormal{ (wall moves to the right)}\\ \theta(-\infty)&=\pi & \theta(+\infty)&=0 & \Leftrightarrow s&<0 \textnormal{ (wall moves to the left)}.
\end{aligned}
\end{equation}
To simplify some notation we will focus on the case of right-moving walls including standing walls ($s\ge0$) and thus make the standing assumptions that $h\ge\beta/\alpha$ as well as $\mu<0$. We therefore have a 1-to-1 relation of parameters $(\alpha,\beta,h,\mu)$ and right-moving DWs from 
$$\m(\xi,t) = \m_0(\xi,t;\alpha,\beta,h,\mu)$$
with speed and frequency given by 
\begin{equation}\label{eq:initialspeedfrequenzy}
s_0=s_0(\alpha,\beta,h,\mu):=\frac{\alpha h-\beta}{\sqrt{-\mu}(1+\alpha^2)}\;,\quad \Omega_0=\Omega_0(\alpha,\beta,h,\mu):=\frac{h+\alpha\beta}{1+\alpha^2}
\end{equation}
where the subindex $0$ emphasizes that $\cc=0$. Since $s_0$ is surjective on $\R_{\geq 0}$ any velocity can be realised. Spatial reflection covers the case $h\le\beta/\alpha$.

\medskip
Based on Lemma~\ref{lem:explicit} as well as Remark~\ref{rem:equilibria} for $\cc=0$ and (homogeneous) speed and frequency~\eqref{eq:initialspeedfrequenzy}, one readily finds the asymptotic states of~\eqref{eq:coherent2} given by

$$E^0\coloneqq Z^0_-\big\vert_{(s_0,\Omega_0)}=\left(0,\sqrt{-\mu},0\right)\tran \quad \textnormal{and} \quad E^\pi\coloneqq Z^\pi_-\big\vert_{(s_0,\Omega_0)}=\left(\pi,\sqrt{-\mu},0\right)\tran,$$
with (spatial) eigenvalues~\eqref{eq:eig1}, \eqref{eq:eig4} given by
\begin{equation}\label{eq:eigenvalues}
\begin{split}
\nu_{k,-}^0 &:= -\alpha s_0 - 2\smu -(-1)^k\, \rmi s_0 \,, \quad \nu^0_{3,-}= \smu\,, \\
\nu_{k,-}^\pi &:= -\alpha s_0 + 2\smu -(-1)^k\, \rmi s_0 \,, \quad \nu^\pi_{3,-}=-\smu\,,
\end{split}
\end{equation}
where $k=1,2$. Note that the above equilibria cannot be centers simultaneously (recall $\mu<0$), hence a cycle-to-cycle connection can not exist close to it (see Remark~\ref{rem:cycle-to-cycle} for details). For this reason, we focus on point-to-point as well as point-to-cycle connections.

\subsection{Inhomogeneous Domain Walls}\label{sec:inhomogeneous DW}
Homogeneous DWs exist only in case $\cc=0$~\cite[Theorem 5]{Melcher2017}, are explicitly given by~\eqref{eq:atan} and completely characterized by~\eqref{eq:initialspeedfrequenzy}. By~\cite[Theorem 6]{Melcher2017}, fast inhomogeneous DW solutions with $|s|\gg1$ exist for any $\cc \in(-1,1)$, but in contrast to \eqref{eq:atan}, the gradient of these profiles is of order $1/|s|$ and thus have a large `width'. The natural question arises what happens for any $s$ in case $\cc\neq 0$.

\medskip
This section contains the main results of this paper: the existence, parameter selection and structure of inhomogeneous DW solutions in case of small $|\cc|$ for any value of the applies field $h$, and thus any speed $s$. This will be achieved by perturbing away from the explicit solution $\m_0$ given by~\eqref{eq:atan}, where the bifurcation structure is largely determined by comparing the dimensions of the un/stable eigenspaces at the asymptotic equilibrium states, which are determined by~\eqref{eq:eigenvalues}. 

\medskip
Let $W^0_{s/u}$ and $W^\pi_{s/u}$ denote the stable and unstable manifolds associated to $E^0$ and $E^\pi$, respectively, and $w_{s/u}^0$ as well as $w_{s/u}^\pi$ be the dimension of these manifolds so that $w^0_s+w^0_u=w^\pi_s+w^\pi_u=3$. Notably $w^0_s=2$ and $w^0_u=1$ for all values of the parameters, and $w^\pi_s$ is either $1$ or $3$. Recall the standing assumption $s_0\ge0$. 

If $w^\pi_s=1$, the heteroclinic connection of $E^0$ and $E^\pi$ generically has \emph{codimen\-sion-2}, while for $w^\pi_s=3$ it has \emph{codimen\-sion-0}, and we refer to the transition point between these cases, following the discussion in \S~\ref{sec:blow-up}, as the \emph{center case}. From~\eqref{eq:eigenvalues} we have $w^\pi_s=1 \Leftrightarrow 0\le s_0<\frac{2\smu}{\alpha}$, $w^\pi_s=3 \Leftrightarrow s_0>\frac{2\smu}{\alpha}$ and the center case at $s_0=\frac{2\smu}{\alpha}$.

Hence, within the family of homogeneous DWs given by \eqref{eq:atan} and satisfying \eqref{eq:initialspeedfrequenzy}, the different bifurcation cases have  speed and frequency relations
\begin{align*}
    \textnormal{codim-0: } & s_0>\frac{2\smu}{\alpha} \quad \textnormal{and} \quad \Omega_0>\frac{\beta}{\alpha}-\frac{2\mu}{\alpha^2}\,,\\
    \textnormal{center: } & s_0=\frac{2\smu}{\alpha} \quad \textnormal{and} \quad  \Omega_0=\frac{\beta}{\alpha}-\frac{2\mu}{\alpha^2}\,,\\ 
    \textnormal{codim-2: } & 0 \le s_0<\frac{2\smu}{\alpha}\quad \textnormal{and} \quad \frac{\beta}{\alpha}\le\Omega_0<\frac{\beta}{\alpha}-\frac{2\mu}{\alpha^2}.
\end{align*}
Using~\eqref{eq:initialspeedfrequenzy} these can be written in terms of the parameters of~\eqref{LLGS}, which gives the characterization mentioned in the introduction \S 1.

\begin{Remark}\label{rem:stab}
The case distinction is also related to the spectral stability of the asymptotic states $\m=\pm\3$ in the dynamics of the full PDE~\eqref{LLGS} which is beyond the scope of this paper, but see Figure~\ref{fig:stability diagram} for an illustration. In short, it follows from, e.g.,~\cite[Lemma 1]{Melcher2017} that $\3$ is $L^2$-stable for $h>\beta/\alpha$, while $-\3$ is $L^2$-stable for $h<\beta/\alpha-\mu$ and unstable for $h>\beta/\alpha-\mu$. Based on this, the stability curves in Figure~\ref{fig:stability diagram} are defined as follows
$$\Gamma^+\coloneqq \frac{\beta/\alpha}{h-\mu}-1, \qquad \Gamma^- \coloneqq 1-\frac{\beta/\alpha}{h+\mu},$$
which intersect at
$$h=\frac{\beta}{2\alpha}+\sqrt{\frac{\beta^2}{4\alpha^2}+\mu^2}.$$
Since the destabilisation of $-\3$ if $\beta>0$ corresponds to a Hopf-instability of the $($purely essential$)$ spectrum, it is effectively invisible in the coherent structure ODE, which detects changes in the linearization at zero temporal eigenvalue only. Visible from the PDE stability viewpoint is a transition of absolute spectrum through the origin in the complex plane of temporal eigenmodes, cf.~\cite{sandstede2000absolute}. Now in the center case, the state $-\3$ is already $L^2$-unstable since $\alpha>0$ as well as $\mu<0$ and $h>\beta/\alpha$ implies
$$h=h^\ast=\frac{\beta}{\alpha}-\frac{2\mu}{\alpha^2}(1+\alpha^2)>\frac{\beta}{\alpha}-\mu$$
and therefore that $\Gamma^-$ never intersects the line $\cc\equiv 0$ at $h=h^\ast$.

\medskip Moreover, it was shown in~\cite{gou2011stability} that the family of explicit homogeneous DWs~\eqref{eq:explicitsol} is ($linearly$) stable for sufficiently small applied fields, actually for $h<-\mu/2$, in case $\beta=0$, hence in the bi-stable case where $\pm\3$ are $L^2$-stable. As mentioned before, $\beta=0$ is equivalent to $\cc=0$ in the~\ref{LLGS} equation with an additional shift in $h$ and $\beta$, which leads to the~\eqref{LLG} case. We expect these DWs are also stable for small perturbations in $\cc$, due to the properties of the operator established in~\cite{gou2011stability}, but further analysis also on the transition from convective/transient to absolute instability will be done elsewhere.
\end{Remark}

With these preparations, we next state the main results, which concern existence of DWs in the three regimes.

\begin{Theorem}\label{theo:codim0}
For any parameter set $(\alpha_0, \beta_0, h_0, \mu_0)$ in the codim-0 case, i.e., $\mu_0<0$ and $h_0>\beta_0/\alpha_0-2\mu_0-2\mu_0/\alpha_0^2$, the following holds. The explicit homogeneous DWs $\m_0$ in~\eqref{eq:explicitsol} lies in a smooth family $\m_{\cc}$ of DWs parameterized by $(\cc, \alpha, \beta, h, \mu, s, \Omega)$ near $(0, \alpha_0, \beta_0, h_0, \mu_0, s_0, \Omega_0)$ with $s_0$, $\Omega_0$ from~\eqref{eq:initialspeedfrequenzy} evaluated at $(\alpha_0, \beta_0, h_0, \mu_0)$. Moreover, in case $\cc=0$ and $(s, \Omega)\neq (s_0, \Omega_0)$ evaluated at $(\alpha, \beta, h, \mu)$, or $\cc\neq 0$, these are inhomogeneous flat DWs.
\end{Theorem}

\begin{proof}
As mentioned, in the codim-0 case we have $w_s^\pi=3$ and for all parameters $w_u^0=1$. Due to the existence of the heteroclinic orbit~\eqref{eq:atan}, this means $W^0_u$ intersects $W^\pi_s$ transversely and non-trivially for $\cc=0$ in a unique trajectory. Therefore, this DW perturbs to a locally unique family by the implicit function theorem for perturbations of the parameters in~\eqref{eq:coherent2}. 

For $\cc\neq 0$ sufficiently small these are inhomogeneous DWs since the derivative of the third equation, the $q$-equation, in~\eqref{eq:coherent2} with respect to $\cc$ is nonzero in this case; hence already the equilibrium states move into the inhomogeneous regime. 

For $\cc=0$ but $(s, \Omega)\neq (s_0, \Omega_0)$ at $(\alpha, \beta, h, \mu)$, it follows from~\cite[Theorem 5]{Melcher2017} that these DWs cannot be homogeneous.
\end{proof}

Next we consider the center case, where $h=h^\ast=\beta/\alpha-2\mu-2\mu/\alpha^2$. We start with a result that follows from the same approach used in the codim-2 case and give refined results below.

\begin{Corollary}\label{cor:transition}
The statement of Theorem~\ref{theo:codim0} also holds for a parameter set in the center case if the perturbed parameters $(\cc, \alpha, \beta, h, \mu, s, \Omega)$ satisfy $\Omega> s^2/2+\b^-/\alpha$. If $\Omega=s^2/2+\b^-/\alpha$ and $\Omega<h+\mu+\frac{s^2}{4}(1+\alpha^2)$ the same holds except the DW is possibly non-flat.
\end{Corollary}

\begin{proof}
It follows from Proposition~\ref{prop:hamilton_general} and the discussion before that $\Omega> s^2/2+\b^-/\alpha$ for the parameter perturbation implies that the eigenvalues of the perturbed equilibrium $Z^\pi_-\approx E^\pi$ satisfy $\Re(\nu_{k,-}^\pi)< 0$, $k=1,2$. Hence, the stable manifold at the target equilibrium is two-dimensional and lies in a smooth family with the center-stable manifold at the transition point. Then the proof is the same as in the codim-0 case. If the perturbation has $\Omega=s^2/2+\b^-/\alpha$ then we consider as target manifold the three-dimensional stable manifold of a neighborhood of $Z^\pi_-$ within the blow-up chart $\theta=0$. This neighborhood consists of periodic orbits by Proposition~\ref{prop:hamilton_general} if $\Omega<h+\mu+\frac{s^2}{4}(1+\alpha^2)$. By dimensionality the intersection with the unstable manifold of $Z^0_-$ persists and yields a heteroclinic orbit from the perturbed equilibrium at $\theta=0$ to the blow-up chart at $\theta=\pi$. Perturbing $\cc$ away from zero moves the left-asymptotic state into the inhomogeneous regime and thus generates an inhomogeneous DWs. Note from Proposition~\ref{prop:hamilton_general} that the right-asymptotic state is either an equilibrium with $q\neq 0$ or a periodic orbit along which $q=0$ happens at most at two points.
\end{proof}

Next we present a refined result in which we show that typical perturbations indeed give non-flat DWs, i.e., heteroclinic connections with right-asymptotic state being a periodic orbit. The existence of flat DWs for $\cc\neq 0$ is severely constrained, but not ruled out by this result. Our numerical results, such as those presented in \S\ref{sec:continuation}, always lead to a selected solution with a periodic asymptotic state.

In addition, attempts to perform numerical continuation (see \S\ref{sec:continuation}) of flat DWs to $\cc\neq 0$ failed. Here we added the constraint $\tH=0$ and allowed adjustment the parameters $h$ and $s$, but the continuation process did not converge, which confirms numerically the generic selection of a periodic orbit. 

\medskip
As mentioned before, the right asymptotic state is $\3$ in either case in the PDE coordinates; the difference between flat and non-flat lies in the finer details of how the profile approaches $\3$ in term of $p$ and also $q$, which relates to $\m$ through~\eqref{eq:q}.

\begin{Theorem}\label{theo:center}
Consider the smooth family of DWs from Corollary~\ref{cor:transition} with $\cc=0$ for parameters satisfying~\eqref{eq:center case} with fixed $\alpha>0$, $\beta\ge0$, and $\mu<0$. Then there is a neighborhood $(\cc, s,h)$ of $(0,s_0,h^*)$ such that the following holds. Flat DWs occur at most on a surface in the $(\cc,s, h)$-parameter space and, for $\beta\neq 0$, satisfy $|h-h_0|^2+|s-s_0|^2=\calO(|\cc|^3)$, more precisely~\eqref{eq:second order approx} holds, where $h_0=h^\ast$ and $s_0=2\smu/\alpha$. Otherwise DWs are non-flat, in particular all DWs not equal to $\m_0$ for $\cc=0$ or $\beta=0$ are non-flat.
\end{Theorem}

Due to its more technical nature, the proof of this theorem is deferred to Appendix~\ref{app:center case proof}.

\medskip
It remains to consider the codim-2 case.

\begin{Theorem}\label{theo:codim2} For any parameter set $(\alpha_0, \beta_0, h_0, \mu_0)$ in the codim-2 case, i.e., $\mu_0<0$ and $\beta_0/\alpha_0\le h_0<\beta_0/\alpha_0-2\mu_0-2\mu_0/\alpha_0^2$, the following holds. The explicit homogeneous DWs $\m_0$ in~\eqref{eq:explicitsol} lies in a smooth family of DWs parameterized by $(\cc ,\alpha, \beta, h, \mu)$ near $(0, \alpha_0, \beta_0, h_0, \mu_0)$. Here the values of $(s, \Omega)$ are functions of the parameters $(\cc, \alpha,\beta, h, \mu)$ and lie in a neighbourhood of $(s_0, \Omega_0)$ from~\eqref{eq:initialspeedfrequenzy}. This family is locally unique near $\m_0$ and for $\cc\neq 0$ consists of inhomogeneous flat DWs.
\end{Theorem}

The proof of Theorem~\ref{theo:codim2} is presented in Appendix~\ref{app:proof} and is based on the Melnikov method for perturbing from $\m_0$. As the unperturbed heteroclinic orbit has codimension two, the bifurcation is studied in a three-parametric family with perturbation parameters $\eta\coloneqq\left(\cc,s,\Omega\right)\tran$, which yields a two-component splitting function $M(\eta)$ that measures the mutual displacement of the manifolds $W^0_u$ and $W^\pi_s$. 

\medskip
Due to the fact that $\Re (\nu^\pi_{j,-})<0$ also for $\beta=0$ in the codim-0 regime and the case $\beta=0$ is included in Theorem~\ref{theo:center} as well as Theorem~\ref{theo:codim2}, we immediately get the following result.

\begin{Corollary}Inhomogeneous flat DWs also exist in the~\ref{LLG} equation $(\beta=0)$, which can be flat or non-flat, respectively.
\end{Corollary}

Theorem~\ref{theo:codim2} completes the existence study of DWs. Therefore, for any value of the applied field $h$ there exists a heteroclinic connection between the blow-up charts with $\cc\neq0$ and $q\not\equiv 0$, thus an inhomogeneous (typically flat) DW. Recall that we have focused on right moving DWs, but all results are also valid for left moving walls due to symmetry. Therefore inhomogeneous DWs exist with $\cc\neq 0$ for any value of the applied field $h\in\R$.

\section{Numerical Results}\label{sec:continuation}
Numerical continuation for ordinary differential equations is an established tool for bifurcation analysis in dynamical system. In this section we present continuation results to illustrate the analytical results discussed in \S~\ref{sec:domain walls}. In particular, we will focus on continuation in the parameter $\cc$ in the range of $(-0.5,0.5)$ as this perturbs away from the known family $\m_0$ from~\eqref{eq:atan} (cf.\ Figure~\ref{fig:initial arctan}) with speed and frequency determined by~\eqref{eq:initialspeedfrequenzy} for a given applied field. Note that we also focus only on right-moving fronts in this section for reasons of clarity. All results were produced by continuation in \textsc{AUTO-07P} and graphics were created with \textsc{Mathematica} as well as \textsc{MATLAB}. 

Heteroclinic orbits were detected as solutions to the boundary value problem given by the desingularized system~\eqref{eq:coherent2} plus a phase condition and boundary conditions at $\xi=-L$ and $\xi=L$ taken from the analytic equilibrium states in $p$ and $q$ on the blow-up charts (Remark~\ref{rem:equilibria}). In the codim-2 case, the four required conditions are the $p, q$ values at the charts. In the center case, the three required conditions are: (1,2) the two $p, q$ values at the left chart and (3) the energy difference determined by the function~\eqref{eq:hamiltonian fct}. In the codim-0 case, the two required conditions are the $p$ values at both charts. Moreover, we found $L=50$ was sufficiently large.

In order to relate to~\eqref{LLGS}, we plot most of the profiles after blowing down to the sphere rather than using the ODE phase space.

\begin{figure}[ht]
    \centering
    \begin{subfigure}[b]{0.25\textwidth}
        \includegraphics[width=\textwidth]{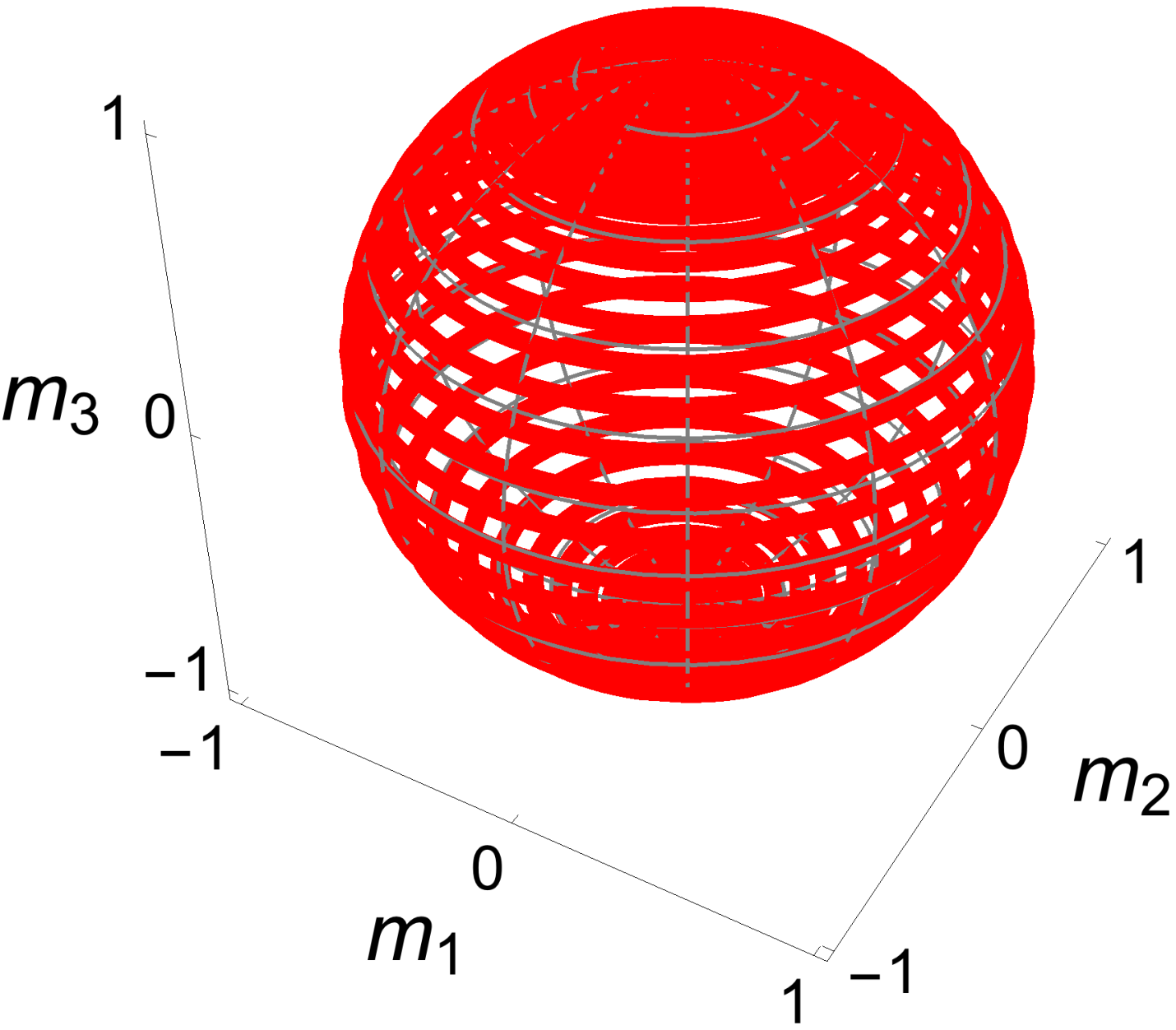}
        \caption{}
    \end{subfigure}
    \hspace*{0.25\textwidth}
    \begin{subfigure}[b]{0.35\textwidth}
        \includegraphics[width=\textwidth]{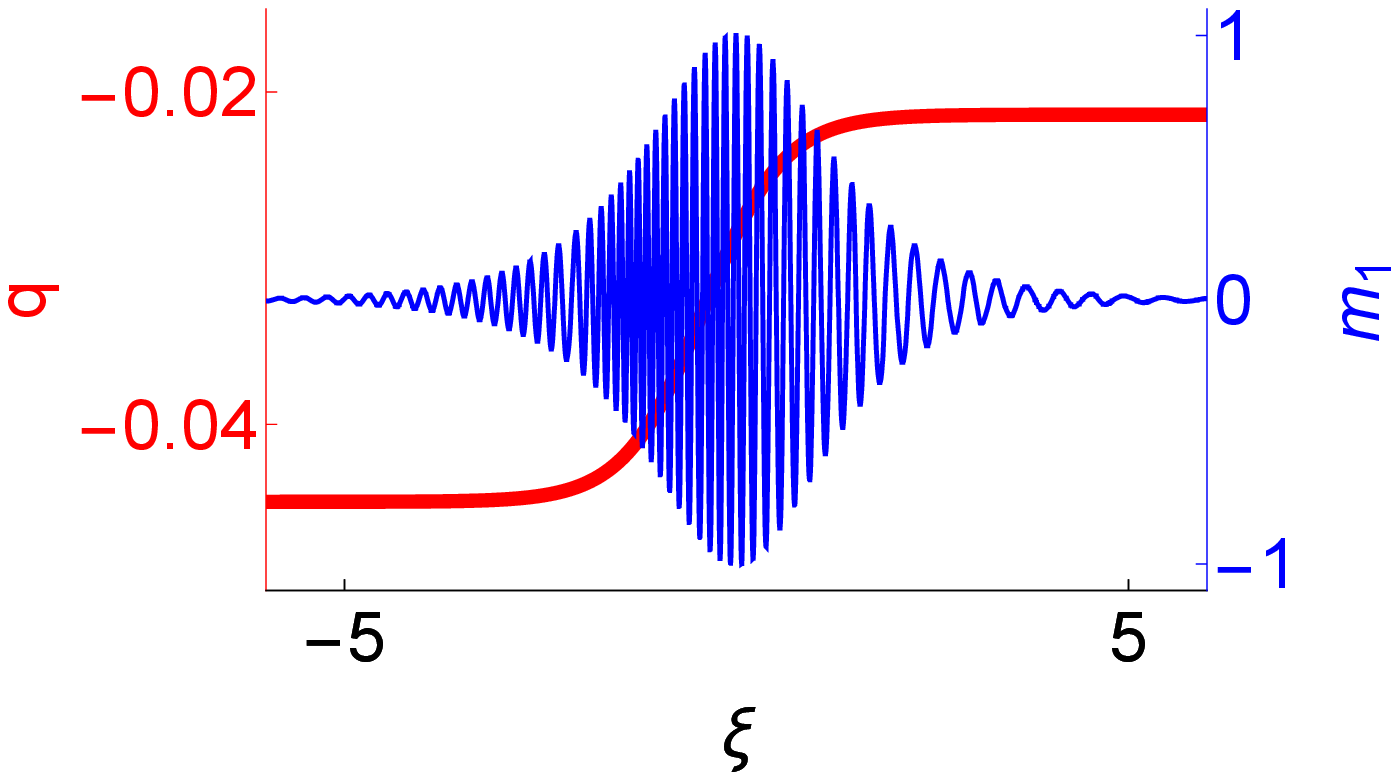}
        \caption{}
    \end{subfigure}
    \caption{\small DWs obtained from continuation of $\m_0$ in system~\eqref{eq:coherent2} in the codim-2 regime $h=0.5$ ($h^\ast=10.2$) with initial speed and frequency $s_0=0.12$ as well as $\Omega_0=0.44$, and $(\cc, s, \Omega)=(-0.5, 0.11221, 0.44077)$. (a) Projection onto the sphere. (b) Zoom-in of corresponding $q$-profile (red) and $m_1$ component (blue).}
    \label{fig:continuation h=0.5}
\end{figure}

\medskip
Following the standing assumption on positive speeds and using $\cc$ as well as $h$ as the main parameters, we keep the other parameters fixed with values
$$\alpha=0.5, \beta=0.1, \mu=-1.$$
The value of the applied field for the center case, given the fixed parameters, is $h^\ast=10.2$ (cf.\ \S~\ref{sec:inhomogeneous DW}), which leads to $s_0=4.0$ as well as $\Omega_0=8.2$ (cf.~\eqref{eq:initialspeedfrequenzy}). 

\subsection{Codim-2 case}
The lower boundary for values of the applied field $h$ lies in the codim-2 regime and is given by $h=\beta/\alpha=0.2$. As a first numerical example we consider the slightly larger value $h=0.5$. The results upon continuation in the negative as well as positive direction of $\cc$ are presented in Figures~\ref{fig:continuation h=0.5 plus},~\ref{fig:continuation h=0.5 plus m1 component zoom-in}, and~\ref{fig:continuation h=0.5}. The inhomogeneous nature of these solutions ($\cc\neq0$) is reflected in the significantly varying azimuthal angles, also visible in the oscillatory nature of the $m_1$ component in Figures~\ref{fig:continuation h=0.5 plus m1 component zoom-in} as well as~\ref{fig:continuation h=0.5}b.

The linear part of the splitting function~\eqref{eq:splitting function} (see Theorem~\ref{theo:codim2}), which predicts the direction of parameter variation for the existence of inhomogeneous DWs ($\cc\neq0$) to leading order, reads in this example
$$M(\cc,s,\Omega)=\begin{pmatrix}
-0.00147567& -0.499245& 0.245945\\ -0.000577908& -0.245945& -0.499245
\end{pmatrix}\cdot \begin{pmatrix}\cc\\s\\\Omega\end{pmatrix}, $$ so that $M=(0, 0)\tran$ for $(s, \Omega)=\left( -0.00283744\cdot \cc, 0.000240252\cdot \cc \right)$. For the parameter values in Figure~\ref{fig:continuation h=0.5} and~\ref{fig:continuation h=0.5 plus} we obtain, respectively, 
$$M(-0.5, -0.007788, 0.000771)=\left( 0.00481558, 0.00181945 \right)\tran,$$
$$M(0.5, -0.007973, 0.007173) =\left(0.00648248, -0.00133122\right)\tran.$$
Note that here the splitting of the (1-dimensional) unstable manifold of the left equilibrium and the (1-dimensional) stable manifold of the right equilibrium differ, i.e., are in opposite directions (signs) in frequency and speed for variations in $\cc$.

\begin{figure}[t]
    \centering
    \begin{subfigure}[b]{0.25\textwidth}
        \includegraphics[width=\textwidth]{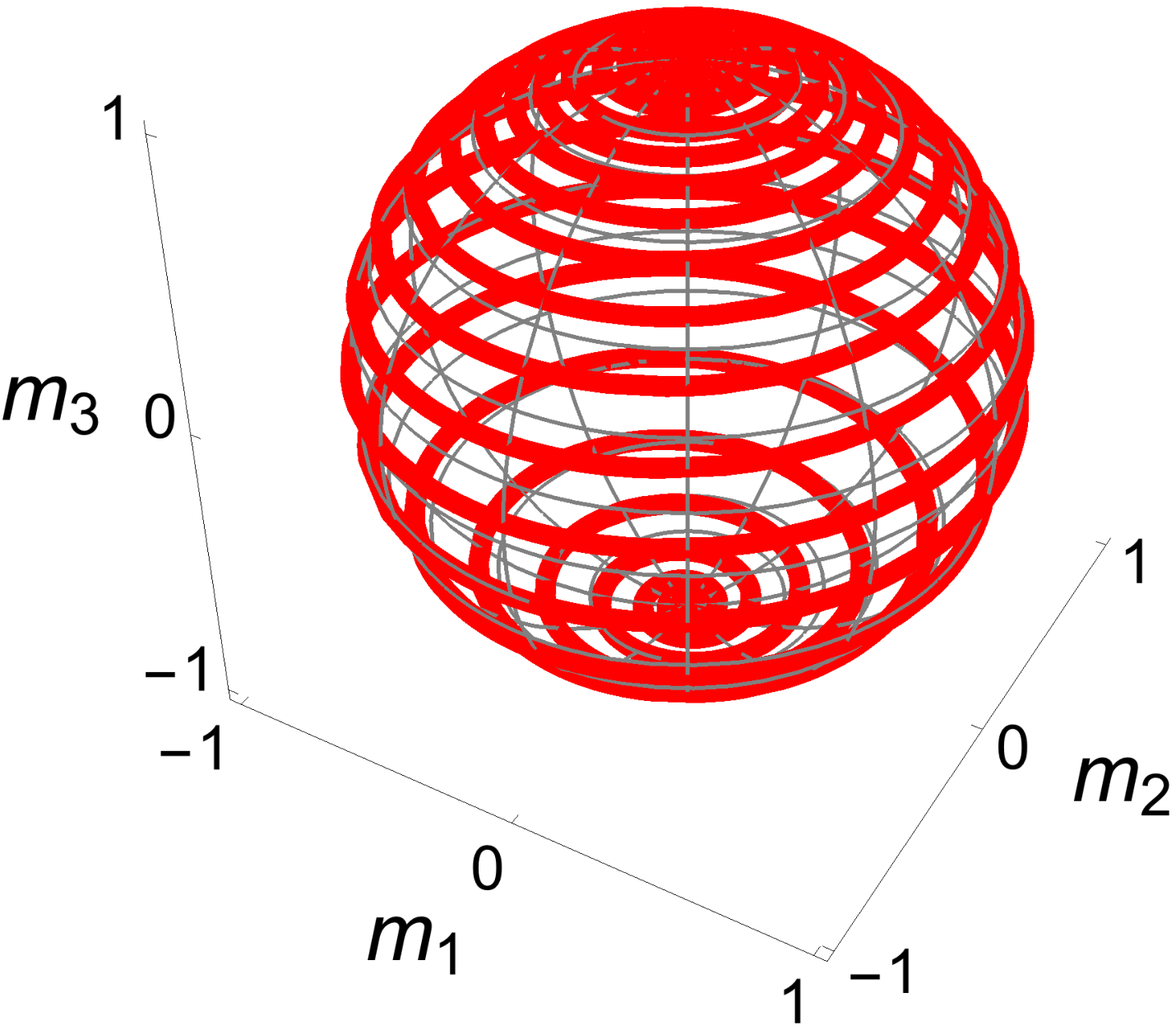}
        \caption{}
    \end{subfigure}
    \hspace{0.25\textwidth}
    \begin{subfigure}[b]{0.25\textwidth}
        \includegraphics[width=\textwidth]{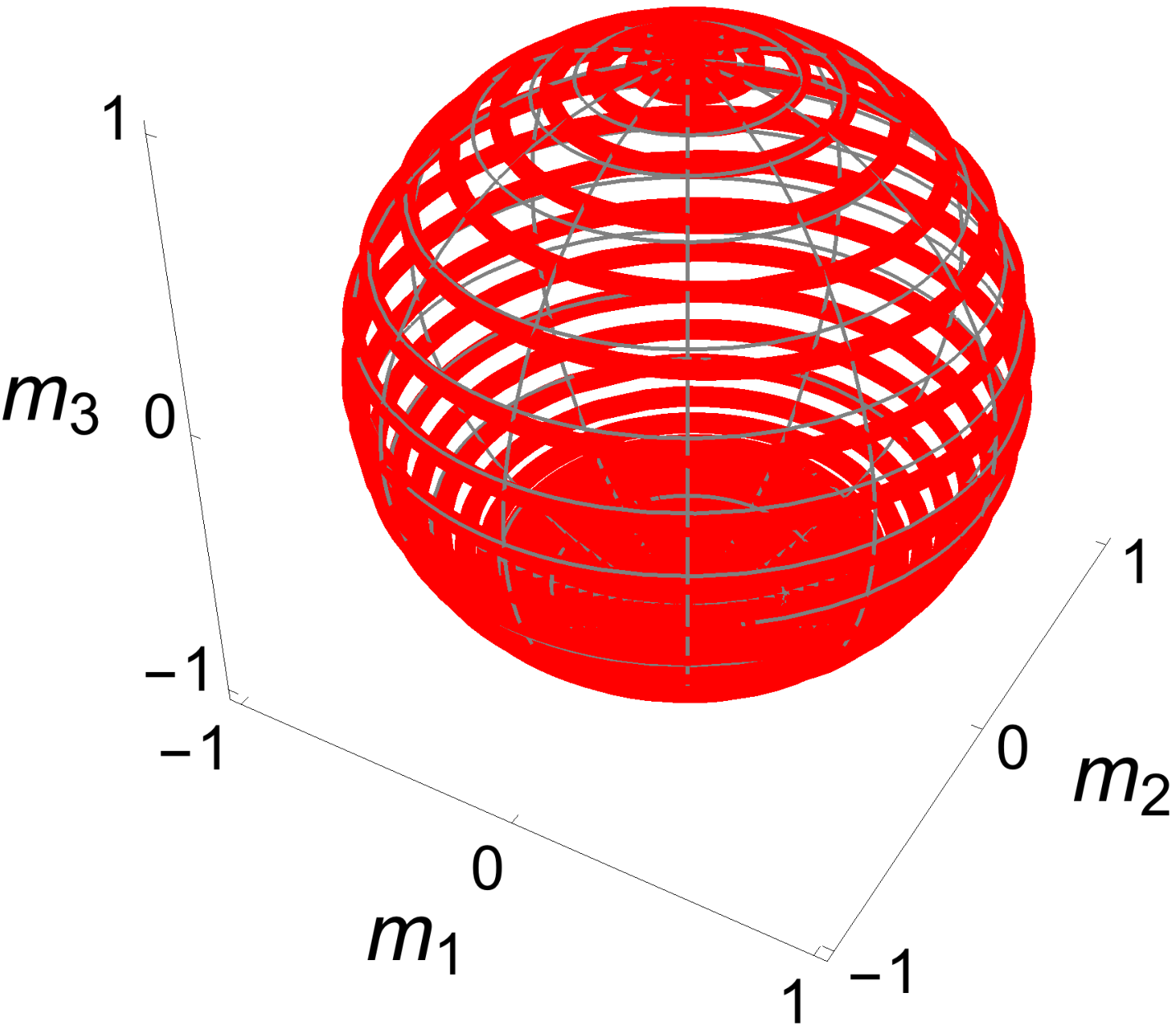}
        \caption{}
    \end{subfigure}
    \caption{\small DWs obtained from continuation of $\m_0$ in system~\eqref{eq:coherent2} projected onto the sphere in the codim-2 regime $h=10.1$ ($h^\ast=10.2$) with initial speed and frequency $s_0=3.96$ and $\Omega_0=8.12$. (a) $(\cc, s, \Omega)=(-0.5, 3.99541, 8.05973)$. (b) $(\cc, s, \Omega)=(0.5, 4.08089, 8.22402)$.}
    \label{fig:continuation h=10.1}
\end{figure}

In addition note the decrease in frequency in the $m_1$ component, and thus also in the $m_2$ as a result of the increase of the $q$ component towards zero, cf.\ Figure~\ref{fig:continuation h=0.5}b. Here, the azimuthal angle decreases since $\phi=\int q$ and $q<0$.

\medskip
As a further example in the codim-2 regime, we consider $h=10.1<10.2=h^\ast$ near the upper boundary of the codim-2 regime in terms of the applied field $h$.  The results of the continuation in $\cc$ are presented in Figure~\ref{fig:continuation h=10.1}. The linear approximation of the splitting in this case is given by $M(-0.5, 0.03541, -0.06027)=\left( -0.000175537, -0.00104378 \right)\tran$ as well as $M(0.5, 0.12089, 0.10402)=\left( -0.00149519,  0.0014975\right)\tran$ in (a) and (b), respectively. Note that the direction of splitting of the two components in this case is also dependent on the polarity sign, as in the previous example. In both cases ($h=0.5$ and $h=10.1$), the continuation results look basically the same in the codim-2 regime, where the solution is, roughly speaking, constantly spiraling down from the north to the south pole. 

\subsection{Center case} 
\begin{figure}[ht]
    \centering
    \begin{subfigure}[b]{0.25\textwidth}
        \includegraphics[width=\textwidth]{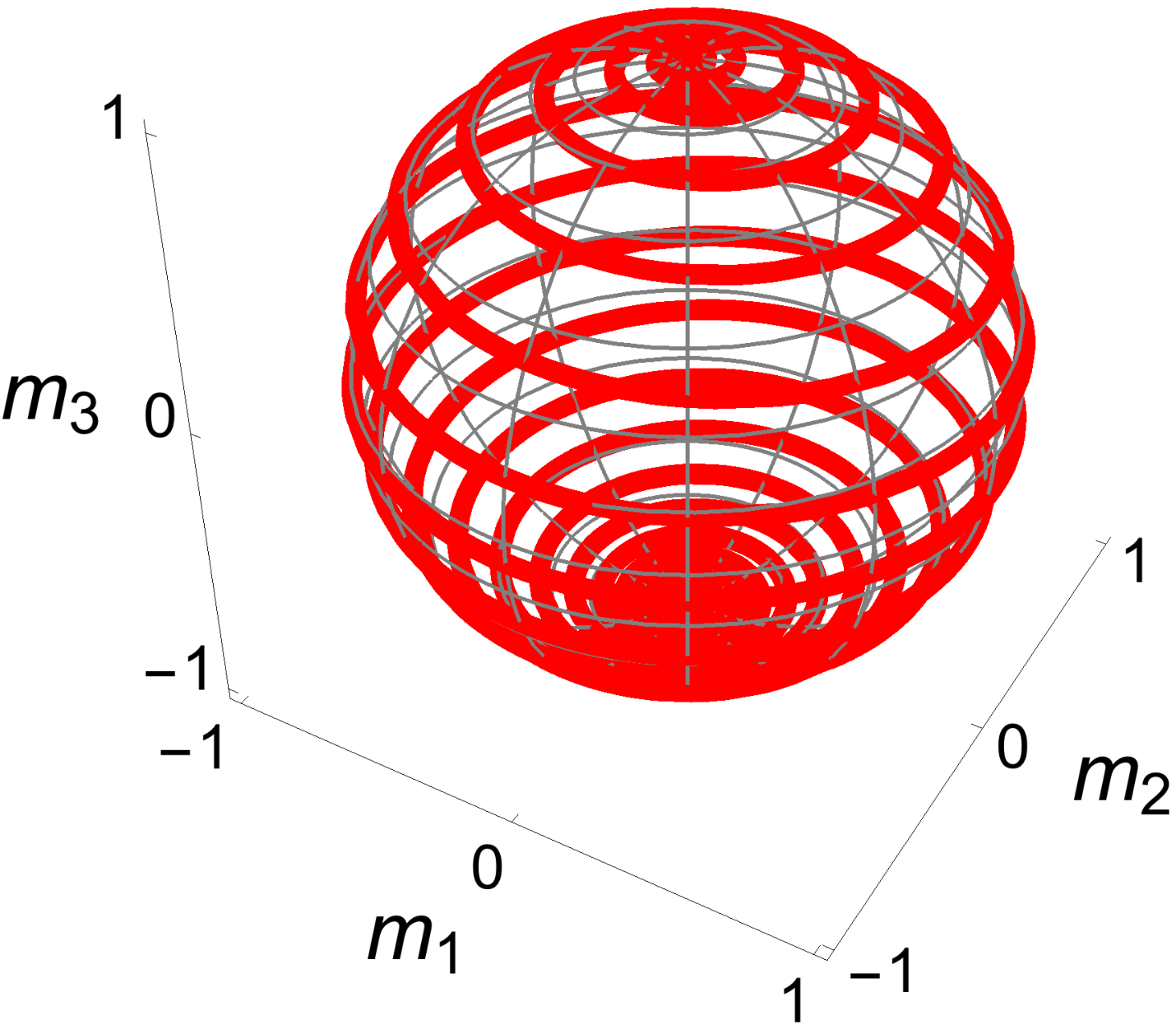}
        \caption{}
    \end{subfigure}
    \hspace{0.25\textwidth}
    \begin{subfigure}[b]{0.27\textwidth}
        \includegraphics[width=\textwidth]{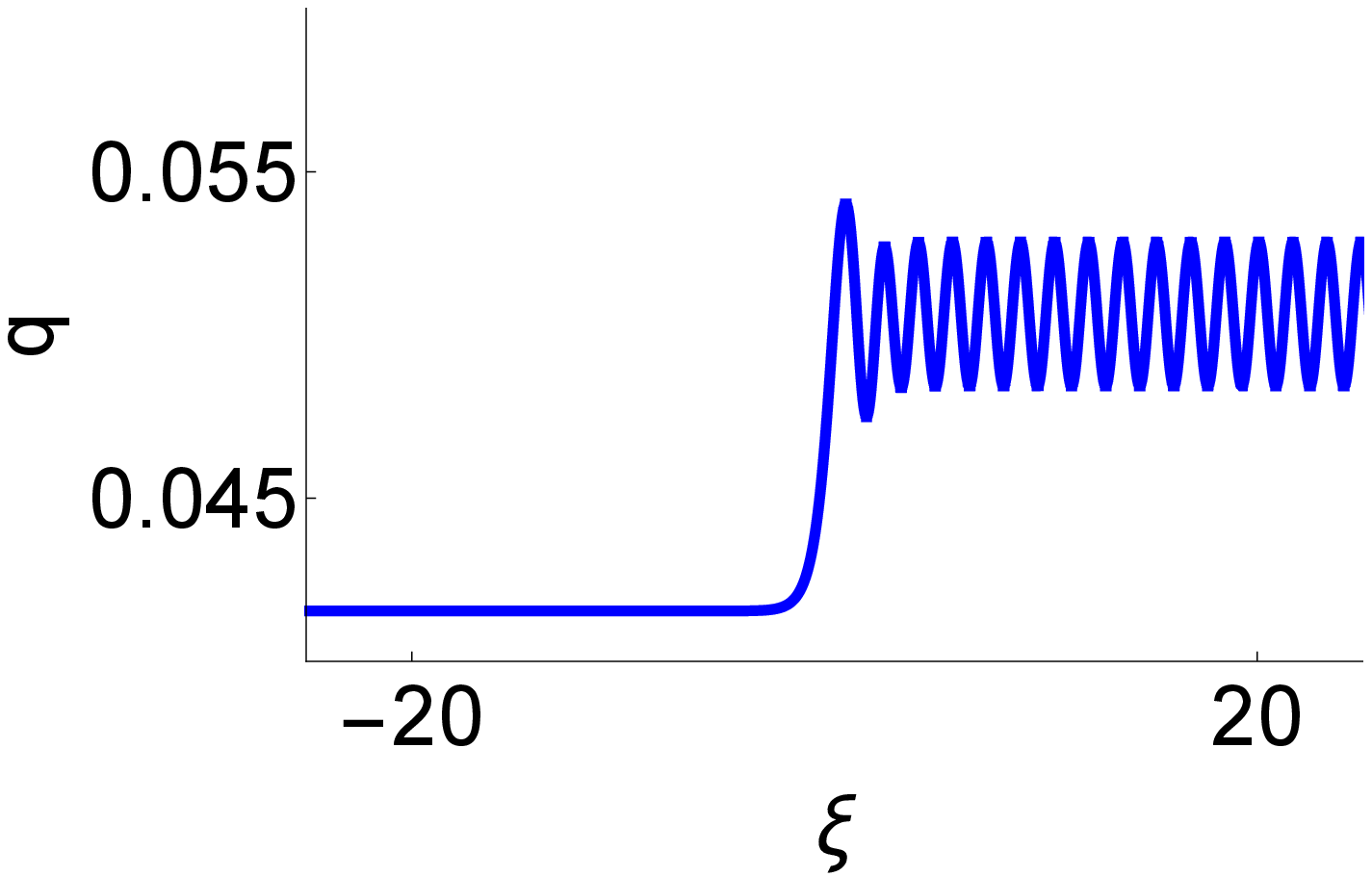}
        \caption{}
    \end{subfigure}
    \caption{\small DWs obtained from continuation of $\m_0$ in system~\eqref{eq:coherent2} in the center regime $h=h^\ast=10.2$ with $s=s_0=4, \Omega=\Omega_0=8.2$, and $\cc=0.5$. (a) Projection onto the sphere. (b) Profile of corresponding $q$-component.}
    \label{fig:continuation h=10.2}
\end{figure}

We perform computations in the center case with applied field $h=h^\ast=10.2$ and fixed frequency $\Omega=\b^-/\alpha+s^2/2$ (see Proposition~\ref{prop:hamilton_general} and its discussion details). Theorem~\ref{theo:center} shows that the right asymptotic state is generically a periodic orbit and more precisely that in case $\cc=0$, no constellation of $h$ and $s$ exists, both not equal to zero, for which the right asymptotic state is the (shifted) equilibrium. The results of continuation in $\cc$ projected on the sphere look quite the same, which is why only the result for $\cc=0.5$ is presented in Figure~\ref{fig:continuation h=10.2}. The fact that the right asymptotic state is a periodic orbit on the blow-up chart $\theta=\pi$ is reflected by the nearly constant oscillations in the $q$-profile for $\xi$ close to the right boundary (cf.\ Figure~\ref{fig:continuation h=10.2}b).

\begin{figure}[h]
    \centering
    \begin{subfigure}[b]{0.25\textwidth}
        \includegraphics[width=\textwidth]{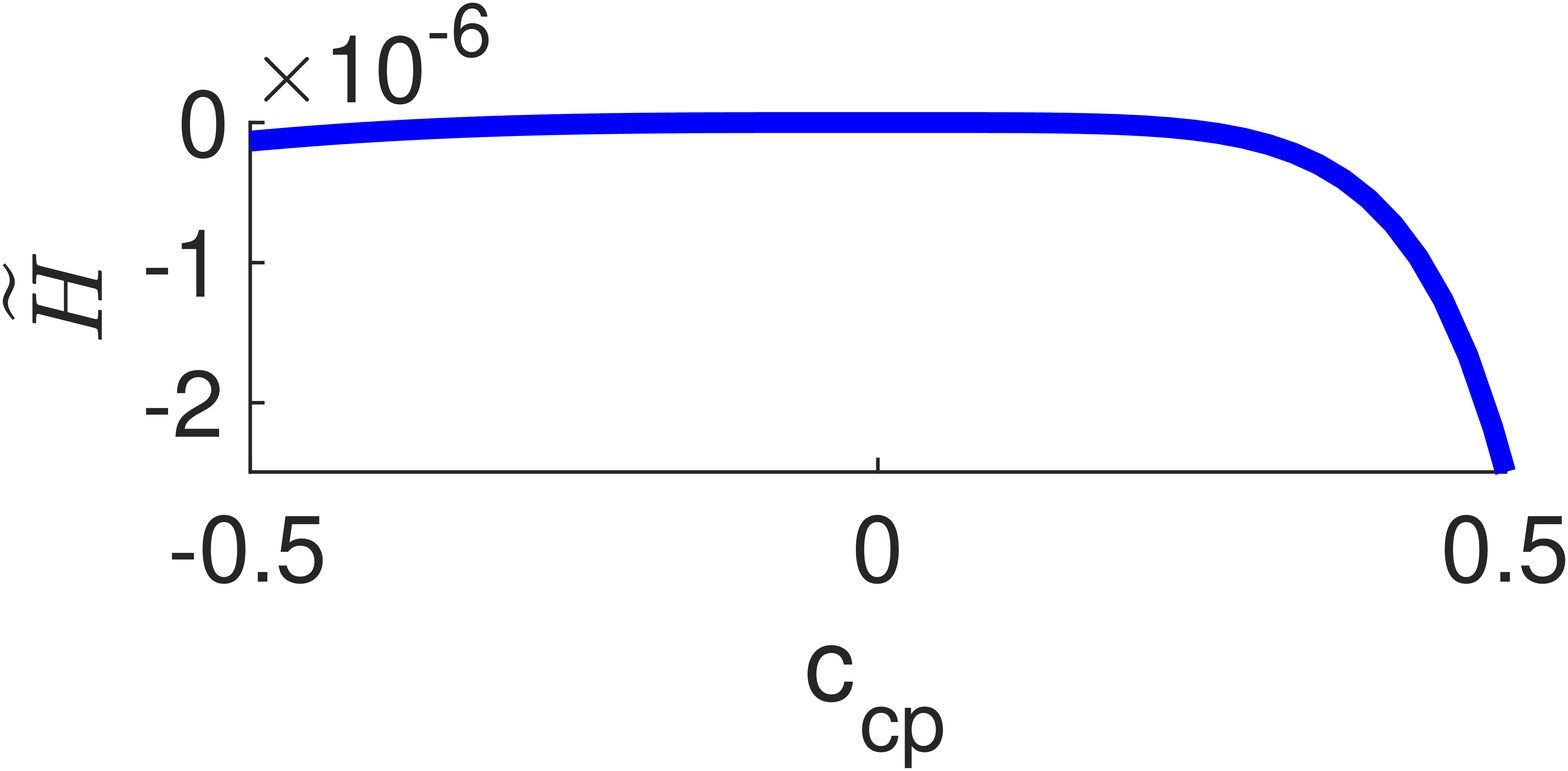}
        \caption{}
    \end{subfigure}
    \hfill
    \begin{subfigure}[b]{0.25\textwidth}
        \includegraphics[width=\textwidth]{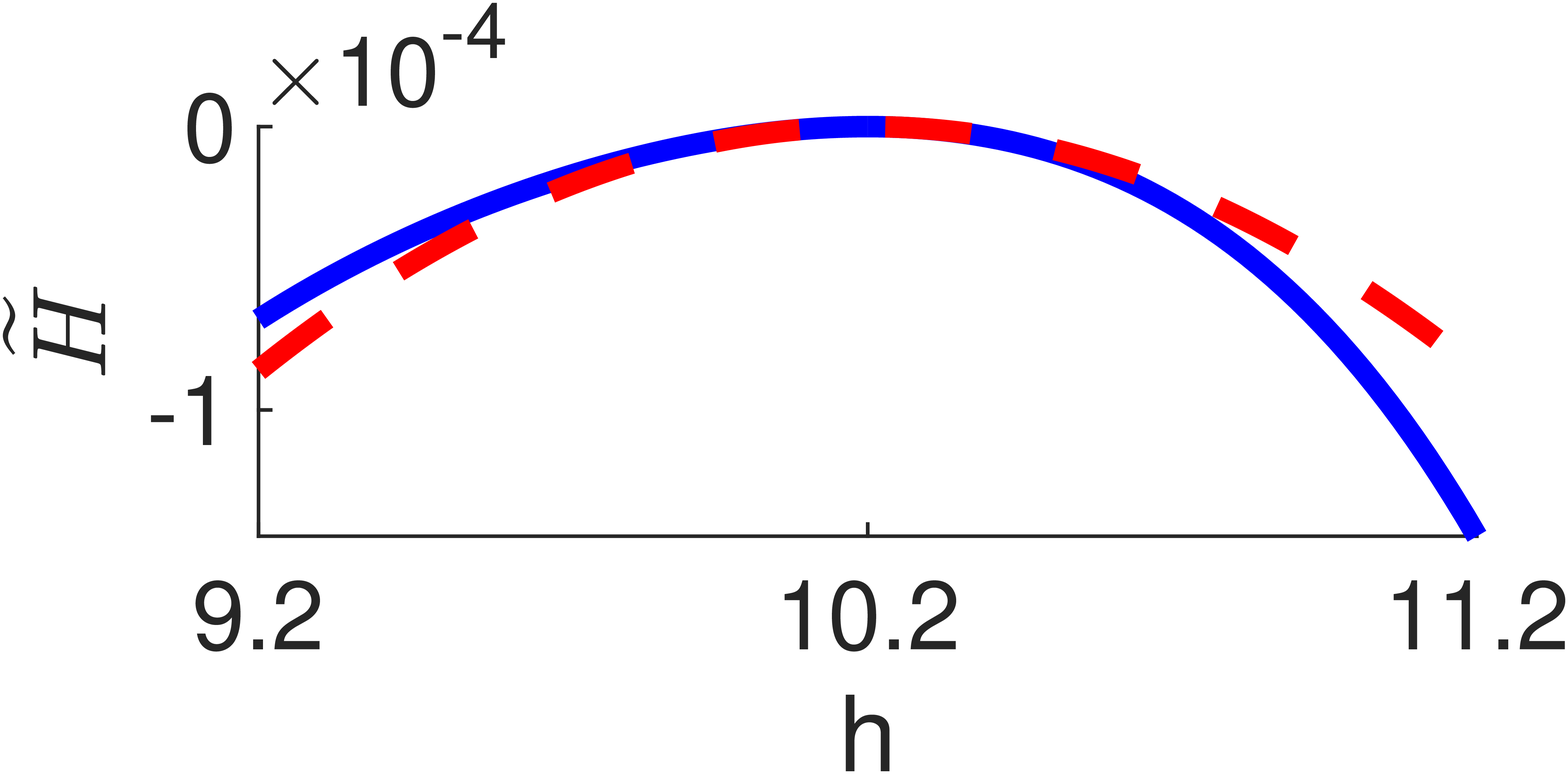}
        \caption{}
    \end{subfigure}
    \hfill
    \begin{subfigure}[b]{0.25\textwidth}
        \includegraphics[width=\textwidth]{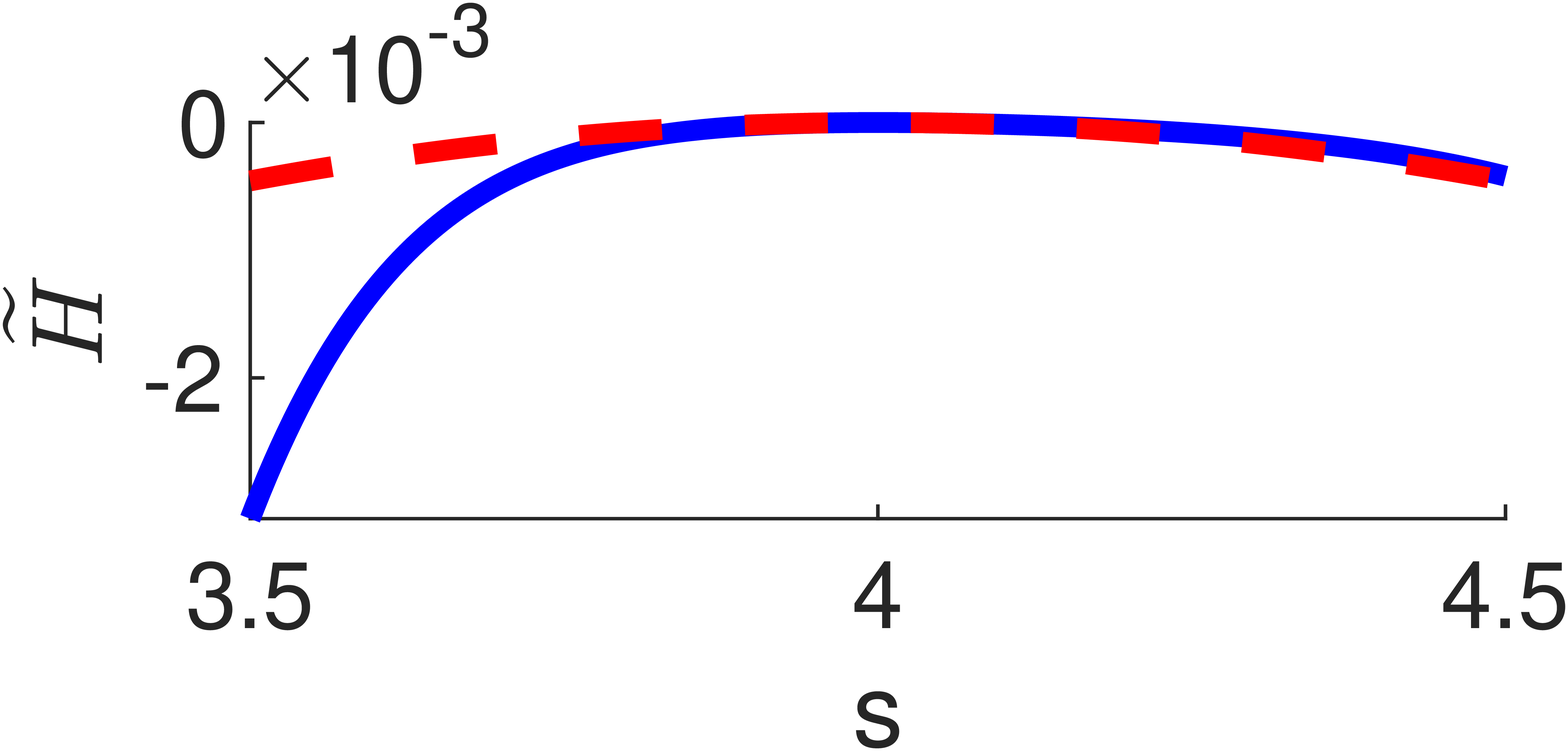}
        \caption{}
    \end{subfigure}
    \caption{\small Continuation of $\m_0$ in system~\eqref{eq:coherent2} in the center case with applied field $h=h^\ast=10.2$ and fixed frequency $\Omega=\b^-/\alpha+s^2/2$; here $s_0=4$. Shown is the energy difference (solid blue line) between the equilibrium and asymptotic state from continuation on right boundary against the continuation parameter $\cc$ in (a), $h$ in (b), and $s$ in (c). The red dashed curve in (b) and (c) is the quadratic approximation~\eqref{eq:second order approx explicit}.}
    \label{fig:energy error}
\end{figure}

That the right state is not the equilibrium in the blow-up chart is further corroborated by computing the difference in energy $\tH$ between this equilibrium state (see Remark~\ref{rem:equilibria}) and the approximate right asymptotic state obtained from continuation. The analytic prediction of this difference up to second order is given by~\eqref{eq:second order approx}, which reads, for the chosen parameters, 
\begin{equation}\label{eq:second order approx explicit}
    -0.006612 + 0.00673 s- 0.00183 s^2 - 0.00134 h - 0.000086 h^2 + 0.00077 h s \,.
\end{equation}
As this analytic prediction is independent of $\cc$ the dependence of $\tH$ on $\cc\approx0$ is of cubic or higher power. Indeed, the results plotted in Figure~\ref{fig:energy error}a suggest an at least quartic dependence since a maximum lies at $\cc=0$. The asymmetric nature of the graph suggests that odd powers appear in the expansion beyond our analysis, but also note the order of \num{e-6} in $\widetilde{H}$. In addition to the dependence on $\cc$, continuations for $\cc=0$ of $\tH$ in $h$ with fixed $s=s_0=4$ and in $s$ with fixed $h=h^\ast=10.2$ are plotted in Figure~\ref{fig:energy error}b and~\ref{fig:energy error}c, respectively. Here we also plot the quadratic prediction~\eqref{eq:second order approx explicit}.

\subsection{Codim-0 case} 
Next, we consider an applied field $h=10.3$ in the codim-0 regime, just above the applied field value for the center case $h^*=10.2$. The results of continuation in $\cc$ are plotted on the sphere in Figure~\ref{fig:continuation h=10.3}. The azimuthal profile in $\phi$ and hence in $q$ are non-trivial as predicted for inhomogeneous DWs. 

In the ODE, $q$ possesses an oscillating profile and has a monotonically decreasing amplitude in both cases. This is a consequence of the proximity to the center case and the convergence to equilibria (see \S~\ref{sec:inhomogeneous DW} for details). Recall that the speed and frequency are not selected by the existence problem during continuation in $\cc$, but are taken as the fixed parameters $(s_0,\Omega_0)$ defined in~\eqref{eq:initialspeedfrequenzy}. 

\begin{figure}[t]
    \centering
    \begin{subfigure}[b]{0.25\textwidth}
        \includegraphics[width=\textwidth]{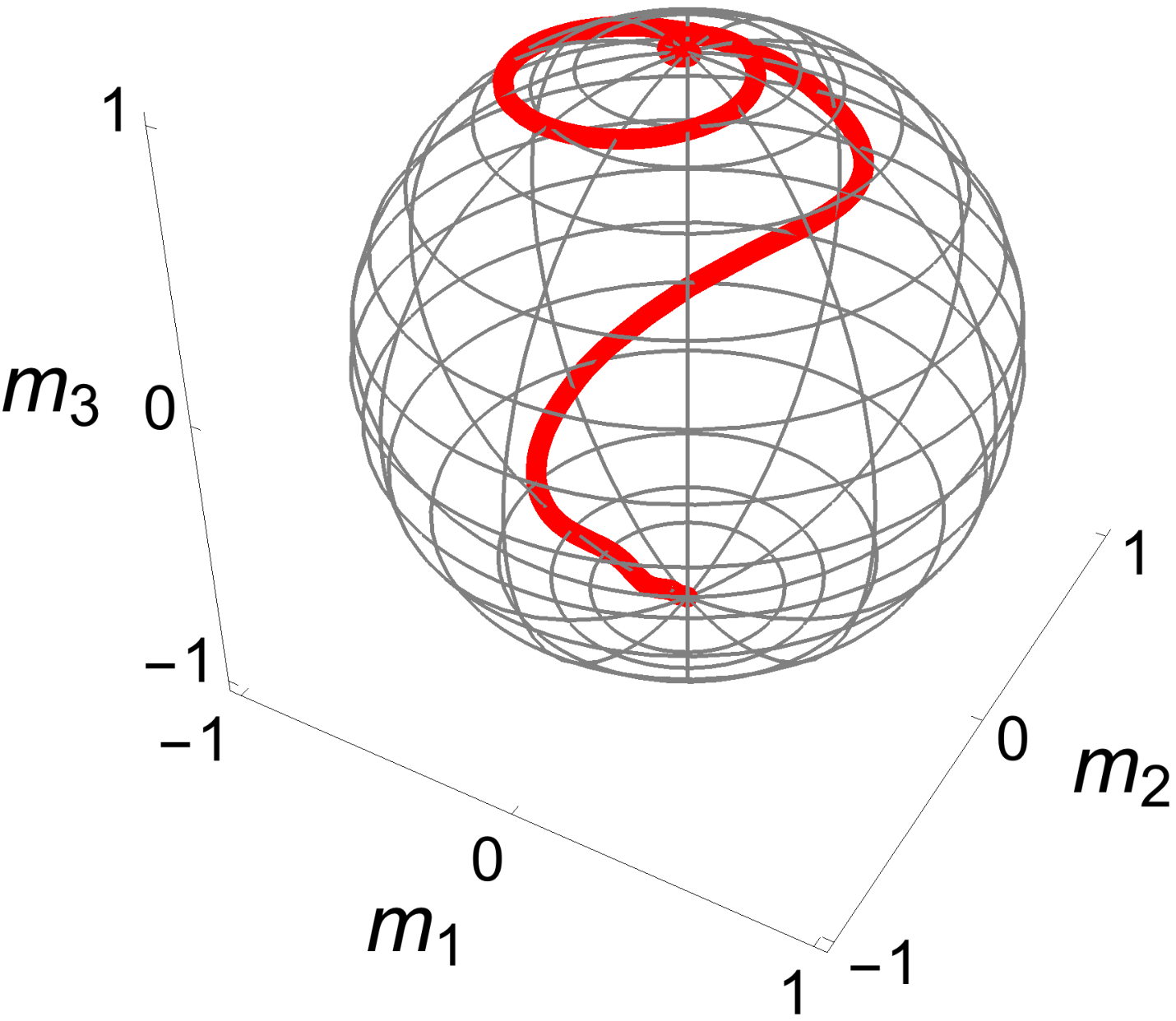}
        \caption{}
    \end{subfigure}
    \hspace{0.25\textwidth}
    \begin{subfigure}[b]{0.25\textwidth}
        \includegraphics[width=\textwidth]{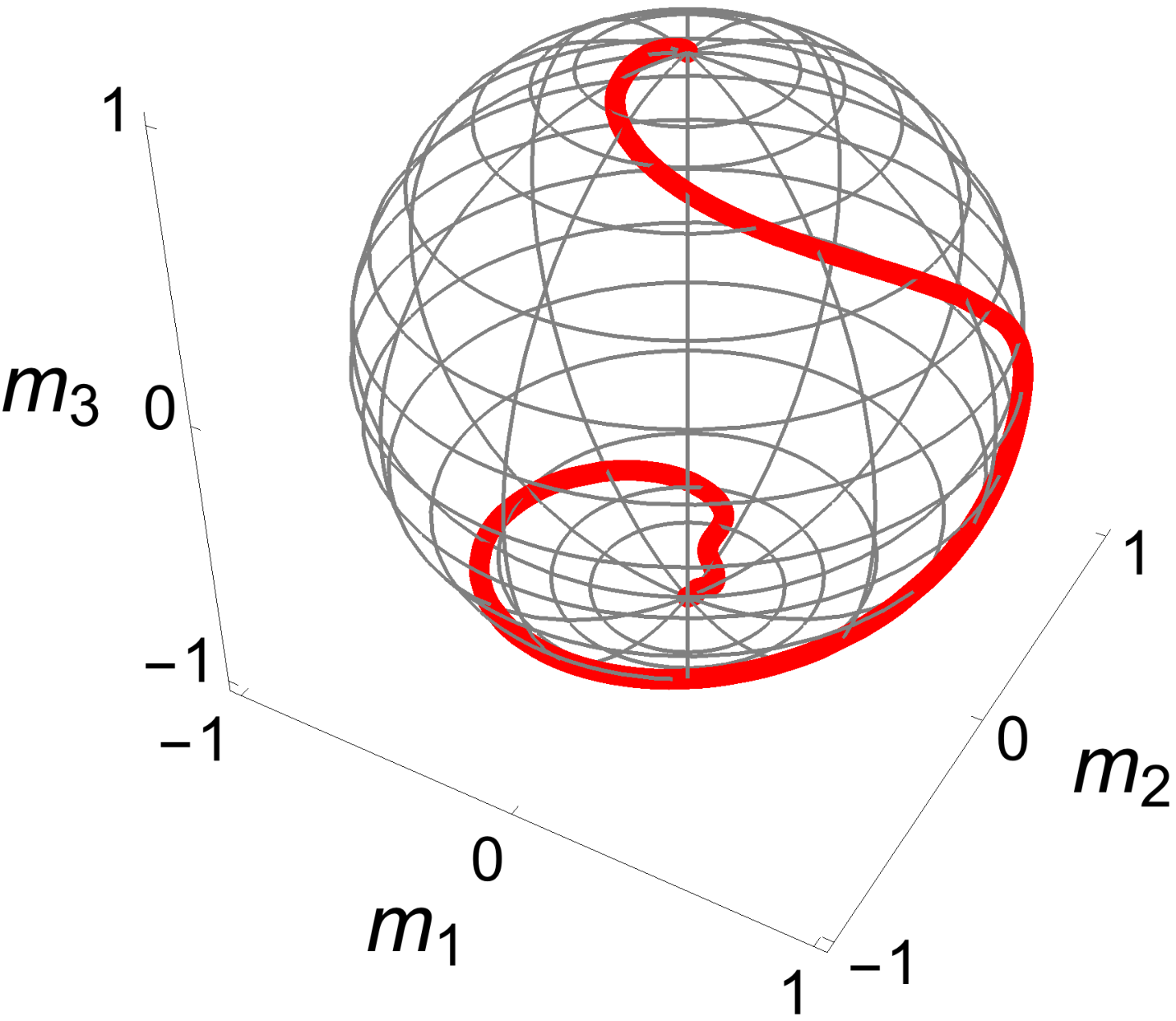}
        \caption{}
    \end{subfigure}
    \caption{\small DWs obtained from continuation of $\m_0$ in system~\eqref{eq:coherent2} projected onto the sphere in the codim-0 regime $h=10.3$ ($h^\ast=10.2$) and initial speed and frequency $s_0=4.04$ and $\Omega_0=8.28$. (a) $\cc=-0.5$. (b) $\cc=0.5$.}
    \label{fig:continuation h=10.3}
\end{figure}

The final example is for a relatively large applied field $h=50$ in the codim-0 regime, far away from the center case, and the results of continuation in $\cc$ projected on the sphere are presented in Figure~\ref{fig:continuation h=50 plus} as well as Figure~\ref{fig:continuation h=50}a. Moreover, the corresponding $m_1$ and $m_2$ profiles for $\cc=0.5$ are presented in Figure~\ref{fig:continuation h=50 plus m1,m2 component}, and for $\cc=-0.5$, in Figure~\ref{fig:continuation h=50}b. As in the previous example, the inhomogeneous nature is visible in the non-trivial azimuthal profile.

\begin{figure}[h]
    \centering
    \begin{subfigure}[b]{0.25\textwidth}
        \includegraphics[width=\textwidth]{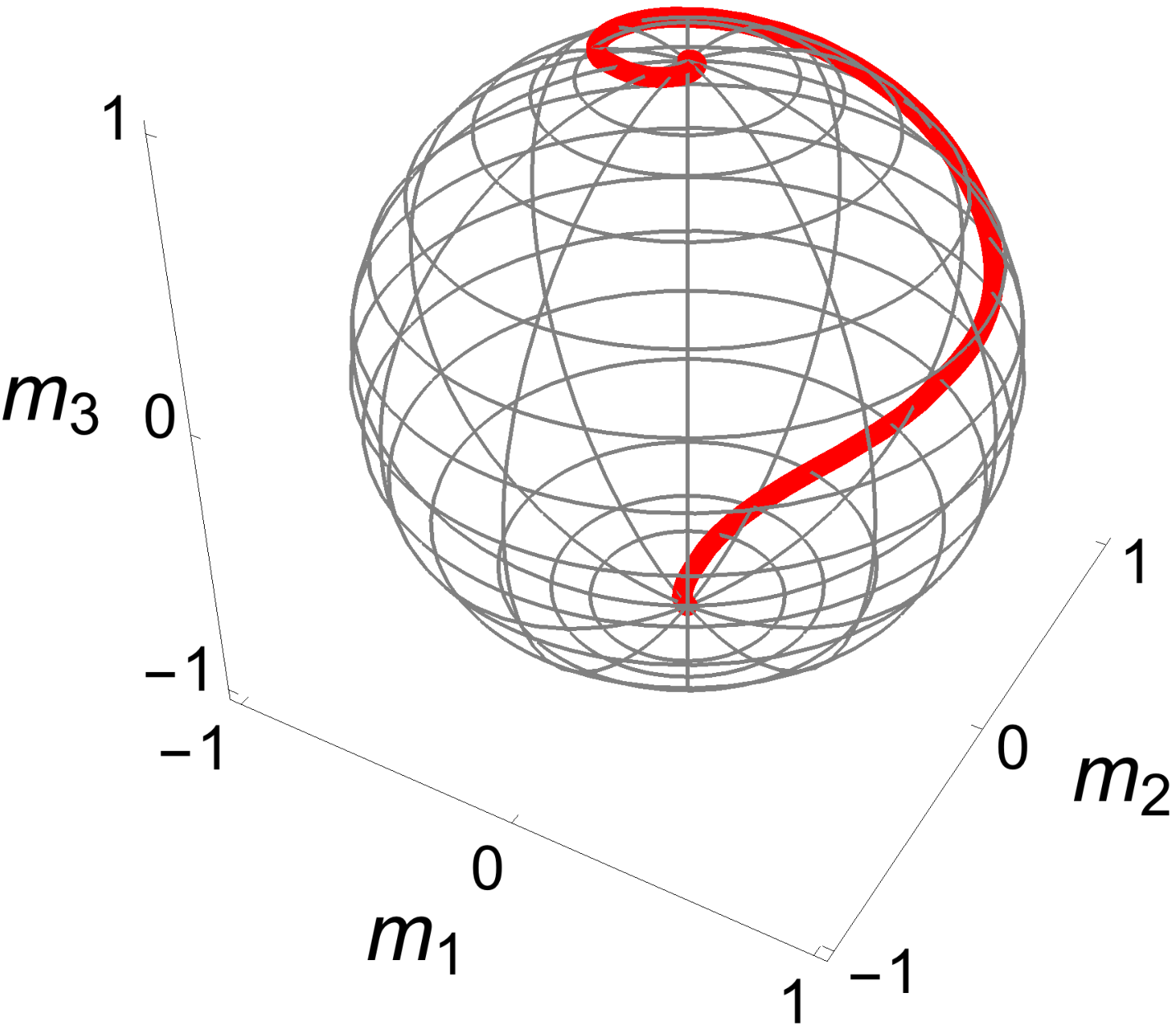}
        \caption{}
    \end{subfigure}
    \hspace{0.25\textwidth}
    \begin{subfigure}[b]{0.25\textwidth}
        \includegraphics[width=\textwidth]{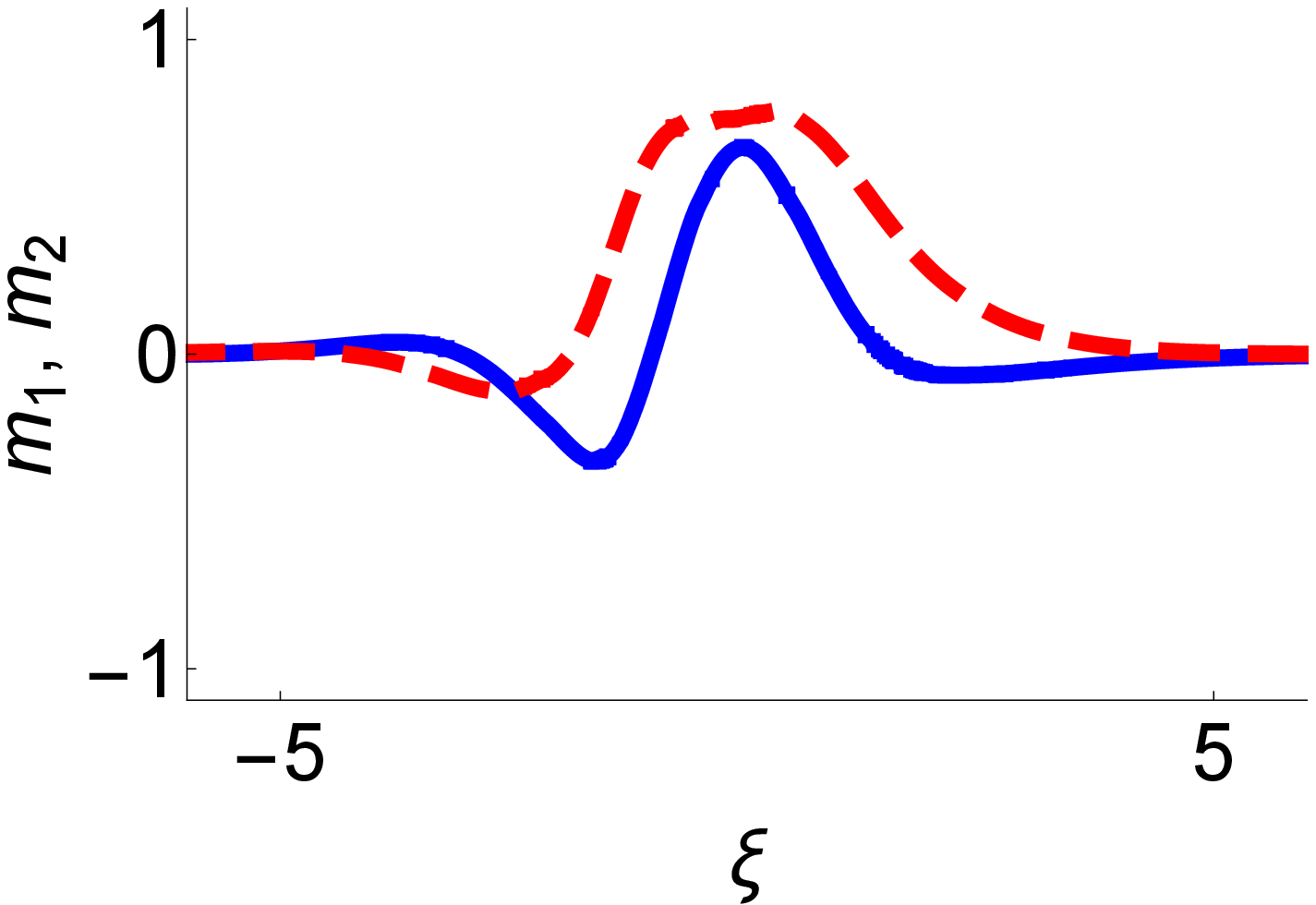}
        \caption{}
    \end{subfigure}
    \caption{\small DWs obtained from continuation of $\m_0$ in system~\eqref{eq:coherent2} in the codim-0 regime $h=50$ ($h^\ast=10.2$) with initial speed and frequency $s_0=19.92, \Omega_0=40.04$, and $\cc=-0.5$. (a) Projection onto the sphere. (b) Zoom-in of the corresponding $m_1$ (solid blue) and $m_2$ (dashed red) component.}
    \label{fig:continuation h=50}
\end{figure}

\medskip
In summary, switching on the parameter $\cc$ leads to a variety of inhomogeneous flat as well as non-flat DW solutions, but also in case $\cc=0$ there exist inhomogeneous DWs (cf.\ Figure~\ref{fig:freezselection}) which are much more complex than the homogeneous one given by~\eqref{eq:explicitsol} (cf.\ Figure~\ref{fig:initial arctan}).

\medskip
Finally, recall from \S~\ref{sec:homogeneous DW} that in the explicit family~\eqref{eq:explicitsol}, the right moving DWs terminate at $s=0$. The question arises what happens for $\cc\neq0$ along the parameter $s$. To study this, we performed a continuation in the parameter $s$ for different values of $\cc$. For decreasing $s$ we found that numerical continuation failed at some $s>0$ for $\cc\neq0$ (cf.\ Figure~\ref{fig:existencecont}). The details of this apparent existence boundary are beyond the scope of this paper. Note that the special role of $s$ is reflected in the splitting function~\eqref{eq:melnikovzerospeed}. In case $s_0=s=0$, the first column of~\eqref{eq:melnikov integral} is zero (see~\eqref{eq:melnikovzerospeed}) and thus the parameterization of $\cc$ can not be written as a function in $s$.

In detail, we continued the analytic solution~\eqref{eq:explicitsol} in $\cc$ away from zero for different initial values of $\beta/\alpha<h<h^\ast$. This led to inhomogeneous ($\cc\neq0$) DWs, which we in turn continued in the parameter $s$ towards zero for different fixed values of $\cc$ until the continuation process fails to converge. Based on this, there is numerical evidence that DWs with opposite speed sign (counter-propagating fronts) can only exist simultaneously for $s=0$ (standing fronts). We took a polynomial fit on these points as an approximation of the existence boundary (cf.\ blue curve in Figure~\ref{fig:existencecont}a). The continuation process towards the boundary is indicated by red arrows in Figure~\ref{fig:existencecont}a for positive $\cc$. Additionally, the corresponding results in the parameter space $\Omega$ and $\cc$ is presented in~\ref{fig:existencecont}b. 

\begin{figure}[h]
    \centering
    \begin{subfigure}[b]{0.25\textwidth}
        \includegraphics[width=\textwidth]{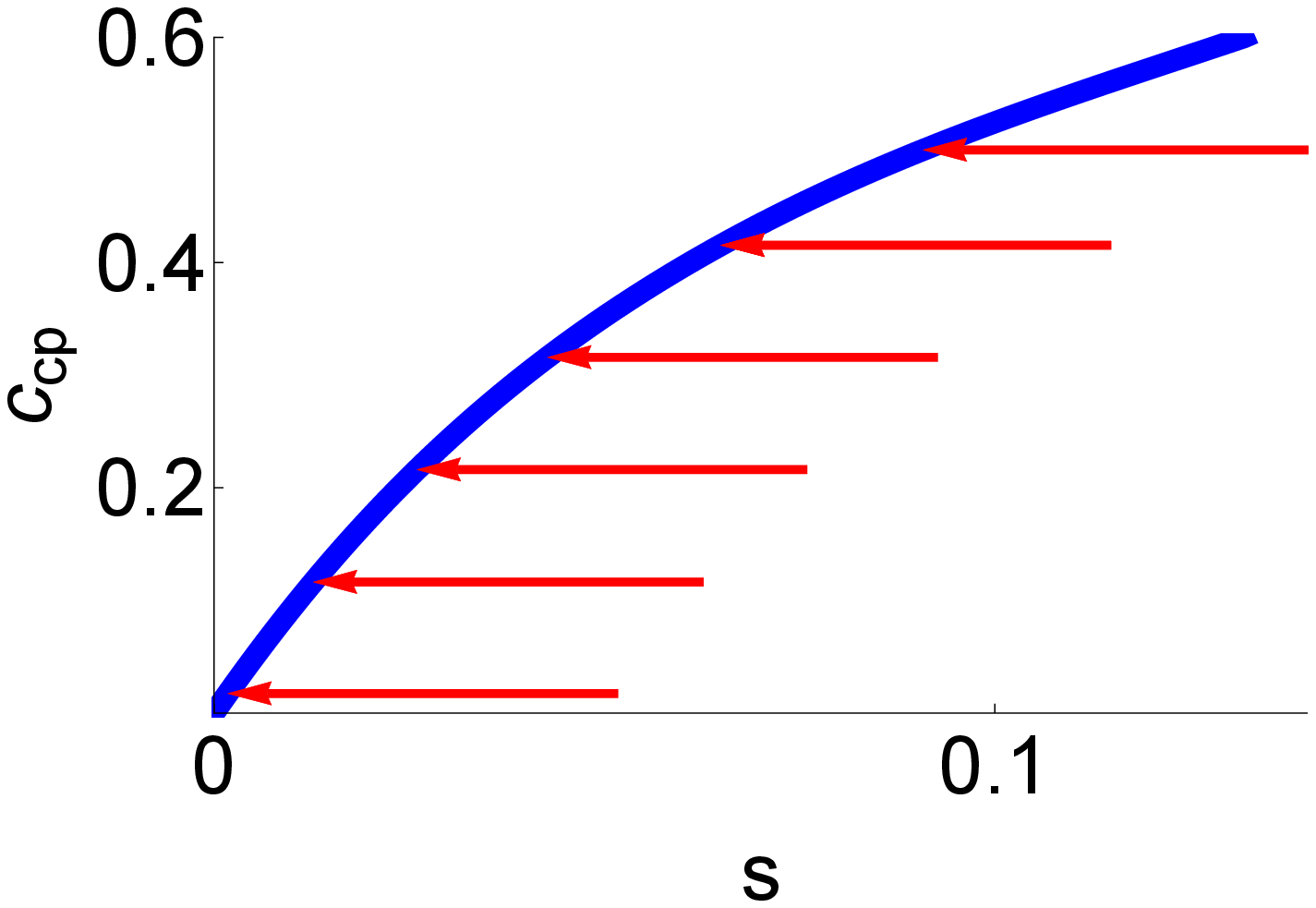}
        \caption{}
    \end{subfigure}
    \hspace{0.25\textwidth}
    \begin{subfigure}[b]{0.25\textwidth}
        \includegraphics[width=\textwidth]{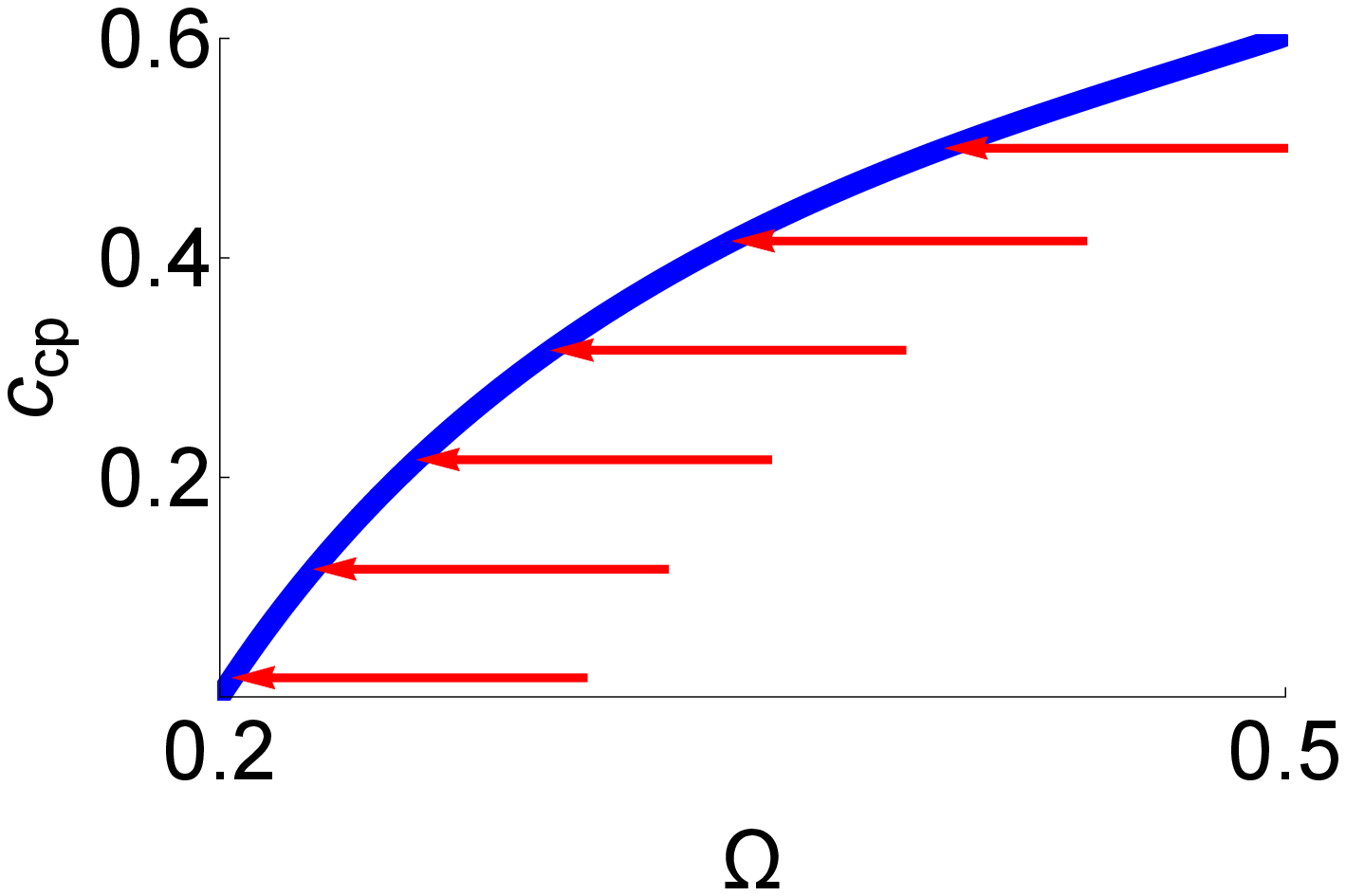}
        \caption{}
    \end{subfigure}
    \caption{\small Continuation results in $s$ and $\Omega$ for fixed $\cc$ in the codim-2 parameter regime. Further parameters are $\alpha=0.5, \beta=0.1, \mu=-1$, and $h$ free. Blue solid line represents interpolated termination boundary for $s$ in (a) and $\Omega$ in (b), respectively. Red arrows indicate continuation approach towards boundary.}
    \label{fig:existencecont}
\end{figure}

As a last point, we briefly describe the numerical method for time-integration near DWs, including freezing of speed and frequency (cf.\ Figure~\ref{fig:freezselection}). All calculations were done with the (free) software package \texttt{pde2path} which is based on a Finite Element Method (FEM), cf.~\cite{dohnal2014pde2path} and the references therein. Time-integration in \texttt{pde2path} with the so-called `freezing' method is discussed in~\cite{rademacher2017symmetries}. In addition to the phase condition for the speed we added a phase condition for the rotation and time-integrated via a semi-implicit Euler scheme.

\section{Discussion and Outlook}\label{sec:discussion}
We have presented results pertaining the existence of different types of domain walls for the~\ref{LLG} as well as~\ref{LLGS} equation. Our main focus has been on a nonzero polarisation parameter $\cc\neq 0$ for any value of the applied field, including the high-field case, and thus for any domain wall speed. These results extend what is known in particular for inhomogeneously structured DWs, and we have discovered an apparently new type of DWs with certain oscillatory tails, referred to as non-flat here. 

\medskip
In detail, we have provided a classification of DWs based on co-dimension properties in a reduced (spatial) coherent structure ODE, which relates to stability and selection properties that we review next. First, we have proven the existence of inhomogeneous flat DWs in case $\cc=0$ as well as $\cc\neq 0$ for an applied field above a certain threshold, which is mainly material depending.  To our knowledge, the only previous existence result for $\cc\neq0$ with 'large' applied fields concerns less relevant non-localized DWs~\cite{Melcher2017}. Here the existence problem does not select speed and frequency.

Second, we have discussed the so-called center case, which is characterized by non-hyperbolic equilibria in the underlying coherent structure ODE. In this case, we have shown the existence of inhomogeneous DWs including the leading order selection mechanism. These solutions are non-flat in case $\cc=0$ and generically also non-flat for $\cc$ away from zero, which was substantiated by numerical results. The fundamental observation has been the existence of a Hamiltonian function in a certain parameter regime in the corresponding coherent structure ODE.

Third, we have proven the existence of inhomogeneous DWs in the so-called codim-2 regime, which is a range of values for the applied field in which the speed $s$ is between zero and the center case speed. In this regime, each solution in case $\cc\neq 0$ is uniquely determined by its speed as well as frequency. Here we have also presented the leading order selection function in the coherent structure ODE variables $p$ and $q$, which depends on the speed $s$, the frequency $\Omega$, and is independent of $\cc$ for standing fronts.

\medskip We believe that these results are not only interesting and relevant from a theoretical and mathematical viewpoint, but  also from an application viewpoint. They could help to better understand the interfaces between different magnetic domains in nanostructures, e.g.\ in the development of racetrack memories, which are a promising prospective high density storage unit that utilize a series of DWs by shifting at high speed along magnetic nanowires through nanosecond current pulses.

\medskip
In order to illustrate and corroborate these theoretical results, we have presented numerical computations for a variety of values for the applied field in \S~\ref{sec:continuation}. On the one hand, the examples in essence show that large applied fields lead to more complex profiles of the DWs in case $\cc\neq0$. On the other hand, while in the center case the DWs projected on the sphere appear similar to those for small applied fields, these solutions approach the poles in a qualitatively different `non-flat' manner -- as predicted by our analysis. Moreover, we compared the numerical and analytical results of the selection mechanism in the center case, showing that the analytical leading order approximation predicts the effect of small perturbations in the parameters. Notably, for applied fields above a certain threshold, where the existence analysis does not provide a selection of speed and frequency, numerically the DWs selected in the PDE dynamics are in the center case, both for $\cc=0$ as well as $\cc\neq0$. Hence, it might be possible to detect these solutions in a high-field regime in real materials.

\medskip
One question concerning existence beyond our analysis is whether inhomogeneous (flat or non-flat) solutions exist for any value of $\cc\in(-1,1)$, and whether this class could be utilized in applications.

\medskip
A natural step towards the understanding of domain wall motion in nanowires beyond the question of existence concerns the dynamic stability. For inhomogeneous solutions there appears to be no rigorous result in this direction. In particular for larger applied fields, stability results would be an essential step towards understanding the selection mechanism of solutions in terms of speed and frequency; our first numerical investigations show that solutions in the center parameter regime are selected, i.e., inhomogeneous non-flat DWs.

\medskip Moreover, preliminary analytic results, for $\cc=0$ as well as $\cc\neq 0$, show that selection mechanism is mainly determined by the value of the applied field, where in the bi-stable case ($\pm\3$ linearly stable) homogeneous DWs are selected, and in the mono-stable case inhomogeneous non-flat DWs are selected, which will be studied in detail in an upcoming work.

\bibliographystyle{unsrt}
\bibliography{references}

\begin{thebibliography}{10}

\bibitem{parkin2008magnetic}
Stuart~SP Parkin, Masamitsu Hayashi, and Luc Thomas.
\newblock Magnetic domain-wall racetrack memory.
\newblock {\em Science}, 320(5873):190--194, 2008.

\bibitem{reohr2002memories}
William Reohr, Heinz Honigschmid, Raphael Robertazzi, Dietmar Gogl, Frank
  Pesavento, Stefan Lammers, Kelvin Lewis, Christian Arndt, Yu~Lu, Hans
  Viehmann, et~al.
\newblock Memories of tomorrow.
\newblock {\em IEEE circuits and devices magazine}, 18(5):17--27, 2002.

\bibitem{lin200945nm}
CJ~Lin, SH~Kang, YJ~Wang, K~Lee, X~Zhu, WC~Chen, X~Li, WN~Hsu, YC~Kao, MT~Liu,
  et~al.
\newblock 45nm low power cmos logic compatible embedded stt mram utilizing a
  reverse-connection 1t/1mtj cell.
\newblock In {\em Electron Devices Meeting (IEDM), 2009 IEEE International},
  pages 1--4. IEEE, 2009.

\bibitem{qureshi2009scalable}
Moinuddin~K Qureshi, Vijayalakshmi Srinivasan, and Jude~A Rivers.
\newblock Scalable high performance main memory system using phase-change
  memory technology.
\newblock {\em ACM SIGARCH Computer Architecture News}, 37(3):24--33, 2009.

\bibitem{mayergoyz2009nonlinear}
Isaak~D Mayergoyz, Giorgio Bertotti, and Claudio Serpico.
\newblock {\em Nonlinear magnetization dynamics in nanosystems}.
\newblock Elsevier, 2009.

\bibitem{hubert2008magnetic}
Alex Hubert and Rudolf Sch{\"a}fer.
\newblock {\em Magnetic domains: the analysis of magnetic microstructures}.
\newblock Springer Science \& Business Media, 2008.

\bibitem{bertotti2008spin}
Giorgio Bertotti.
\newblock Spin-transfer-driven magnetization dynamics.
\newblock {\em Magnetic Nanostructures in Modern Technology}, pages 37--60,
  2008.

\bibitem{berger1996emission}
L~Berger.
\newblock Emission of spin waves by a magnetic multilayer traversed by a
  current.
\newblock {\em Physical Review B}, 54(13):9353, 1996.

\bibitem{slonczewski1996current}
John~C Slonczewski.
\newblock Current-driven excitation of magnetic multilayers.
\newblock {\em Journal of Magnetism and Magnetic Materials}, 159(1-2):L1--L7,
  1996.

\bibitem{osborn1945demagnetizing}
JA~Osborn.
\newblock Demagnetizing factors of the general ellipsoid.
\newblock {\em Physical review}, 67(11-12):351, 1945.

\bibitem{slonczewski2002currents}
JC~Slonczewski.
\newblock Currents and torques in metallic magnetic multilayers.
\newblock {\em Journal of Magnetism and Magnetic Materials}, 247(3):324--338,
  2002.

\bibitem{Melcher2017}
Christof Melcher and Jens~DM Rademacher.
\newblock Pattern formation in axially symmetric
  landau--lifshitz--gilbert--slonczewski equations.
\newblock {\em Journal of Nonlinear Science}, 27(5):1551--1587, 2017.

\bibitem{goussev2010domain}
Arseni Goussev, JM~Robbins, and Valeriy Slastikov.
\newblock Domain-wall motion in ferromagnetic nanowires driven by arbitrary
  time-dependent fields: An exact result.
\newblock {\em Physical review letters}, 104(14):147202, 2010.

\bibitem{aakerman2005toward}
Johan {\AA}kerman.
\newblock Toward a universal memory.
\newblock {\em Science}, 308(5721):508--510, 2005.

\bibitem{landautheory}
LD~Landau and EM~Lifshitz.
\newblock On the theory of the dispersion of magnetic permeability in
  ferromagnetic bodies, phy. z. sowjetunion 8: 153.
\newblock {\em Reproduced in Collected Papers of LD Landau}, pages 101--114,
  1935.

\bibitem{gilbert2004phenomenological}
Thomas~L Gilbert.
\newblock A phenomenological theory of damping in ferromagnetic materials.
\newblock {\em IEEE Transactions on Magnetics}, 40(6):3443--3449, 2004.

\bibitem{lakshmanan2011fascinating}
M~Lakshmanan.
\newblock The fascinating world of the landau--lifshitz--gilbert equation: an
  overview.
\newblock {\em Philosophical Transactions of the Royal Society of London A:
  Mathematical, Physical and Engineering Sciences}, 369(1939):1280--1300, 2011.

\bibitem{sandstede2000absolute}
Bj{\"o}rn Sandstede and Arnd Scheel.
\newblock Absolute and convective instabilities of waves on unbounded and large
  bounded domains.
\newblock {\em Physica D: Nonlinear Phenomena}, 145(3-4):233--277, 2000.

\bibitem{gou2011stability}
Yan Gou, Arseni Goussev, JM~Robbins, and Valeriy Slastikov.
\newblock Stability of precessing domain walls in ferromagnetic nanowires.
\newblock {\em Physical Review B}, 84(10):104445, 2011.

\bibitem{dohnal2014pde2path}
T~Dohnal, J~Rademacher, H~Uecker, and D~Wetzel.
\newblock pde2path 2.0.
\newblock {\em ENOC}, 2014.

\bibitem{rademacher2017symmetries}
Jens~DM Rademacher and Hannes Uecker.
\newblock Symmetries, freezing, and hopf bifurcations of traveling waves in
  pde2path, 2017.

\bibitem{kuznetsov2013elements}
Yuri~A Kuznetsov.
\newblock {\em Elements of applied bifurcation theory}, volume 112.
\newblock Springer Science \& Business Media, 2013.

\end{thebibliography}

\section{Appendix}
\subsection{Proof of Theorem~\ref{theo:center}}\label{app:center case proof}
We use the notation $$u=u(\xi; \eta, \alpha, \beta, \mu)=\left(\theta(\xi; \eta, \alpha, \beta, \mu), p(\xi; \eta, \alpha, \beta, \mu), q(\xi; \eta, \alpha, \beta, \mu)\right)\tran$$ and bifurcation parameters $\eta=\left(\cc, s, h\right)\tran$, where $s_0$ and $h_0=h^\ast$ are defined below (see \S~\ref{sec:inhomogeneous DW} for details). The starting point for our perturbation analysis are the unperturbed parameters and explicit heteroclinic solution in the center case~\eqref{eq:center case}, where the frequency is $\Omega_0=s_0^2/2+\beta/\alpha$. These are given by $$\eta_0:=\begin{pmatrix}  c_{\textnormal{cp}_0}\\s_0\\h_0\end{pmatrix}:=\begin{pmatrix} 0\\\frac{2\smu}{\alpha}\\\frac{\beta}{\alpha}-2\mu-\frac{2\mu}{\alpha^2} \end{pmatrix}$$ as well as $$u_0=u_0(\xi;\eta_0,\alpha,\beta,\mu):=\begin{pmatrix} \theta_0(\xi;\eta_0,\alpha,\beta,\mu)\\p_0(\xi;\eta_0,\alpha,\beta,\mu)\\q_0(\xi;\eta_0,\alpha,\beta,\mu) \end{pmatrix}:=\begin{pmatrix}2\arctan\left( \exp(\smu\;\xi) \right)\\\smu \\ 0\end{pmatrix}.$$ Unless stated otherwise, we suppress the explicit dependence of $u$ on $\alpha, \beta$, and $\mu$ in the following discussion. Let us write $Z^\pi:=Z^\pi_-$ with the notation from Remark~\ref{rem:equilibria} so that the unperturbed right asymptotic state is given by $$Z^\pi(\eta_0)=\left(\pi,\frac{\alpha s_0}{2},\frac{s_0}{2}-\sqrt{-\frac{\mu}{\alpha^2}}\right)\tran=\left(\pi,\smu,0\right)\tran $$ and its derivative with respect to $\eta$ is given by
$$Z^\pi_\eta(\eta_0)=\begin{pmatrix} 0&0&0\\0&\frac{\alpha}{2}&0\\ \frac{\beta}{2\smu}&\frac{2+\alpha^2}{2}&-\frac{\alpha}{2\smu} \end{pmatrix}.$$
We write system \eqref{eq:coherent2} for brevity as
\begin{equation}\label{eq:usystem}u'=f(u;\eta)\,,
\end{equation}
so $f(u;\eta)$ denotes the right side of\ \eqref{eq:coherent2}. The linearization w.r.t.\ $\eta$ in the unperturbed heteroclinic connection $u_0$, given by~\eqref{eq:atan}, is the non-autonomous linear equation
\begin{equation}\label{eq:linearizazion}
u_\eta' = f_u(u_0;\eta_0)u_\eta + f_\eta(u_0;\eta_0) \eta,
\end{equation}
where $u_\eta=\left(\theta_\eta,p_\eta,q_\eta\right)\tran$. Its homogeneous part is
\begin{equation}\label{eq:linearization homogeneous}
\begin{array}{rl}
\theta'_\eta &= \smu\cos(\theta_0)\theta_\eta + p_\eta\sin(\theta_0)\\
p'_\eta &= -\left( \alpha s_0 + 2\smu \cos (\theta_0) \right) p_\eta + s_0q_\eta\\
q'_\eta &= -s_0p_\eta -\left(\alpha s_0 + 2\smu \cos(\theta_0)\right)q_\eta
\end{array},
\end{equation}
with $\theta_0(\xi)=2\arctan\left(\exp(\smu \,\xi)\right)$ due to~\eqref{eq:atan}. We next solve~\eqref{eq:linearization homogeneous} and determine its fundamental solution matrix.

\medskip
The first obvious vector-solution of it is $U_1 = u_0' = (\theta'_0,0,0)$ since the second and the third equation of~\eqref{eq:linearization homogeneous} do not depend
on $\theta_\eta$. The other solutions can be obtained from $U_1$ and the result of Lemma~\ref{lem:explicit}. Changing to polar coordinates $$p_\eta = r \cos\varphi\, , \quad q_\eta= r \sin\varphi\,,
$$ the equations for $p_\eta$ and $q_\eta$ become
$$
\begin{array}{rl}
r' &=  -\left(\alpha s_0 +2\smu \cos (\theta_0) \right) r\\
\varphi' &= -s_0\,,
\end{array}
$$
whose general solution can be written as
\begin{equation*}
\begin{array}{rl}
p_\eta &= r_0 r(\xi) \cos(-s_0 \xi + \varphi_0)\\
q_\eta &= r_0 r(\xi) \sin(-s_0 \xi + \varphi_0)
\end{array}
\end{equation*}
where
$$r(\xi) = \exp\left(-\alpha s_0 \xi - 2 \smu \int\limits_\xi \cos(\theta_0(\tau)) d\tau \right) = \left(1 + \rme^{2 \smu \xi}\right)^2 \, \rme^{(-2 \smu -
\alpha s_0)\xi},$$ 
and $r_0, \varphi_0$ are arbitrary integration constants corresponding to suitable initial conditions. Note that $\lim\limits_{\xi \to \pm \infty} r(\xi) = \infty$ for $0\le s_0<2\smu/\alpha$.

Next, the values of the integration constants have to be selected in order for the second and the third vector-solutions
\begin{equation}\label{eq:Usol}
U_2 = \begin{pmatrix} \theta^1_1 \\ r_1 r(\xi) \cos(-s_0 \xi + \varphi_1) \\ 
r_1 r(\xi) \sin(-s_0 \xi + \varphi_1) \end{pmatrix}\,, \quad 
U_3 = \begin{pmatrix} \theta^2_1 \\ r_2 r(\xi) \cos(-s_0 \xi + \varphi_2) \\
r_2 r(\xi) \sin(-s_0 \xi + \varphi_2) \end{pmatrix}
\end{equation}
to be linearly independent. Here $\theta^1_1, \theta^2_1$ are not relevant for what follows. The determinant of the fundamental matrix reads
$$
\det \Phi(\xi) = \det \left( U_1(\xi), U_2(\xi), U_3(\xi) \right) = r_1 r_2 r^2(\xi)\theta'_0(\xi) \sin(\varphi_2 - \varphi_1),
$$
which is non-zero for $r_1 = r_2 = 1$, $\varphi_1 = 0$ and $\varphi_2 = \pi / 2$, i.e.\ $\det \Phi(\xi) = r^2(\xi)\theta'_0(\xi)$. Together, we get the fundamental solution matrix of the homogeneous part as

\begin{equation}\label{eq:fundMatrix}
    \Phi(\xi) = \begin{pmatrix} \theta'_0(\xi) & \theta^1_1(\xi) & \theta^2_1(\xi) \\ 
    0 & r(\xi) \cos(-s_0 \xi) & -r(\xi) \sin(-s_0 \xi) \\ 
   0 & r(\xi) \sin(-s_0 \xi) &  r(\xi) \cos(-s_0 \xi) \end{pmatrix}.
\end{equation}

The derivative of~\eqref{eq:usystem} with respect to $\eta$ is given by~\eqref{eq:linearizazion} and from the variation of constants formula we get for some $\xi_0$ that 
$$u_\eta(\xi)=\Phi_{\xi,\xi_0} u_\eta(\xi_0)+\int\limits_{\xi_0}^{\xi}{\Phi_{\xi,\tau}f_\eta(u_0(\tau);\eta_0)d\tau}\,,$$ 
where $\Phi_{\xi,\tau}=\Phi(\xi) \cdot \Phi^{-1}(\tau)$ is the evolution operator.
Using~\eqref{eq:fundMatrix} we find
\begin{equation}\label{eq:fundamental solution}
    \Phi_{\xi,\tau}(\xi,\tau;\eta)=\begin{pmatrix}
    \Theta_1 & \Theta_2 & \Theta_3 \\ 0 & \frac{r(\xi)}{r(\tau)}\cos\left( -s_0(\xi-\tau) \right) & -\frac{r(\xi)}{r(\tau)}\sin \left(-s_0(\xi-\tau)\right)\\ 0 & \frac{r(\xi)}{r(\tau)}\sin\left( -s_0(\xi-\tau) \right) & \frac{r(\xi)}{r(\tau)}\cos\left(-s_0(\xi-\tau)\right)
    \end{pmatrix},
\end{equation}
where the explicit forms of the functions $ \Theta_{1,2,3}(\xi)$ are not relevant for the remainder of this proof. Since $u_\eta(\xi)$ tends to $\partial_\eta Z^0_-$ for $\xi\to-\infty$ the hyperbolicity of $Z^0_-$ (more precisely the resulting exponential dichotomy) implies $\Phi_{\xi,\xi_0} u_\eta(\xi_0)\to 0$ as $\xi_0\to-\infty$ and so
\begin{equation}\label{eq:u variation}
   u_\eta(\xi) = \int\limits_{-\infty}^{\xi}{\Phi_{\xi,\tau}f_\eta(u(\tau;\eta_0);\eta_0)d\tau}\,.
\end{equation}

Regarding the limiting behavior as $\xi\to\infty$, recall that Corollary~\ref{cor:transition} states that the right asymptotic limit of the perturbed heteroclinic orbit is either the perturbed equilibrium $Z^\pi(\eta)$ or a periodic orbit around it in the blow-up chart at $\theta=\pi$. The integral~\eqref{eq:u variation} distinguishes these case in the sense that either it has a limit as $\xi\to +\infty$ so the heteroclinic orbit connects the two equilibria, or it does not and the heteroclinic orbit connects to a periodic solution.

\medskip
We next determine $u_\eta(\xi)$ componentwise $$u_\eta(\xi)=v\coloneqq\begin{pmatrix} v_{11}&v_{12}&v_{13}\\v_{21}&v_{22}&v_{23}\\v_{31}&v_{32}&v_{33} \end{pmatrix},$$
where $v_{ij}$ are the components of~\eqref{eq:u variation} and index $i=1,2,3$ relates to $\theta, p, q$ as well as $j=1,2,3$ to $\cc, s, h$.

Towards this, we compute
$$f_\eta(u_0(\tau),\eta_0)=\begin{pmatrix} 0&0&0\\-\beta/\alpha&-\frac{\smu}{\alpha}(2+\alpha^2)&1\\ \frac{2\beta}{1+\rme^{2\smu \;\tau}} &\smu& 0 \end{pmatrix},
$$
and together with~\eqref{eq:u variation} and~\eqref{eq:fundamental solution} we obtain

\begin{table}[h]
    \centering
    \begin{tabular}{lll}
    $v_{21}=-\frac{\beta}{\alpha}I_C-2\beta J_S\,,$ &$v_{22}=-\frac{\smu}{\alpha}(2+\alpha^2)I_C-\smu I_S\,, $ & $v_{23}=I_C\,,$\\ & & \\
    $v_{31}=-\frac{\beta}{\alpha}I_S+2\beta J_C\,,$ & $v_{32}=-\frac{\smu}{\alpha}(2+\alpha^2)I_S+\smu I_C$ & $v_{33}=I_S\,,$
    \end{tabular}
\end{table}

where 
$$I_C=I_C(\xi)\coloneqq \int\limits_{-\infty}^{\xi} \frac{\left( 1+\exp\left(-2\smu \;\xi\right) \right)^2}{\left( 1+\exp\left(-2\smu \;\tau\right) \right)^2}\cos\left(-s_0(\xi-\tau)\right)d\tau\,,$$
$$I_S=I_S(\xi)\coloneqq \int\limits_{-\infty}^{\xi} \frac{\left( 1+\exp\left(-2\smu \;\xi\right) \right)^2}{\left( 1+\exp\left(-2\smu \;\tau\right) \right)^2}\sin\left(-s_0(\xi-\tau)\right)d\tau\,,$$
$$J_C=J_C(\xi)\coloneqq \int\limits_{-\infty}^{\xi} \frac{\exp\left(-2\smu\; \tau\right)\left( 1+\exp\left(-2\smu \;\xi\right) \right)^2}{\left( 1+\exp\left(-2\smu \;\tau\right) \right)^3}\cos\left(-s_0(\xi-\tau)\right)d\tau\,,$$
$$J_S=J_S(\xi)\coloneqq \int\limits_{-\infty}^{\xi} \frac{\exp\left(-2\smu\; \tau\right)\left( 1+\exp\left(-2\smu \;\xi\right) \right)^2}{\left( 1+\exp\left(-2\smu \;\tau\right) \right)^3}\sin\left(-s_0(\xi-\tau)\right)d\tau\,.$$
Note that we do not provide explicit formulas for $v_{11}$, $v_{12}$ and $v_{13}$, because they are not needed for further computations. This is the reason why we neglected the explicit expressions of $\Theta_1\,,\Theta_2$, and $\Theta_3$ before. We now introduce the following complex-valued integrals for further computations: $$I(\xi)\coloneqq I_C(\xi)+\rmi  I_S(\xi)=\int\limits_{-\infty}^{\xi}\frac{\left( 1+\rme^{-2\smu \xi}\right)^2}{\left(1+\rme^{-2\smu \tau}\right)^2}\,\exp\left(-\rmi \frac{2\smu}{\alpha}(\xi-\tau)\right) d\tau\,,$$
for $I_C$ and $I_S$ as well as $$J(\xi)\coloneqq J_C(\xi)+\rmi J_S(\xi)=\int\limits_{-\infty}^{\xi}{\frac{\rme^{-2\smu\tau}\left( 1+\rme^{-2\smu \xi} \right)^2}{\left( 1+\rme^{-2\smu \tau}\right)^2}\exp\left(-\rmi \frac{2\smu}{\alpha}(\xi-\tau) \right)d\tau}\,$$
for $J_C$ and $J_S$.
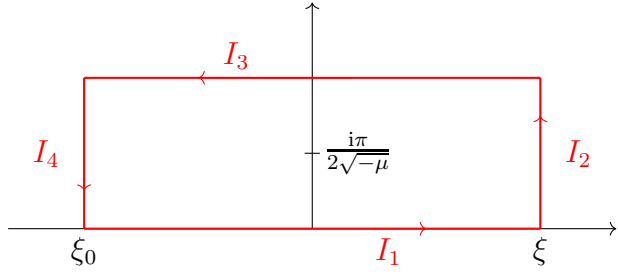
\begin{figure}[t]
    \centering
    \begin{tikzpicture}
        \draw[->, black, solid] (-4,0) -- (4,0);
        \draw[->, black, solid] (0,0) -- (0,3);
        \draw[red, solid, thick] (-3,0) -- (3,0);
        \draw[red, solid, thick] (-3,2) -- (3,2);
        \draw[red, solid, thick] (3,0) -- (3,2);
        \draw[red, solid, thick] (-3,0) -- (-3,2);
        \draw[->, red] (-2,0) -- (1.5,0);
        \draw[<-, red] (-1.5,2) -- (0.5,2);
        \draw[->, red] (3,0) -- (3,1.5);
        \draw[<-, red] (-3,0.5) -- (-3,2);
        \node at (3,-0.3) {$\xi$};
        \node at (-3,-0.3) {$\xi_0$};
        \node at (0.6,1) {$\frac{\rmi \pi}{2\smu}$};
        \node[red] at (1,-0.3) {$I_1$};
        \node[red] at (3.5,1) {$I_2$};
        \node[red] at (-1,2.3) {$I_3$};
        \node[red] at (-3.5,1) {$I_4$};
        \draw[solid] (-0.1,1) -- (0.1,1);
    \end{tikzpicture}
    \caption{\small Contour $\mathbf{C}$ in $\C$ for the integrals $I$ and $J$.}
    \label{fig:contour integral}
\end{figure}
We extend the above integrals to the complex plane and integrate along the counter-clockwise oriented rectangular contour $\mathbf{C}$ as illustrated in Figure\ \ref{fig:contour integral}, and let $\xi_0\to -\infty$. We will provide the details of the computation of $I$ only, as $J$ can be calculated in a fully analogous way.

\medskip
The complex integrand of $I$ is
$$g(z;\xi)\coloneqq \frac{\left( 1+\rme^{-2\smu \xi} \right)^2}{\left( 1+\rme^{-2\smu z} \right)^2} \exp\left(-\rmi \frac{2\smu}{\alpha}(\xi-z)\right)\,,$$
with singularities in $\C$ at the points $\frac{\rmi(\pi + 2 k \pi)}{2\smu}$, $k\in\Z$, one of which lies in the interior of $\mathbf{C}$, namely $z_0 := \frac{\rmi\pi}{2\smu}$. 
The contour integral $I$ can now be written via the residue theorem as $$I_1(\xi) + I_2(\xi) + I_3(\xi) + I_4(\xi) = 2\pi\rmi \sum\limits_{\textnormal{int} \mathbf{C}}  \textnormal{Res}\, g\left(z;\xi\right)\,,$$ where $I_1,\ldots ,I_4$ are given by
\begin{enumerate}
    \item[$I_1:$] $z=x$, \begin{align*}I_1(\xi_0, \xi) &= \left(1+\rme^{-2\smu \xi}\right)^2\exp\left(-\rmi \frac{2\smu}{\alpha} \xi\right)\int\limits_{\xi_0}^{\xi}{\frac{\exp\left(\rmi \frac{2\smu}{\alpha} x\right)}{\left( 1+\rme^{-2\smu x} \right)^2}dx}\\
    I(\xi) = \lim_{\xi_0 \to -\infty} I_1(\xi_0, \xi) &=\left(1+\rme^{-2\smu\xi}\right)^2\exp\left(-\rmi \frac{2\smu}{\alpha} \xi\right)\int\limits_{-\infty}^{\xi}{\frac{\exp\left(\rmi \frac{2\smu}{\alpha} x\right)}{\left( 1+\rme^{-2\smu x} \right)^2}dx}, \end{align*}
    
    \item[$I_2:$] $z=\xi+\rmi y$, $$I_2(\xi) = \left( 1+\rme^{-2\smu \xi} \right)^2\int\limits_{0}^{\frac{\pi}{\smu}}{\frac{\rmi \exp\left(-\frac{2\smu}{\alpha}y\right)}{\left( 1+\rme^{-2\smu( \xi +\rmi y)} \right)^2}dy},$$
    
    \item[$I_3:$] $z=x+\frac{\pi}{\smu}\rmi$, \begin{align*}I_3(\xi_0, \xi) &= \left( 1+\rme^{-2\smu \xi} \right)^2\exp\left(-\rmi \frac{2\smu}{\alpha} \xi\right) \rme^{-\frac{2 \pi}{\alpha}} \int\limits_{\xi}^{\xi_0}{\frac{\exp\left(\rmi \frac{2\smu}{\alpha} x\right)}{\left( 1+\rme^{-2\smu \xi} \right)^2}dx}\\
    &= -\rme^{-\frac{2 \pi}{\alpha}}I_1(\xi_0, \xi)\,,\end{align*}
    
    \item[$I_4:$] $z=\xi_0+\rmi y$, \begin{align*} I_4(\xi_0, \xi) &= \left( 1+\rme^{-2\smu \xi} \right)^2\exp\left(\rmi \frac{2\smu}{\alpha} (\xi_0-\xi)\right) \int\limits_{\frac{\pi}{\smu}}^{0}{\frac{\rmi \exp\left(-\frac{2\smu}{\alpha}y\right)}{\left( 1+\rme^{-2\smu(\xi_0+iy)} \right)^2}dy,}\\ \lim_{\xi_0\to -\infty}  I_4(\xi_0, \xi) &= 0\,.\end{align*}
\end{enumerate}

Utilizing the Laurent series of $g$ we obtain
$$
\textnormal{Res} \, g\left( z;\xi \right)\vert_{z = z_0} = \frac{\alpha+\rmi}{2\alpha\smu}\rme^{-\frac{\pi}{\alpha}} \left(1 + \rme^{-2\smu \xi} \right)^2\exp\left(-\rmi \frac{2\smu}{\alpha} \xi\right)\,,
$$ 
which leads to
$$I(\xi) = \left(1-\rme^{-\frac{2 \pi}{\alpha}} \right)^{-1} \left(\frac{\pi \rmi(\alpha+\rmi)}{\alpha\smu}\rme^{-\frac{\pi}{\alpha}}\left(1+\rme^{-2\smu \xi} \right)^2\exp\left(-\rmi \frac{2\smu}{\alpha} \xi\right) - I_2(\xi) \right).$$
Now we can write $$I_C(\xi)=\Re\, I(\xi)$$
$$=\frac{\pi\left( 1+\rme^{-2\smu \xi} \right)^2}{\alpha \smu \left( \rme^{\frac{\pi}{\alpha}}-\rme^{-\frac{\pi}{\alpha}} \right)}\left[-\cos\left( -\frac{2\smu}{\alpha}\xi \right) - \alpha \sin\left( -\frac{2\smu}{\alpha}\xi\right) \right] -\frac{1}{1-\rme^{-\frac{2\pi}{\alpha}}}I_2^r(\xi)$$ as well as 
$$I_S(\xi)=\Im\, I(\xi)$$
$$=\frac{\pi\left( 1+\rme^{-2\smu \xi} \right)^2}{\alpha \smu \left( \rme^{\frac{\pi}{\alpha}}-\rme^{-\frac{\pi}{\alpha}} \right)}\left[\alpha\cos\left( -\frac{2\smu}{\alpha}\xi \right)- \sin\left( -\frac{2\smu}{\alpha}\xi\right) \right] -\frac{1}{1-\rme^{-\frac{2\pi}{\alpha}}}I_2^i(\xi)\,,$$ where $I_2^r(\xi)$ and $I_2^i(\xi)$ are the real and imaginary part of $I_2(\xi)$, respectively.

\medskip
Studying the integral $J$ in a similar fashion, we obtain
$$J_C(\xi)=\Re\, J(\xi)$$ $$=\frac{\pi\left( 1+\rme^{-2\smu\xi} \right)^2}{2\alpha^2\smu \left(\rme^{\frac{\pi}{\alpha}}-\rme^{-\frac{\pi}{\alpha}} \right)} \left[ \alpha\cos\left( -\frac{2\smu}{\alpha}\xi \right) -\sin\left( -\frac{2\smu}{\alpha}\xi \right)\right]-\frac{1}{\left( 1-\rme^{-\frac{2\pi}{\alpha}} \right)}J_2^r(\xi)$$ as well as $$J_S(\xi)=\Im\, J(\xi)$$ $$=\frac{\pi\left( 1+\rme^{-2\smu\xi} \right)^2}{2\alpha^2\smu \left(\rme^{\frac{\pi}{\alpha}}-\rme^{-\frac{\pi}{\alpha}} \right)} \left[\cos\left( -\frac{2\smu}{\alpha}\xi \right) +\alpha\sin\left( -\frac{2\smu}{\alpha}\xi \right)\right]-\frac{1}{\left( 1-\rme^{-\frac{2\pi}{\alpha}} \right)}J_2^i(\xi)\,,$$
where also here $J_2^r(\xi)$ and $J_2^i(\xi)$ are the real and imaginary part of $J_2(\xi)$. Direct computations show also that $$\lim_{\xi\to +\infty}I_2^i(\xi)=\frac{\alpha}{2\smu}\left(1-\exp\left( -\frac{2\pi}{\alpha} \right) \right)$$ and $$\lim\limits_{\xi\to +\infty}I_2^r(\xi)=\lim\limits_{\xi\to +\infty}J_2^r(\xi)=\lim\limits_{\xi\to +\infty}J_2^i(\xi)=0\,.$$
Summing up, the second and third component of $u_\eta(\xi)\cdot \eta$ for sufficiently large $\xi$ are  $$\begin{pmatrix} \frac{\alpha}{2}s\\ \frac{\beta}{2\smu}\cc+\frac{2+\alpha^2}{2}s-\frac{\alpha}{2\smu}h\end{pmatrix}+$$ $$\frac{\pi}{\rho}\begin{pmatrix}
\left( -\frac{1}{\alpha \smu}h+\frac{2}{\alpha^2}s \right)\cos\left( -\frac{2\smu}{\alpha}\xi\right)+\left( -\frac{1}{\smu}h+\frac{(3+\alpha^2)}{\alpha}s \right) \sin\left(-\frac{2\smu}{\alpha}\xi\right)\\
\left( \frac{1}{\smu}h-\frac{(3+\alpha^2)}{\alpha}s \right)\cos\left(-\frac{2\smu}{\alpha}\xi\right)+\left( -\frac{1}{\alpha\smu}h+\frac{2}{\alpha^2}s \right)\sin\left(-\frac{2\smu}{\alpha}\xi\right)\end{pmatrix} +$$ $$+ \calO(\rme^{-2\smu \xi}),$$
where $\rho\coloneqq \exp(\pi/\alpha)-\exp(-\pi/\alpha)$. One readily verifies that the oscillatory part in the expression above vanishes if and only if $s$ and $h$ are zero and thus we infer that the heteroclinic connection cannot be between equilibria to first order in the parameters.

In order to detect cancellations of these oscillatory parts for higher orders of $s$ and $h$, we next consider the behavior of the quantity~\eqref{eq:hamiltonian fct} with respect to parameter perturbations. With slight abuse of notation, for $u=(\theta,p,q)\tran$ we write $H(u;\eta):=H(p,q)$ evaluated at parameters $\eta$, and other parameters at some fixed value, and we always consider the heteroclinic solutions from Corollary~\ref{cor:transition}.

\medskip
Our strategy in the following steps is as follows:\ we utilize the quantity $H$ because $\lim_{\xi\to+\infty} H$ always exists along these solutions. In order to distinguish whether this limit is an equilibrium or a periodic orbit, we consider 

$$\tH(u;\eta)\coloneqq H(u;\eta)-H(Z^\pi;\eta),$$ i.e., the difference of the $H$-values of the (parameter dependent) equilibrium $Z^\pi$ and the limit of $u$ as $\xi$ tends to infinity. Expanding $\tH$ in the limit $\xi\to\infty$ with respect to the parameter $\eta$ yields conditions for periodic asymptotics. In the following, subindices of $H$ denote partial derivatives, e.g. $H_u=\partial_u H$.

\medskip
Clearly, $H(u_0;\eta_0)=H(Z^\pi(\eta_0);\eta_0)$, thus $\tH_0=0$ and, since equilibria are critical points of $H$, we have $\tH_u(u_0;\eta_0)=\tH_\eta(u_0;\eta_0)=(0,0,0)\tran$. The second derivative is given by
\begin{equation}\label{eq:tH second derivative}\frac{d^2}{d\eta^2}\tH=\langle u_\eta,\tH_{uu}u_\eta \rangle+\langle u_\eta,\tH_{u\eta} \rangle+\langle \tH_{\eta u},u_\eta \rangle\,, \end{equation}
since $\tH_{\eta\eta}$ is the zero matrix, $\tH_u$ the zero vector, and $\tH_{u\eta}=\tH_{\eta u}\tran$. Thus
\begin{equation}\label{eq:tH expansion}
    \begin{split}
        \tH(u_0+u_\eta \eta;\eta)&=\frac{1}{2}\left(u_\eta \eta \right)\tran H_{uu}(u_0;\eta_0)(u_\eta \eta) + (u_\eta \eta)\tran H_{u\eta}(u_0;\eta_0)\eta\\ &\quad -\frac{1}{2}\left(Z^\pi_\eta(\eta_0)\eta\right)\tran H_{uu}\left( Z^\pi(\eta_0);\eta_0 \right) \left( Z^\pi_\eta(\eta_0)\eta \right)\\&\quad-\left( Z^\pi_\eta(\eta_0)\eta \right)\tran H_{u\eta} \left( Z^\pi(\eta_0);\eta_0 \right) \eta + \calO\left( \| \eta\|^3\right)\,.
\end{split}
\end{equation}
With the derivatives $H_{uu}, H_{u\eta}$ in~\eqref{eq:tH expansion} given by
\begin{align*}
H_{uu}(u;\eta)&=\begin{pmatrix}
0&0&0\\0&\frac{2}{q-s/2}&-\frac{2p-\alpha s}{(q-s/2)^2}\\ 0&-\frac{2p-\alpha s}{(q-s/2)^2}&2\frac{p^2-\alpha s p+h-\beta^-/\alpha+\mu-s^2/4}{(q-s/2)^3}
\end{pmatrix},\\
H_{u\eta}(u;\eta)&=\begin{pmatrix}
0 & 0 & 0\\0& \frac{p-\alpha q}{(q-s/2)^2}& 0 \\\frac{\beta/\alpha}{(1-\cc)^2(q-s/2)^2}& -\frac{p^2-\alpha p q-\frac{\alpha s}{2}p-\frac{s}{2}q+h-\beta^-/\alpha +\mu}{(q-s/2)^3} & -\frac{1}{(q-s/2)^2}
\end{pmatrix},
\end{align*}
for the right hand side of\ \eqref{eq:tH expansion} in the limit $\xi\to +\infty$ we obtain 
\begin{align*}
&\frac{1}{2}\left(u_\eta \eta \right)\tran H_{uu}(u_0;\eta_0)\left(u_\eta \eta \right)+\left(u_\eta \eta \right)\tran H_{u\eta}(u_0;\eta_0)\eta=-\frac{\alpha \beta^2}{4\mu \smu}\cc^2 \\
&~+ \frac{(4+5\alpha^2+\alpha^4)(\alpha^2\rho^2-4(1+\alpha^2)\pi^2)}{4\alpha^3\rho^2\smu}(s-s_0)^2-\frac{\alpha^4\rho^2-4(1+\alpha^2)\pi^2}{4\alpha\rho^2\mu\smu}(h-h_0)^2\\& -\frac{\alpha \beta(2+\alpha^2)}{2\mu}\cc (s-s_0) + \frac{\alpha^2\beta}{2\mu\smu}\cc (h-h_0)\\&~ + \frac{(2+\alpha^2)(\alpha^4\rho^2-4(1+\alpha^2)\pi^2)}{2\alpha^2\rho^2\mu}(s-s_0)(h-h_0)\,,
\end{align*}
as well as
\begin{align*}
&\frac{1}{2}\left(Z^\pi_\eta(\eta_0) \eta \right)\tran H_{uu}(Z^\pi(\eta_0);\eta_0)\left(Z^\pi_\eta(\eta_0) \eta \right)+\left(Z^\pi_\eta(\eta_0) \eta \right)\tran H_{u\eta}(Z^\pi(\eta_0);\eta_0)\eta=\\
 &~-\frac{\alpha \beta^2}{4\mu\smu}\cc^2+\frac{\alpha(4+5\alpha^2+\alpha^4)}{4\smu}(s-s_0)^2-\frac{\alpha^3}{4\mu\smu}(h-h_0)^2\\&\quad-\frac{\alpha\beta(2+\alpha^2)}{2\mu}\cc (s-s_0)+\frac{\alpha^2\beta}{2\mu\smu}\cc (h-h_0) + \frac{\alpha^2(2+\alpha^2)}{2\mu}(s-s_0)(h-h_0).
\end{align*}
Therefore, the expansion in the limit $\xi\to +\infty$ is independent of $\cc$ and reads
\begin{equation}\label{eq:second order approx}
\begin{split}
\lim\limits_{\xi\to\infty}\tH(u_0+u_\eta \eta;\eta)&=-\frac{(1+\alpha^2)^2(4+\alpha^2)\pi^2}{\alpha^3 \rho^2\smu}(s-s_0)^2\\
&\quad -\frac{2(1+\alpha^2)(2+\alpha^2)\pi^2}{\alpha^2\rho^2\mu}(s-s_0)(h-h_0)\\&\quad+\frac{(1+\alpha^2)\pi^2}{\alpha\rho^2\mu\smu}(h-h_0)^2+\calO\left(\|\eta-\eta_0\|^3\right)\,.
\end{split}
\end{equation}

Recall $\rho= \exp(\pi/\alpha)-\exp(-\pi/\alpha)$. One readily verifies that the resulting (binary) quadratic form of~\eqref{eq:second order approx} is negative definite for all $\alpha>0$ so the only solution to the leading order problem $$\frac{d^2}{d\eta^2}\tH(u_0;\eta_0)=0$$ is the trivial one $(s,h)=(0,0)$,
and any non-trivial solution satisfies $|s-s_0|^2+|h-h_0|^2=\calO(|\cc|^3)$.

\medskip
In particular, for $\cc=0$ there is a neighborhood of $(s_0,h_0)$ such that the only solution is the trivial one, which is therefore also the case in the~\ref{LLG} equation. In case $\cc\neq 0$, higher orders may lead to a solution with non-zero $s$ and/or $h$, but there is numerical evidence that such solutions do not exist (see \S~\ref{sec:continuation} for details).

\subsection{Proof of Theorem~\ref{theo:codim2}}\label{app:proof}
The idea of the proof is to apply Lyapunov-Schmidt reduction, i.e., to determine a bifurcation equation whose solutions are in one-to-one correspondence with heteroclinic connections between equilibria~\eqref{eq:coherent2} near one of the explicit solutions $u_0$ from \eqref{eq:atan} connecting the equilibria $Z^0_-$ and $Z^\pi_-$. In the present context this is known as Melnikov's method, see for example~\cite{kuznetsov2013elements}.

Recall $u_0$ corresponds to a homogeneous DW for $\cc=0$ with speed $s_0$ and rotation frequency $\Omega_0$ given by~\eqref{eq:initialspeedfrequenzy}. In the present codim-2 parameter regime we will show that the bifurcation equation defines a codimension two bifurcation curve in the three-dimensional parameter space $(\cc, s, \Omega)$, which passes through the point $(0, s_0, \Omega_0)$. The main part of the proof is to show the existence of certain integrals for the considered parameter set. These integrals are almost identical to the ones studied within the proof of Theorem~\ref{theo:center} and we use the same approach. 

\medskip
In this section we denote the parameter vector by $\eta \coloneqq (\cc, s, \Omega)\tran \in \R^3$, with initial value $\eta_0 =  (0,s_0,\Omega_0)\tran$ corresponding to the unperturbed values. The solutions of the perturbed system close to $u_0$ has the form $u(\xi; \eta) = u_0(\xi) + u_\eta(\xi; \eta_0)(\eta-\eta_0) + \calO(\|\eta-\eta_0\|^2)$, where $u_\eta =  (\theta_\eta, p_\eta, q_\eta)\tran = \calO(\|\eta-\eta_0\|)$. 

\medskip
As discussed in Appendix~\ref{app:center case proof}, the linearization~\eqref{eq:linearizazion} of system~\eqref{eq:coherent2} around the unperturbed heteroclinic connection $u_0$ has the fundamental solution matrix $\Phi(\xi)$ as defined in~\eqref{eq:fundMatrix}. In the present codim-2 case with $\dim(W^0_u)=\dim(W^\pi_s)=1$ in $\R^3$, the bifurcation equation $M(\eta)=0$ entails two equations. Here $M(\eta)$ measures the displacement of the manifolds $W^0_u$ and $W^\pi_s$, and we will choose this to be near the point $u_0(0)=\left(\frac{\pi}{2},\smu,0\right)\tran$ in the directions given by vectors $v_1(0)$ and $v_2(0)$ from adjoint solutions as detailed below. From the Taylor expansion $M(\eta) = M_\eta(\eta_0)(\eta-\eta_0) + \calO(\|\eta-\eta_0\|^2)$ we infer by the implicit function theorem that a full rank of $M_\eta(\eta_0)$ implies a one-to-one correspondence of solutions to the bifurcation equation with elements in the kernel of $M(\eta_0)$. 

\medskip
In order to compute $M_\eta(\eta_0)$ and its rank, we project onto the transverse directions to $u_0$, which means to project the inhomogeneous part of equation~\eqref{eq:linearizazion} onto two linearly independent bounded solutions $v_1, v_2$ of the adjoint variational equation $v'=-A\tran\cdot v$, where
$$A\tran = \begin{pmatrix} \smu \cos(\theta_0) & 0 & 0\\\sin(\theta_0) & -\alpha s_0 - 2\smu \cos(\theta_0) & -s_0\\0&s_0& -\alpha s_0 - 2\smu \cos(\theta_0)\end{pmatrix}.$$
The solutions are given in terms of~\eqref{eq:Usol} by 
$$v_1=\frac{U_1 \times U_2}{\det \Phi}=\begin{pmatrix} 0\\-\frac{1}{r(\xi)}\sin(-s_0\xi)\\\frac{1}{r(\xi)}\cos(-s_0\xi) \end{pmatrix} \quad \textnormal{and} \quad v_2=\frac{U_3 \times U_1}{\det \Phi}=\begin{pmatrix} 0\\\frac{1}{r(\xi)}\cos(-s_0\xi)\\\frac{1}{r(\xi)}\sin(-s_0\xi) \end{pmatrix}\,.$$
Implementing the projection onto these, we obtain the so-called
\emph{Melnikov integral} \begin{equation}\label{eq:melnikov integral}
\begin{split}M_\eta(\eta_0)&\coloneqq\int\limits_{-\infty}^{+\infty}{\left(v_1,v_2\right)\tran \cdot f_\eta(u_0;\eta_0)}\,d\xi\\ &=\begin{pmatrix} \beta I_{CC} & \alpha\smu I_S-\smu I_C & I_S+\alpha I_C \\ \beta I_{CS} & -\alpha\smu I_C - \smu I_S & -I_C+\alpha I_S \end{pmatrix},
\end{split}
\end{equation}
where
\begin{align*}
I_{CC}&:=\int\limits_{-\infty}^{+\infty}{\frac{\left(1-\rme^{2\smu \xi} \right)\rme^{\alpha s_0 \xi + 2 \smu \xi} \cos(-s_0 \xi)}{\left( 1+\rme^{2\smu \xi} \right)^3}}d\xi\,,\\
I_{CS}&:=\int\limits_{-\infty}^{+\infty}{\frac{\left(1-\rme^{2\smu \xi} \right)\rme^{\alpha s_0 \xi + 2 \smu \xi} \sin(-s_0 \xi)}{\left( 1+\rme^{2\smu \xi} \right)^3}}d\xi\,,\\
I_C&:=\int\limits_{-\infty}^{+\infty}{\frac{\rme^{\alpha s_0 \xi} \cos(-s_0 \xi)}{\left(1+\rme^{2\smu \xi} \right) \cdot \left(1+\rme^{-2\smu \xi} \right) }}d\xi\,,\\
I_S&:=\int\limits_{-\infty}^{+\infty}{\frac{\rme^{\alpha s_0 \xi} \sin(-s_0 \xi)}{\left(1+\rme^{2\smu \xi} \right) \cdot \left(1+\rme^{-2\smu \xi} \right) }}d\xi\,.
\end{align*}

We next show that the second and third columns in~\eqref{eq:melnikov integral} have non-vanishing determinant so that the rank is always 2, in particular also for $\beta=0$.

For brevity, we present the calculations of $I_{CC}$ and $I_{CS}$ only, which are based on the same idea as the computations in Appendix~\ref{app:center case proof}. The solutions for $I_C$ and $I_S$ can be computed in an analogous way.

\medskip
From Appendix~\ref{app:center case proof} we know that the following integral would not exist in case $s_0=2\smu/\alpha$ and one readily verifies the existence for $s_0=0$. Therefore, we first assume $0<s_0<2\smu/\alpha$ for the moment and discuss the case $s_0=0$ later. We set $$I\coloneqq\int\limits_{-\infty}^{+\infty}{g(\xi)}d\xi\,, \quad g(\xi)\coloneqq\frac{\left(1-\rme^{2\smu \xi} \right)\rme^{(\alpha s_0 + 2 \smu)\xi} \rme^{-\rmi s_0 \xi}}{\left( 1+\rme^{2\smu \xi} \right)^3}\,,$$
so that $I_{CC}=\Re(I)$ and $I_{CS}=\Im(I)$. Utilizing the same idea for the contour integral as in Appendix~\ref{app:center case proof} (cf.\ Figure~\ref{fig:contour integral}) and the residue theorem, we obtain $$I-\rme^{\alpha s_0 \frac{\rmi \pi}{\smu}}\cdot \rme^{s_0 \frac{\pi}{\smu}}\cdot I=2 \pi \rmi \sum{\mathrm{Res} (g)}\,.$$

The function $g$ has a pole of order three at $\frac{\rmi \pi}{2\smu}$ and the \emph{Laurent} series gives
$$\mathrm{Res}(g)=-\frac{(\alpha-\rmi)^2s^2_0}{8\mu\smu}\exp\left( (\alpha-\rmi)\frac{\rmi \pi s_0}{2\smu} \right),$$
and therefore
\begin{align*} I&=\frac{2\pi s_0^2 \rme^{\frac{\pi s_0}{2\smu}}\left[ -2\alpha \cos\left(\frac{\pi \alpha s_0}{2\smu}\right)+\left(\alpha^2 -1\right)\sin\left(\frac{\pi \alpha s_0}{2\smu}\right)\right]}{8\mu\smu\left(1 -  \rme^{\frac{\pi s_0}{\smu}}\cos\left(\frac{\pi \alpha s_0}{\smu}\right)-\rmi \rme^{\frac{\pi s_0}{\smu}}\sin\left(\frac{\pi \alpha s_0}{\smu}\right)\right)}\\&\quad+\frac{\rmi 2\pi s_0^2 \rme^{\frac{\pi s_0}{2\smu}}\left[\left(1-\alpha^2\right)\cos\left(\frac{\pi \alpha s_0}{2\smu}\right) - 2\alpha \sin\left(\frac{\pi \alpha s_0}{2\smu}\right)\right]}{8\mu\smu\left(1 -  \rme^{\frac{\pi s_0}{\smu}}\cos\left(\frac{\pi \alpha s_0}{\smu}\right)-\rmi \rme^{\frac{\pi s_0}{\smu}}\sin\left(\frac{\pi \alpha s_0}{\smu}\right)\right)}\,. \end{align*}
Separating the real and imaginary parts of $I$ we get
$$I_{CC} = \frac{\pi s_0^2\smu\,\rme^{\frac{\pi s_0}{2\smu}}}{4\mu^2}\cdot\frac{2\alpha \left( 1-\rme^{\frac{\pi s_0}{\smu}} \right)\cos\left(\frac{\pi \alpha s_0}{2\smu} \right)  + \left(1-\alpha^2\right)\left( 1+\rme^{\frac{\pi s_0}{\smu}}\right)\sin\left( \frac{\pi \alpha s_0}{2\smu} \right)}{1+\rme^{\frac{2\pi s_0}{\smu}} - 2 \rme^{\frac{\pi s_0}{\smu}}\cos\left(\frac{\pi \alpha s_0}{\smu} \right)}$$
$$I_{CS}= \frac{\pi s_0^2\smu\,\rme^{\frac{\pi s_0}{2\smu}}}{4\mu^2 }\cdot\frac{ \left(1-\alpha^2\right)\left( 1-\rme^{\frac{\pi s_0}{\smu}}\right)\cos\left( \frac{\pi \alpha s_0}{2\smu} \right)+ 2\alpha \left( 1+\rme^{\frac{\pi s_0}{\smu}} \right)\sin\left(\frac{\pi \alpha s_0}{2\smu} \right)}{1+\rme^{\frac{2\pi s_0}{\smu}} - 2 \rme^{\frac{\pi s_0}{\smu}}\cos\left(\frac{\pi \alpha s_0}{\smu} \right)}\,.$$

As mentioned before, one analogously gets $$I_C=\frac{\pi s_0\, \rme^{\frac{\pi s_0}{2\smu}}}{2\mu}\cdot \frac{\left( 1-\rme^{\frac{\pi s_0}{\smu}} \right)\cos\left( \frac{\pi \alpha s_0}{2\smu} \right) -\alpha \left(1+\rme^{\frac{\pi s_0}{\smu}} \right) \sin\left(\frac{\pi \alpha s_0}{2\smu}\right)}{1+\rme^{\frac{2 \pi s_0}{\smu}}-2\rme^{\frac{\pi s_0}{\smu}}\cos\left(\frac{\pi \alpha s_0}{\smu}\right)}$$
$$I_S=\frac{\pi s_0\, \rme^{\frac{\pi s_0}{2\smu}}}{2\mu}\cdot\frac{\alpha\left( 1-\rme^{\frac{\pi s_0}{\smu}} \right)\cos\left( \frac{\pi \alpha s_0}{2\smu} \right) + \left(1+\rme^{\frac{\pi s_0}{\smu}} \right) \sin\left(\frac{\pi \alpha s_0}{2\smu}\right)}{1+\rme^{\frac{2 \pi s_0}{\smu}}-2\rme^{\frac{\pi s_0}{\smu}}\cos\left(\frac{\pi \alpha s_0}{\smu}\right)}\,.$$

Having the explicit expressions for $I_C$, $I_S$, $I_{CC}$, as well as $I_{CS}$, we can study the rank of~\eqref{eq:melnikov integral} in case $0<s_0<2\smu/\alpha$. The determinant of the second and third column of~\eqref{eq:melnikov integral} simplifies to
$$\left( \alpha I_S-I_C\right)^2 + \left( I_S+\alpha I_C\right)^2= \left(1+\alpha^2\right)^2\pi^2s_0^2\rme^{\frac{\pi s_0}{\smu}}\neq0 \,,\quad \forall  s_0\neq0\,.$$
Therefore, $M_\eta(\eta_0)$ has full rank and we obtain
\begin{equation}\label{eq:splitting function}M(\eta)=\begin{pmatrix} \beta I_{CC} & \alpha\smu I_S-\smu I_C & I_S+\alpha I_C \\ \beta I_{CS} & -\alpha\smu I_C - \smu I_S & -I_C+\alpha I_S \end{pmatrix}\cdot(\eta-\eta_0)+\calO\left(\|\eta-\eta_0\|^2\right).
\end{equation}

\medskip
In the remaining case $s_0=0$, which means $h=\beta/\alpha$, we have $I_{CC}=I_{CS}=I_S=0$, $I_C=\frac{1}{2\smu}$, and thus
\begin{equation}\label{eq:melnikovzerospeed}
M_\eta(\eta_0)=\begin{pmatrix} 0&-\frac{1}{2}&\frac{\alpha}{2\smu}\\0&-\frac{\alpha}{2}&-\frac{1}{2\smu}
\end{pmatrix},
\end{equation}
whose rank is always 2. Thus the splitting directions are independent of $\cc$ in first order in case $h=\beta/\alpha$. Moreover, note that the splitting in $s$ is independent of the anisotropy $\mu$ to first order in case $h=\beta/\alpha$. This completes the proof of Theorem~\ref{theo:center}.
\end{document}